\documentclass[12pt,a4paper]{article}
\usepackage{amsmath,amsfonts,amssymb,amscd}
\usepackage[margin=2cm]{geometry}
\usepackage{url}
\usepackage{hyperref}
\usepackage{tikz-cd} %\usetikzlibrary{cd}
\usetikzlibrary {graphs}
\usepackage{theorem}
\usepackage{dsfont}
\usepackage{stmaryrd}
\usepackage{verbatim}
\usepackage{enumerate}
\usepackage{appendix}
\usepackage{mathabx}
\usepackage{float}

\usepackage{enumitem}

\usepackage{mathabx,epsfig}
\def\acts{\mathrel{\reflectbox{$\righttoleftarrow$}}}

\newcommand{\ol}{\overline}

\newcommand{\smt}[4]{\left( \begin{smallmatrix} #1 &#2\\ #3 &#4\\ \end{smallmatrix} \right) }
\newcommand{\GLmod}[1]{\mathrm{GL}_2(\mathbb Z/#1\mathbb Z)}
\newcommand{\SLZ}{\mathrm{SL}_2(\mathbb Z)}
\newcommand{\modstar}[1]{(\mathbb Z/#1\mathbb Z)^\times}

\usepackage{graphicx}

\renewcommand{\P}{\mathbb P}
\newcommand{\C}{\mathbb C}
\newcommand{\R}{\mathbb R}
\newcommand{\Q}{\mathbb Q}
\newcommand{\Z}{\mathbb Z}
\newcommand{\F}{\mathbb F}
\newcommand{\G}{\mathbb G}
\newcommand{\HH}{\mathbb H}

\newcommand{\Gm}{\G_{\mathrm{m}}}

\newcommand{\divisor}{\mathrm{div}}
\newcommand{\quot}{\mathrm{quot}}

\newcommand{\Norm}{\mathrm{Norm}}
\newcommand{\Hom}{\mathrm{Hom}}
\newcommand{\End}{\mathrm{End}}
\newcommand{\GL}{\mathrm{GL}}

\newcommand{\pr}{\mathrm{pr}}
\newcommand{\rank}{\mathrm{rank}}
\newcommand{\Pic}{\mathrm{Pic}}
\newcommand{\Gal}{\mathrm{Gal}}
\newcommand{\Spec}{\operatorname{Spec}}
\newcommand{\Frob}{\mathrm{Frob}}
\newcommand{\Ver}{\mathrm{Ver}}
\newcommand{\Id}{\mathrm{Id}}
\newcommand{\ord}{\mathrm{ord}}
\newcommand{\new}{\mathrm{new}}
\newcommand{\pnew}{{p\text{-new}}}
\newcommand{\Aut}{\mathrm{Aut}}

\newcommand{\lto}{\longrightarrow}

\newcommand{\calL}{\mathcal{L}}
\newcommand{\calM}{\mathcal M}
\newcommand{\calA}{\mathcal A}
\newcommand{\calO}{\mathcal{O}}

\newcommand{\tT}{\tilde T}
\newcommand{\dia}[1]{\widetilde{\langle #1 \rangle}}

\numberwithin{equation}{section}

\newtheorem{Theorem}[equation]{Theorem}
\newtheorem{Proposition}[equation]{Proposition}
\newtheorem{Lemma}[equation]{Lemma}
\newtheorem{Corollary}[equation]{Corollary}
\newtheorem{subTheorem}[subsubsection]{Theorem}
\newtheorem{subProposition}[subsubsection]{Proposition}
\newtheorem{subLemma}[subsubsection]{Lemma}
\newtheorem{subCorollary}[subsubsection]{Corollary}

\newtheorem{Definition}[equation]{Definition}
\newtheorem{question}[equation]{Question}

\theorembodyfont{\rmfamily}
\newtheorem{Remark}[equation]{Remark}
\newtheorem{subDefinition}[subsubsection]{Definition}
\newtheorem{subRemark}[subsubsection]{Remark}



\makeatletter 
\newenvironment{subeqn}{\refstepcounter{subsubsection}
	$$}{\leqno{\rm(\thesubsubsection)}$$\global\@ignoretrue}
\makeatother 

\newenvironment{prf}[1]{\trivlist
\item[\hskip \labelsep{\it
#1\hspace*{.3em}}]}{~\hspace{\fill}~$\square$\endtrivlist}

\newenvironment{proof}{\begin{prf}{\bf Proof}}{\end{prf}}

\tikzset{
labl/.style={anchor=south, rotate=90, inner sep=.5mm}
}

\title{Spectral theory of isogeny graphs}

\author{Giulio Codogni \& Guido Lido \\
	\small{
		\href{mailto:codogni@mat.uniroma2.it}{codogni@mat.uniroma2.it}
		\quad
		\href{mailto:guidomaria.lido@gmail.com}{guidomaria.lido@gmail.com}} \\
		\small{Universit\`{a} di Roma Tor Vergata}
}

\begin{document}

\maketitle

\begin{abstract}
We consider finite graphs whose vertices are supersingular elliptic curves, possibly with level structure, and edges are isogenies. They can be applied to the study of modular forms and to isogeny based cryptography. The main result of this paper is an upper bound on the absolute values of  the eigenvalues of their adjacency matrices, which in particular implies that these graphs are Ramanujan. We also study the asymptotic distribution of the eigenvalues of the adjacency matrices, the number of connected components, the automorphisms of the graphs, and the connection between the graphs and modular forms. 
\end{abstract}

\tableofcontents

\section{Introduction}\label{S:graph_definition}
Given two distinct prime numbers $p$ and $\ell$, supersingular isogeny graphs are finite graphs whose vertices are isomorphism classes of supersingular elliptic curves defined over a field of characteristic $p$, possibly enriched with some level structure, and edges are degree $\ell$ isogeny (see Definitions \ref{def:level_structures} and \ref{def:graph}). The number of vertices of these graphs grows linearly in $p$.

Theorems \ref{thm:main1} and \ref{thm:main2}, our main results, give information about the spectrum of the adjacency matrices of these graphs. They rely on algebraic geometry constructions.

The spectrum of the adjacency matrix is not a complete invariant of a graph. Non-isomorphic graphs whose adjacency matrices have the same spectrum are sometimes called cospectral mates. Spectral graph theory allows to gather important information about the geometry of the graph only out of the spectrum of the adjacency matrix, and this is why our work provides a better understanding of isogeny graphs.

\medskip

Isogeny graphs were first studied by Mestre \cite{Mestre} in the 80's. His goal was to study modular forms, in particular to compute eigenforms out of eigenvectors of adjacency matrices of isogeny graphs. This approach has been recently made very practical in \cite{Cow}. Our Theorems \ref{thm:modular_general}, \ref{thm:modular_special} generalize \cite[Theorem 2.1]{Mestre}, and we hope they lead to possible extensions of Mestre's ``M\'ethode des graphes'', even though an analogue for formula (1) in loc. cit. is needed. 

In the 90’s people from graph theory were looking for explicit examples of graphs with optimal spectral gap, and consequently optimal expansion constant and mixing time. Surprisingly, classical isogeny graphs, i.e. without level structure, provided such examples! These facts are discussed in Section \ref{S:graph}, where we also show, as corollary of our main results, that isogeny graphs with level structure also have this property.

More recently, isogeny graphs started to play %G have played 
an important role in cryptography, as many protocols from isogeny based cryptography rely on their features. For instance, information about the spectrum of isogeny graphs with Borel level structure is used to prove Statistical Zero Knowledge for a proof of knowledge in \cite{noi} and for signature schemes in \cite{SQIsign2D-West,SQIsignHD:}. 
This is discussed in Section \ref{S:crypto}.

\subsection{Main Definitions and Results}

The vertices of the graphs we study are supersingular elliptic curves enriched with some torsion data. We start by specifying what kind of data we are interested in.

\begin{Definition}[Level structure on elliptic curves]\label{def:level_structures}
Fix a positive integer $N$ and a subgroup $H$ of $\GLmod{N}=\Aut((\Z/N\Z)^2)$. Let $k$ be a field where $N$ is invertible. For an elliptic curve $E/k$, a \emph{level $H$ structure} on $E$ is an isomorphism $\phi\colon (\Z/N\Z)^2\to E[N] $ considered up to composition with an element of $H$, i.e. two isomorphisms $\phi$ and $\phi'$ are equivalent if there exists an element $h$ in $H$ such that $\phi=\phi' \circ h$.
\end{Definition}

Sometimes level $H$ structures have more explicit interpretations, as illustrated below.

\begin{itemize}%[leftmargin=*]
\item\emph{Trivial level structure}: when $H=\GLmod{N}$, there is a unique level structure on every elliptic curve;
\item\emph{Borel level structure:} when $H = \{\smt *0**\}$ is the subgroup of lower triangular matrices, a level $H$ structure is equivalent to the choice of a cyclic subgroup of order $N$ in $E[N]$ (equivalently, Borel level structure can be defined using upper triangular matrices; we prefer the lower ones as they make some computations with modular forms less cumbersome);
\item\emph{Full level structure}: when $H=\{\mathrm{Id}\}$, an $H$ structure is equivalent to the choice of a basis of $E[N]$;
\item\emph{Split Cartan level structure}: when $H=\{(\begin{smallmatrix} *&0\\0&*	\end{smallmatrix})\}$, a level structure is equivalent to the choice of an ordered pair of cyclic subgroups $C_1, C_2 < E[N]$ having order $N$ and trivial intersection. This level structure gives a graph isomorphic to a graph with Borel level structure, see Section \ref{S:BorelversusCartan}, so we will not discuss it in details. If we instead consider $H=\{(\begin{smallmatrix} *&0\\0&*	\end{smallmatrix})\} \cup \{(\begin{smallmatrix} 0&*\\ *&0	\end{smallmatrix})\}$, which is equal to the normalizer of $\{(\begin{smallmatrix} *&0\\0&*	\end{smallmatrix})\}$ when $N$ is an odd prime power, an $H$-level structure corresponds to a non-ordered pair of cyclic subgroups; the corresponding graph is a quotient of the graph with Cartan level structure.

\item\emph{Torsion point level structure}: when $H=\{(\begin{smallmatrix} *&0\\ *&1	\end{smallmatrix})\}$, a level $H$ structure is equivalent to the choice of a point of order $N$;
\item\emph{Non split Cartan level structure}:  when $H$ is a non-split Cartan subgroups of $\GLmod{N}$, which is unique up to conjugation. Details are given \cite{PV} and  in \cite{RW} these structures are interpreted as  ``necklaces'' of subgroups of $E[N]$ for $N$ prime.
% are a generalization of subgroups $ \F_{q^2}^\times \cong \{ \smt{a}{b\xi}{b}{a} : (a,b)\neq 0 \} \subset \GLmod{q}$ an odd prime $q$ and $\xi\in \F_q^\times$ a non-square. %Non split Cartan subgroups are unique up to conjugation, d

\end{itemize}

Fix $(E_1, \phi_1)$ and $(E_2,\phi_2)$, where $E_1, E_2$ are elliptic curves over a common field $k$, and $\phi_i$ is a level $H$ structure on $E_i$. A morphism 
$\alpha \colon (E_1, \phi_1)\to (E_2,\phi_2)$ is an isogeny $\alpha\colon E_1\to E_2$ such that $\alpha\circ\phi_1 = \phi_2$ as level $H$ structures on $E_2$, or equivalently such that there exists an element $h\in H$ satisfying $\alpha\circ\phi_1 = \phi_2\circ h$. 
The degree of such a morphism is the degree of the corresponding isogeny $E_1\to E_2$. 
We call such an $\alpha$ a \emph{morphism, or isogeny, of elliptic curves with level structures}.
A morphism is an \emph{isomorphism} if it is an isomorphism at the level of elliptic curves, i.e. it has degree one.

\begin{Remark}
In the context pairs $(E,\phi)$ of elliptic curves with level $H$ structure, if $u$ is an automorphism of $E$, then the pairs $(E,\phi)$ and $(E,u\circ \phi)$ are always isomorphic. Nevertheless $u$ does not always define an automorphism of $(E,\phi)$: it does if and only if $$\phi^{-1}\circ u\circ \phi \colon(\Z/N\Z)^2 \to (\Z/N\Z)^2 $$ lies in $H$. 
In particular, if $\smt{-1}{}{}{-1} \notin H $, then $-1$ is not an automorphism of $(E,\phi)$ even though $(E,\phi)\cong (E,-\phi)$.
\end{Remark}

\begin{Definition}[Supersingular isogeny graph]\label{def:graph}
Fix a positive integer $N$, a subgroup $H$ of $\GLmod{N}$ and distinct prime numbers $p, \ell$ not dividing $N$.

The isogeny graph with level structure $G = G(p,\ell,H)$ is the directed graph with:
\begin{itemize}
	\item vertices $V=\{(E_1,\phi_1),\dots ,( E_r, \phi_r)\}$ 
	a set of representatives of isomorphism classes of supersingular elliptic curves $E/\overline{\F}_p$ with a level $H$ structure $\phi$.
	\item edges: given vertices $(E_i,\phi_i)$ and $(E_j,\phi_j)$, edges between them are degree $\ell$ morphisms $(E_i,\phi_i)\to (E_j,\phi_j)$, modulo postcomposition by automorphisms of $(E_j,\phi_j)$.
\end{itemize}

We denote by $A={(a_{ij})_{i,j}}$ 
the adjacency matrix of $G$, namely the matrix whose entries $a_{ij}$ are the number of edges $(E_j,\phi_j)\to (E_i,\phi_i)$.
\end{Definition}

% \begin{Remark}
% Suppose $E/\overline{\F}_p$ is a supersingular elliptic curve with an automorphism $u$, and that $\phi\colon (\Z/N\Z)^2 \to E[N]$ is a level $H$ structure on $E$. Then, the pairs $(E,\phi)$ and $(E,u\circ \phi)$ are always isomorphic, hence there is one vertex $(E_i,\phi_i)$ of $G(p,\ell,H)$ representing both. Nevertheless $u$ does not always define an automorphism of $(E_i,\phi_i)$: it does if and only if $$\phi^{-1}\circ u\circ \phi \colon(\Z/N\Z)^2 \to (\Z/N\Z)^2 $$ lies in $H$. In particular, if $\smt{-1}{}{}{-1} \notin H $, then $-1$ is not an automorphism of $(E,\phi)$ even though $(E,\phi)\cong (E,-\phi)$.
% \end{Remark}

In the context of the above definition, given a vertex $(E_i,\phi_i)$, taking the kernel of isogenies gives a bijection between cyclic subgroups of cardinality $\ell$ of $E_i[\ell]$, and edges coming out of the vertex $(E_i,\phi_i)$, since edges correspond to isogenies up to postcomposition by automorphisms. In particular, there are exactly $\ell{+}1$ edges coming out of each vertex. The graph does not depend on the choices of the representatives $(E_i,\phi_i)$.

The graph $G$ might not be connected. For every connected component $G_i$, consider the vector $v_i$ in $\mathbb{C}^V$ obtained as formal sum of the vertex of $G_i$. Then $A^t\, v_i=(\ell+1)v_i$, where $^t$ denotes the transpose. This shows that $\ell{+}1$ is an eigenvalue of $A$.

Our first main result gives information about
the eigenvalues of $A$ in the cases where $H$ contains the scalar matrices and its image $\det(H) \subset (\Z/N/\Z)^\times$ under the determinant is maximal. 
For example, it covers the case of graphs with Borel, Cartan (both split and non-split) and trivial level structure (which give the classical isogeny graph).
In particular, we recover Pizer's result \cite[Theorem 1]{Pizer} (see also the detailed discussion in Section \ref{sec:otherworks}). 
% We notice that the graph with trivial level structure coincides with  the classical isogeny graphs, 

\begin{Theorem}\label{thm:main1}
With the notation of Definition \ref{def:graph}, if $H$ contains the scalar matrices and $\det(H)=(\Z/N\Z)^{\times}$, then the graph $G(p,\ell,H)$ is connected, its adjacency matrix $A$ is diagonalizable over $\R$, the eigenvalue $\ell+1$ has multiplicity one, and all the other eigenvalues have absolute value smaller than 
$$2\sqrt{\ell}-\left(4\sqrt{\ell}\right)^{-2|V|+3}\,,$$
where $|V|$ is the number of vertices of $G(p,\ell,H)$. In particular, all eigenvalues different from $\ell+1$ are contained in the open Hasse interval $(-2\sqrt{\ell},2\sqrt{\ell})$.
\end{Theorem}

When the graph contains pairs $(E,\phi)$ with non-trivial automorphisms (i.e. automorphisms not induced by $\pm 1\in \Aut(E)$), the adjacency matrix $A$ is not symmetric.
% hence the fact the spectrum is real requires some non-trivial argument.

\medskip

When $\det(H)$ is strictly contained in $(\Z/N\Z)^{\times}$, we need to introduce some further notations to describe the connected components of the graphs, and their partitions. Let $\mu_N^\times(\overline{\F}_p)$ be the set of primitive $N$-th root of unity in $\overline{\F}_p$. This is a principal homogeneous space for the right action of $(\Z/N\Z)^{\times}$ given by $\zeta\cdot a = \zeta^a$. The group $\det(H)$ is a subgroup of $(\Z/N\Z)^{\times}$, so it also acts on $\mu_N^\times(\overline{\F}_p)$ and we can form the quotient $R_H:=\mu_N^\times(\overline{\F}_p)/\det(H)$.

\begin{Definition}[Weil invariant of a level structure]\label{def:Weil_invariant}
Consider an elliptic curve with level $H$ structure $(E,\phi)$. Let $w$ be the Weil pairing on the $N$-torsion of $E$ and let 
$$
w(\phi) = w(\phi(\begin{smallmatrix}1\\0\end{smallmatrix}), \phi(\begin{smallmatrix}0\\1\end{smallmatrix}))  \,.
$$
As $\phi$ is defined only modulo the action of $H$, the invariant $w(\phi)$ is an element of the quotient $R_H$. We call this invariant the Weil invariant of the level structure. 
\end{Definition}

The Weil invariant gives an obstruction to $G = G(p,\ell,H)$ being connected: if two vertices $v_1,v_2$  are connected by a degree $\ell$ isogeny, then \cite[Chapter III, Proposition 8.2]{Sil} implies that their corresponding Weil invariants are connected by the action of $\ell$, i.e. $w(v_2) = w(v_1)^\ell$. Hence in a connected component of $G$, the Weil invariant has image an orbit of the action of $\ell$ on $R_H$. Let $\{C_1,\dots , C_n\}$ be the orbits of $\ell$ acting on $R_H$ and for each $i$ we denote  
$$G_i := w^{-1}(C_i)$$
which is a subgraph of $G$ by the previous argument. Our second main result generalizes Theorem~\ref{thm:main1}.

\begin{Theorem}\label{thm:main2} 
With the notation of Definition \ref{def:graph}, let  $G=G(p,\ell,H)$ with $H$ any subgroup of $\GLmod{N}$, and let  $G_1, \ldots, G_n$ be the subgraphs of $G$ defined above.

\begin{description}[leftmargin=\parindent,labelindent=\parindent]

\item[Connected components] Each $G_i$ is connected, i.e. the graph $G$ has $n$ connected components. Let $\mathcal N_H$ be the normalizer of $H$ in $\GLmod{N}$. If $p$, $\ell$ and $\det(\mathcal N_H)$ together generate $(\Z/N\Z)^\times$, then the $G_i$'s  are all isomorphic.

\item[Spectrum of the adjacency matrix] Denote by $k$ the order of $\ell$ in $(\Z/N\Z)^\times/\det(H)$, and by $k'$ the smallest positive integer such that 
$\ell^{k'}\Id \in H$. %? OLD non funziona perche' \GLmod{N}/H potrevbe non essere un gruppo $k'$ the order of $\ell\Id$ in $\GLmod{N}/H$. 
The adjacency matrix $A_i$ of $G_i$ is diagonalizable over $\C$ and, for each $k$-th root of unity $\zeta$, the number $(\ell{+}1)\zeta$ is an eigenvalue of $A_i$ of multiplicity one. The other eigenvalues of $A_i$ are complex numbers with angle in  $\tfrac{\pi}{k'}\mathbb{Z}$ and absolute value smaller than $$2\sqrt{\ell}-\left(4\sqrt{\ell}\right)^{-2(d-k)k'+1}\,,$$ where $d$ is the number of vertices of $G_i$.

\end{description}

\end{Theorem}

Theorem \ref{thm:main2} applies to the case of full level structure, where the adjacency matrix has non-real eigenvalues. In this case $\mathcal N_H=\GLmod{N}$, hence all connected components are isomorphic. We also have that $k=k'$ is the multiplicative order of $\ell$ in $(\Z/N\Z)^{\times}$, and the number of connected components is $n=\varphi(N)/k$, where $\varphi$ is the Euler totient function.

We can also apply Theorem \ref{thm:main2} to the isogeny graphs with torsion point level structure, namely $H=\{\smt *0*1\}$. 
In this case $\det(H)= (\Z/N\Z)^\times$, hence $G$ is connected and $k=1$. One might have $k'>1$, and indeed Corollary \ref{cor:asy_dist} implies that for $p$ big enough the adjacency matrix has non-real eigenvalues.

\begin{Remark}[Multipartite graphs]\label{rem:partition}

Given a finite connected directed graph $G=(V,E)$, a $k$-\emph{multipartition} is a partition of $V$ into $k$ disjoint subsets $V_j$ such that vertices of $V_j$ are connected only to vertices of $V_{j+1 \pmod k}$. A 2-partite graph is called bipartite. When $G$ is $d$-regular, this is related to the spectrum of the adjacency matrix $A$ of $G$ in the following way.  Let $u_j=\sum_{v\in V_j}v$, and $U$ the span of $\{u_1,\dots , u_k\}$ in $\C^V$. Then $U$ is stabilized by $A^t$, and $A^t$ restricted to $U$ acts as $d$ times a cyclic permutation, hence the spectrum of $A$ contains $d$ times the group of $k$-th roots
of unity. 

The Weil invariant gives a $k$-multipartion of the vertices of $G_i$, and by the above discussion this is a $k$-multipartion of $G_i$; the existence of this partition implies the existence of the eigenvalues $(\ell+1)\zeta$'s appearing in the statement of Theorem \ref{thm:main2}. Theorem \ref{thm:main2} also says that there are no other eigenvalues of absolute value equal to $\ell+1$, hence this partition can not be refined.
\end{Remark}

\subsubsection*{Organization of the paper} In Section \ref{S:Weil_pairing}, we reduce the proof of Theorems \ref{thm:main1} and \ref{thm:main2} to Theorem \ref{thm:spet1}. Along the way, we prove a few elementary results about isogeny graphs. 

Sections \ref{S:modular_curve} and \ref{S:relation} are devoted to set-up a more general framework to study isogeny graphs, and to prove Theorem \ref{thm:spet1} (= Theorem \ref{thm:spet2}). These sections rely on more advanced algebraic geometry notions. 
 
 In Section \ref{s:mocular_forms} we develop the connection between isogeny graphs and modular forms. This connection is of independent interest, and it is used to prove Corollary \ref{cor:asy_dist}.
 Throughout the paper, we keep track of automorphisms of the graphs. We relate them to automorphisms of modular curves and modular forms, such as the Fricke and Atkin-Lehner automorphisms. These results are not used in the proof of our main theorems, but we think they could be useful for further developments.

\subsection{Ramanujan graphs and expander sequences}\label{S:graph}

In this section we discuss the implication of our results from the point of view of graph theory. We refer the reader to the textbooks \cite{DSV,HLW,Trev}, the papers \cite{Bordenave,Friedman} and references therein for detailed discussions of the concepts introduced here.

Let $G$ be a $d$-regular non-bipartite (see Remark \ref{rem:partition}) connected finite graph with symmetric adjacency matrix $A$. The spectrum of $A$ contains the eigenvalue $d$, called trivial eigenvalue, with multiplicity one. All other eigenvalues are called non-trivial and are contained in the interval $(-d,d)$ (\cite[Proposition 1.1.2]{DSV}). The \emph{spectral gap} is the minimum of $d-|\lambda|$, where $\lambda$ runs among all non-trivial eigenvalues.
Notice that our main results give lower bounds on the spectral gap
of isogeny graphs.
Lower bounds on the spectral gap can be used, among the other things, to bound the diameter, the expansion constant and the mixing time of a graph, see \cite{DSV,HLW}. 
%we refer to \cite{DSV,HLW} and references therein for a detailed discussion of these concepts and their relations with the spectral gap.

A graph is called \emph{Ramanujan} if all non-trivial eigenvalues of $A$ are contained in the Hasse interval $[-2\sqrt{d-1},2\sqrt{d-1}]$. The Alon-Boppana inequality says that there exists a constant $c_d>0$, depending only on $d$, such that for every $d$-regular graphs with $n$ vertices there exists a non-trivial eigenvalue with absolute value at least $2\sqrt{d-1}-c_d/(\log(n))^2$; in a more colloquial language, it says that Ramanujan graphs have the largest possible spectral gap among big graphs (\cite[Section 5.2]{HLW}, \cite[Section 1.3]{DSV}, \cite[Introduction]{Bordenave}).

A key result, conjectured by Alon and proven in \cite{Friedman} and \cite{Bordenave}, says the following: fixing a strictly positive real number $\varepsilon$, using the uniform distribution on the set of $d$-regular simple graphs with $n$ vertices, the probability that all non trivial eigenvalues of the adjacency matrix lie in the interval $[-2\sqrt{d-1}-\varepsilon,2\sqrt{d-1}+\varepsilon]$ tends to $1$ when $n$ tends to infinity. In other words, a random graph is close to be Ramanujan. In \cite{HMY}, it is shown that, for $n$ big enough, approximately 69\% of the
$d$-regular 
graphs are Ramanujan.
It is however challenging to construct  examples of Ramanujan graphs, as discussed for instance in \cite[Introduction]{Bordenave}. 
Our results give the following.

\begin{Corollary}\label{cor:ram}
With the notation of Definition \ref{def:graph}, if $p$ is congruent to 1 modulo 12,  $H$ contains $\ell\cdot \Id$, and $\det(H)=(\Z/N\Z)^\times$, then the isogeny graph $G(p,\ell,H)$ is a Ramanujan graph.    
\end{Corollary}

The first three conditions guarantee that the adjacency matrix is symmetric, see Proposition \ref{prop:diag}; if we drop them, our main results say that the graphs are Ramanujan in some generalized sense. Corollary \ref{cor:ram} can be applied for instance to isogeny graphs with Borel level structure.

\medskip

With the same spirit, people have looked at \emph{expander sequences of graphs}. A sequence of $d$-regular connected finite graphs $G_i$ is an expander sequence if the adjacency matrices $A_i$ are symmetric, the number of vertices tends to infinity, and there exists a constant $\varepsilon>0$ independent of $i$ such that the spectral gap of $G_i$ is at least $\varepsilon$ for every $i$. We again refer to \cite{DSV,HLW} and references therein for a detailed discussion. Observe that in loc. cit. the definition is given in terms of the expansion constant; our definition in terms of spectral gap is equivalent to the classical one because of the Cheeger inequality (\cite[Sections 4.4 and 4.5]{HLW} and \cite[Section 1.2]{DSV}). The importance of constructing explicit examples is highlighted for instance in \cite[Section 2.1]{HLW} or \cite{Ram-LPS}. The following corollary 
of Theorem \ref{thm:main2} and Proposition \ref{prop:diag} provides many new examples of expander sequences of graphs.

\begin{Corollary}\label{cor:expander}
Fix a prime $\ell$ and a sequence of graphs $\{G_i\} = \{G(p_i,\ell,H_i)\}$ with $p_i\equiv 1 \pmod{12}$ and $H_i<\GLmod{N_i}$ a subgroup containing $\ell$, with determinant $\det(H_i) =(\Z/N_i\Z)^\times$, and such that $[\GLmod{N_i}:H_i]\cdot p_i$ tends to infinity. Then, $\{G_i\}$ is an expander sequence of graphs.

\end{Corollary}

The first example where Corollary \ref{cor:expander} can be applied is the classical sequence of isogeny graphs: $N_i=1$ for every $i$, and $p_i$ grows. New examples are for instance when $p_i$ is fixed and $[\GLmod{N_i}:H_i]\to\infty$, which happens e.g. if $N_i$ grows, and $H_i$ is of a fixed type such as Borel or Cartan; or when $p_i$ grows, $N_i$ and $H_i$ can be anything.

Again, if we drop the condition that $p_i$ is congruent to 1 modulo 12, and that $H_i$ contains $\ell$, then the adjacency matrix might not be symmetric and the sequence is expander in a generalized sense.

\medskip

Consider the largest number $\eta=\eta(p,\ell,H)$ such that the non-trivial eigenvalues of the adjacency matrix of $G(p,\ell,H)$ are contained in the Hasse interval shrunk by $\eta$, i.e. the interval $[-2\sqrt{\ell}+\eta,2\sqrt{\ell}-\eta]$. In other words, we are interested in the gap 
\begin{equation}\label{eq:gap}
\eta(p,\ell,H):=2\sqrt{\ell}-\max_{\lambda}|\lambda|\,,
\end{equation}
where the maximum is taken over all non-trivial eigenvalues of the adjacency matrix of $G(p,\ell,H)$. 
The estimates in our Theorems \ref{thm:main1} and \ref{thm:main2} can be rephrased as lower bounds on $\eta$.  Alon-Boppana inequality implies that there is a constant $c_{\ell}$ which depends only on $\ell$ such that $\eta(p,\ell,H)\leq c_\ell \log(|V(p,\ell,H)|)^{-2}$, where $V(p,\ell,H)$ is the set of vertexes of the graph $G(p,\ell,H)$.  Numerical experiments from Appendix \ref{appB} shows that our bounds is not sharp. Let us formulate the following general questions, which are open and interesting already in the case without level structure.

\begin{question}\label{q:gap}
With the notations of Corollary \ref{cor:ram}, fix $\ell$, $N$, and a subgroup $H$ of level $N$. Let $v(p)$ be the number of vertices of $G(p,\ell,H)$, and $\eta(p)=\eta(p,\ell,H)$. What are good lower bounds for $\eta(p)$? What is the asymptotic of $\eta(p)$? Does this sequence achieve the Alon-Boppana bound (i.e.  $\lim_{p\to \infty}\eta(p)\log(v(p))^{2}$ is equal to a positive constant)? 

\end{question}

\subsubsection{Non-backtracking random walks and mixing time} \label{sec:mix}
Random walks can be used to generate probability distributions on the vertices of a graph: one starts from a given distribution $\pi$ and, after a length $k$ random walk, the end vertex has distribution $\pi^{(k)}$. A distribution is stationary if it  does not change after random walk. By general graph theory, if the graph is connected and not multipartite, the sequence $\pi^{(k)}$ converges to the unique stationary distribution; a general references is \cite{HLW}. The mixing time measures
the speed of convergence, i.e. the number $k$ of steps such that $\pi^{(k)}$ is distant at most a certain $\varepsilon$ from the stationary distribution.

Motivated by applications to isogeny based cryptography, in Proposition \ref{mixing_times} we give an upper bound for the mixing time of non-backtracking random walks on isogeny graphs.

Let us first recall the definition of non-backtracking walk on isogeny graphs. Assume that $\ell \cdot \Id \in H$; for each edge $(E,\phi) \overset e\to (E', \phi')$ we can choose an edge
$(E',\phi') \overset {\ol e}\to (E, \phi)$ by choosing an isogeny $\alpha$ representing $e$ and taking $\ol e = [\hat \alpha]$, where $\hat \alpha$ is the dual isogeny of $\alpha$. A walk is non-backtracking if an edge $e$ is never followed by $\ol e$. 

(Note that if $j(E') = 0,1728$ the definition of $\ol e$ is not necessarily canonical: if $u \in \Aut (E',\phi')$ is not $u \pm 1$, the isogeny $\alpha \circ u$ still represents $e$, but $\hat\alpha$ and $\widehat{\alpha \circ u} = \hat{\alpha} \circ \hat u$ might have different kernels, hence correspond to different edges; more details in \cite[Section 3.3]{noi})

% , then for each edge $e$

% Due parole su non backtracking. Se j non 0 w 1728 bla , altrimenti seguiamo.

In this set-up, if we further assume that $\det(H)=(\Z/N\Z)^\times$ so that $G(p,\ell,H)$ is connected and not multipartite, the distribution $\pi^{(k)}$ obtained after a non-backtracking random walk always converges to a stationary distribution $s$ when $k$ goes to infinity. The Borel case is treated in details in \cite{noi}: as consequence of the Ramanujan property of the graph, \cite[Theorem 11]{noi} shows that the speed of convergence of $\pi^{(k)}$ to $s$ is $O(\tfrac{k}{\ell^{k/2}})$, and compute the precise constants. 

In our Theorem \ref{thm:main2} we show that the non-trivial part of the spectrum is contained in an interval $[-2\sqrt \ell +\eta, 2\sqrt\ell-\eta]$ for $\eta = (4\sqrt\ell)^{-2|V|+3}$, see also \eqref{eq:gap}, Question \ref{q:gap} and Appendix  \ref{appB}. This result is stronger than the Ramanujan property, and it implies that the rate of convergence of $\pi^{(k)}$ to the stationary distribution is $O(\tfrac{1}{\ell^{k/2}})$. We give a precise statement.

% \medskip

\begin{Proposition}\label{mixing_times}
In the notation of Definition \ref{def:graph}, suppose that $H$ contains $\ell\cdot \Id$ and that $\det(H) = (\Z/N\Z)^{\times}$. Let $G$ be the graph $ G(p,\ell,H)$ with set of vertices $V$. 
% Define $\eta = (4\sqrt\ell)^{-2|V|+3}$.
% , following \cite{thm} such that all the eigenvalues different from $\ell{+}1$ are contained in $[-2\sqrt{\ell}+\eta,2\sqrt{\ell}-\eta]$. 

Let $s$ be the probability distribution on $V$, where each vertex $(E,\phi)$ has probability  proportional to $|\Aut(E,\phi)|^{-1}$. 
Then $s$ is a stationary distribution and for each probability distribution $\pi$ on $V$, the distribution $\pi^{(k)}$ defined above converges to $s$.  More precisely for $p\neq 2,3$ we have the following inequality
$$
d_{TV}(\pi^{(k)},s)\leq \frac14 \sqrt{(p{-}1)[\GLmod{N}:H]} \cdot \min\left( \left( \tfrac{\eta}{\sqrt\ell}-\tfrac{\eta^2}{4\ell} \right)^{-\tfrac12}, \tfrac{(\ell+1)(k+1)-2}{(\ell+1)} \right) \cdot \frac{1}{\sqrt{\ell^k}}\,,
$$  
where $d_{TV}(p_1,p_2) = \tfrac 12\sum_{v\in V}|p_1(v)-p_2(v)|$ is the total variation distance and $\eta>0$ is such that all the non-trivial eigenvalues of $G$ lie in $[-2\sqrt \ell +\eta, 2\sqrt\ell-\eta]$.

For $p=2$, respectively $p=3$, the above inequality is true after substituting ``\,$\le \tfrac 14 \ldots$'' with ``\,$\le \tfrac 12 \ldots$'', respectively with ``\,$\le \tfrac 1{\sqrt 2} \ldots$'.
\end{Proposition}

The proof follows the same strategy as in \cite[Theorem 11]{noi}, the main difference is that we bound the value in Equation (5) in loc. cit. not only as in Equation (6), but also using the inequalities $|\sin((k+1)\theta)|\le 1$, $|\sin((k-1)\theta)|\le 1$  and 
$$|\sin(\theta)| = \sqrt{1-(\cos\theta)^2} = \sqrt{1-\tfrac{\lambda_i^2}{4\ell}} \ge \sqrt{1-\tfrac{(2\sqrt\ell-\eta)^2}{4\ell}}\,.$$

To connect with the notation in  \cite[Theorem 11]{noi}, the quadratic form $Q$ in that theorem is the same as $\frac 12 H$ here (see Equation \eqref{eq:Hermitian}) and the constants $K,M$, related to $||
\cdot ||_{TV}$ and $Q$, can be computed or bounded as in \cite[Theorem 7]{noi}:
$$
K = \left(\tfrac{p-1}{12}[\GLmod{N}:H]\right)^{-1/2}\,, \qquad 
M \le \left( \max_{i}  \tfrac{|\Aut(E_i,\phi_i)|}{2} \right)^{1/2} \, ,
$$
so that $M \le \sqrt 3$ for $p\neq 2,3$, $M \le \sqrt 6$ for $p=3$ and $M \le \sqrt {12}$ for $p=2$.

Automophisms

\subsubsection{Asymptotic distribution of the eigenvalues}

Let us now look at the distribution of all eigenvalues, the bulk of the spectrum following the terminology of \cite[Section 7.1]{HLW}, see also \cite[Section 8]{Serre}. Given a sequence $G_i$ of graphs as above and an angle $\theta$, we introduces the probability measure 
$$\mu(G_i,\theta) := \frac{1}{|\sigma(A_i,\theta)|}\sum_{\lambda \in \sigma(A_i,\theta)}\delta_{\lambda}\,,$$
where $\sigma(A_i,\theta)$ is the set of eigenvalues of the adjacency matrix $A_i$ with phase $\theta$ or $\theta+\pi$, 
 and $\delta_{\lambda}$ is a Dirac mass at $\lambda$; of course the definition makes sense only if $|\sigma(A_i,\theta)|\neq 0$. The limits of the sequences $\{\mu(G_i,\theta)\}$, if they exist, give the asymptotic distribution of the spectrum of $G_i$. If all eigenvalues of $A_i$ are real, we omit the dependence from $\theta$.

Let us also introduce the Kesten–McKay measure (also known as Kesten–McKay law or distribution)
\begin{equation}\label{eq:mckay}
	\mu_{\ell}:=\frac{\ell+1}{\pi}\frac{\sqrt{\ell-x^2/4}}{\ell(\ell^{1/2}+\ell^{-1/2})^2-x^2}dx
\end{equation}
supported in the Hasse interval $[-2\sqrt{\ell},2\sqrt{\ell}]$; it is the asymptotic distribution of the eigenvalues of a random sequence of ($\ell{+}1)$-regular graphs with increasing number of vertices, see \cite{McKay}, \cite[Theorem 7.2]{HLW} and references therein.

The following result, which relies on the theory of modular forms, is a corollary of 
Theorems \ref{thm:modular_special} and \ref{thm:trace_asymptotic}. 
\begin{Corollary}\label{cor:asy_dist} Fix a subgroup $H < \GLmod{N}$, a prime number $\ell$ coprime with $N$, and let $\{p_i\}$ be an increasing sequence of prime numbers not dividing $N\ell$. Let $G_i = G(p_i,\ell,H)$,
\begin{itemize}

\item If $H =\{\Id\}$, i.e. $G_i$ are isogeny graphs with full level structure, given $k'$ the order of $\ell$ in $(\Z/N\Z)^\times$, 
%\Id$ in $\GLmod{N}$, 
then for every $\theta$ in $\frac{\pi}{k'}\mathbb{Z}$ we have 
$$
\lim_{i\to \infty}\mu(G_i,\theta)= e^{i \theta}\mu_{\ell} \ ,
%G\theta\mu_{\ell}
$$
and for all other choices of $\theta$ there are no eigenvalues.

\item %Let $G_i$ be the graph $G(p_i,\ell,H)$, for 
If $H$ is the Borel subgroup, then all eigenvalues are real and 
$$
\lim_{i\to \infty}\mu(G_i)=\mu_{\ell}\ ,
$$ 
\item If $H= \{ \smt *0*1\}$, i.e. the $G_i$'s are graphs with torsion point structure, denoting $k'$ the order of $\ell$ in $(\Z/N\Z)^\times$, then for every $\theta$ in $\frac{\pi}{k'}\mathbb{Z}$ we have 
$$
\lim_{i\to \infty}\mu(G_i,\theta) = e^{i\theta}\mu_{\ell}\ ,
$$ 
and for all other choices of $\theta$ there are no eigenvalues.

\item %Let $G_i$ be the graph $G(p_i,\ell,H)$, where 
If $H$ is a non-split Cartan, %G $N$-level
then all eigenvalues are real and 
$$
\lim_{i\to \infty}\mu(G_i)=\mu_{\ell}\ .
$$ 
\end{itemize}
\end{Corollary}

It is instructive to note that Corollary \ref{cor:asy_dist} alone does not imply that all eigenvalues are contained in the Hasse interval: it does not prevent  a small number of eigenvalues to lie outside the support of the asymptotic distribution.

By general graph theory, Corollary \ref{cor:asy_dist} implies that $G_i$ has few cycles, more precisely the number of cycles of a fixed length divided by the number of vertices of $G_i$ tends to zero when $i$ tends to infinity, see \cite{McKay} and \cite[Theorem 10]{Serre}.

\subsection{Relation with isogeny based cryptography}\label{S:crypto}

Usually the security and sometime the design of protocols from isogeny-based cryptography relies on features of isogeny graphs. Often the security is related to the mixing time, the number of cycles, or to the spectral gap of the graphs. All these features can be studied looking at the spectrum of the adjacency matrix. We again refer to \cite{HLW} or other textbooks in Graph Theory or Markov Chains for a general discussion of this topic, see also our Section \ref{S:graph}

The first appearance of isogeny graphs in cryptography is the Charles--Goren--Lauter hash function \cite{CGL}, where the digest of a message is computed trough a walk on a classical isogeny graph. 

An important instance of isogeny based cryptography is the key exchange protocol SIDH \cite{SIDH}. In this protocol, the public key is a pair of vertices on the isogeny graph with full level structure
% at a known distance, 
and the private key is a walk between them. This protocol has been broken \cite{CD, MM,Rob}: if $N$ is big enough with respect to the length of the walk, as in SIDH, there are efficient algorithms to find a path between the two vertices. If $N$ is small with respect to the length of the walk, still we do not know an efficient algorithm to find a path. %Let us recall that the classical generic algorithm to find a path between two vertexes in a graph is meet-in-the-middle; its efficiency depends on the spectral gap of the graph.

Given the importance of the full level structure in the SIDH attacks, new light has been shed on the torsion and on how to leverage it: new key exchanges, some with the intent of ``fixing'' SIDH, have been proposed and once again public and private keys can be represented using a convenient isogeny graph $G(p,\ell,H)$, now with $H$ different from $\{\Id\}$ and $\GLmod{N}$.
% !) and private keys are walks between the two vertices.
% By now, many variants of SIDH have been proposed. Public keys can always be interpreted as a pair of vertices on an isogeny graph with convenient level structure. 
% Depending on the protocol, their distance can be either a public or a private parameter. Private keys are walks between the two vertices. 
For instance, in \cite{BorisM}, the group $H$ defining the level structure is the group of scalar matrices. In \cite{Festa} the authors use a commutative  $H$ to hide information:  they mainly use a split Cartan subgroup and they also propose the group of circulant invertible matrices; we notice that the non-split Cartan subgroup could be considered as a viable option. Another scheme leveraging early torsion attacks is proposed in \cite{Seta}. The article \cite{torsionDFP} discusses the security of various level structures in cryptographic protocols.

It is not known if there is some intrinsic property of the isogeny graphs which makes the path finding problem more difficult for some level structure rather than others. Let us recall that the efficiency of general path find algorithms on graphs depends on the spectral gap.

\medskip
% From a different perspective, in \cite{noi} a Zero Knowledge Proof is defined using random walks on the isogeny graph with Borel level structure. A precise analysis of the spectral gap and, consequently, of the mixing time, is used to prove that the Proof of Knowledge is statistically secure \cite[Theorem 11]{noi}. Our stronger-than-usual bound on the size of non-trivial eigenvalues, i.e. the bound $2\sqrt{\ell}-\left(4\sqrt{\ell}\right)^{-2d+1}$ instead than the usual $2\sqrt{\ell}$ in Theorem \ref{thm:main1}, permits to speed up the protocol as explained in \cite[Page 12]{noiBis}. In particular, In \cite[Theorem 11]{noi}, the linear term in $k$ can be replaced by a constant. The value of the constant depends on exact value of the spectral gap: the bigger is the value $\eta$  introduced in Equation \eqref{eq:gap}, the smaller is the constant.

%It is clearer, instead, that mixing properties of isogeny graphs imply security properties for signatures and identification protocols. Precise estimates on the mixing time as in Proposition \ref{mixing_times} allow to determine which parameters ensure that a certain datum is randomly distributed.  
From a different perspective, estimates on the mixing time of random walks on isogeny graph, as the one from Proposition \ref{mixing_times}, are used to ensure that certain schemes are secure. 
For example, in \cite{noi} non-backtracking random walks on the isogeny graph with Borel level structure are used to define a Zero Knowledge Proof. The length of the walk is a parameter of the scheme, let us denote it by $k$. As discussed in Section \ref{sec:mix}, these walks give a probability distribution $\pi^{(k)}$ on the graph. Following \cite{noi}, the statistical security of the Zero Knowledge Proof is determined by the total variation distance between $\pi^{(k)}$ and the stationary distribution: the closer they are, more statistically secure is the Proof. The analysis of the mixing time from \cite[Theorem 7]{noi} is thus used to give a precise relation between the length of the walk and the security of the Proof.

%allow to determine which parameters ensure that a certain datum is randomly distributed.  

The last isogeny based protocol being submitted to the NIST for standardization is the signature scheme SQISign (see \cite{SQI1D} and the recent developments \cite{SQIsignHD:}, \cite{SQIsign2D-West}, \cite{SignBoris} and \cite{SQISign2D-East}). In SQISign, a random walk on an isogeny graph with trivial level structure is used to choose a supersingular elliptic curve. In the same fashion of the Zero Knowledge Proof discussed above, the length of the walk is related to the security of the scheme via an estimate of the mixing time. The estimate used in SQIsign is discussed in \cite[Lemma 20]{SQIsign2D-West} and \cite[Proposition 29]{SQIsignHD:}, which again rely on \cite[Theorem 11]{noi}. 

Proposition \ref{mixing_times} improves \cite[Theorem 7]{noi}; in principle, it could imply that to achieve a certain security level one needs shorter walks than the ones requested in the above cited papers, and this could improve the efficiency of the schemes. Looking at the commonly used parameters and the dependence on $p$ of the constants in Proposition \ref{mixing_times}, right now it seems that our result does not give relevant improvements. However, the numerical experiments from Appendix \ref{appB} suggest that the constants in Proposition \ref{mixing_times} could be improved.

%but still a dependence on $p$ is forced by the Alon-Boppana inequality discussed in Section \ref{S:graph}. 

\medskip

To conclude this section, we explain how Theorem \ref{thm:main2} implies Conjectures 1 and 2 from \cite{DG} (as long as the restriction from \cite[Section 4]{DG} about the Weil pairing is taken into account). These conjectures are about valid keys in isogeny based cryptography are true. Their proof will rely on the following corollary of Theorem \ref{thm:main2} applied to the case $H = \{\Id\} \times B_0(M)$ in $\GLmod{NM} = \GLmod{N} \times \GLmod{M}$ for $M,N$ coprime integers (we actually need only the description of the connected components of the isogeny graphs from Theorem \ref{thm:main2}).

\begin{Corollary}\label{cor:valid_keys}
    Let $p$, $\ell$, $N$ and $M$ be pairwise coprime numbers such that $p$ and $\ell$ are prime; let $(E,P,Q, C)$ and $(E',P',Q',C')$ be two supersingular elliptic curves over $\overline{\F}_p$ endowed with basis 
    % $(P, Q)$ and $(P',Q')$ 
    of the $N$ torsion and cyclic subgroups 
    % $C,C'$ 
    of order $M$. 
    
    Then, there exists $m\in \Z$ and an isogeny $\phi\colon E \to E'$ of degree $\ell^m$ such that 
    $\phi(P)=P'$, $\phi(Q)=Q'$ and $\phi(C)=C'$ 
        % $$ \phi(P)=P'\,, \quad \phi(Q)=Q \quad \text{and} \quad \phi(C)=C' \,,$$ 
    if and only if for some integer $k>0$ we have $w(P,Q)=w(P',Q')^{\ell^k},$
    where $w$ is the Weil pairing on the $N$-torsion.
\end{Corollary}
Let us now explain and prove the conjectures in \cite{DG}. Fix a supersingular elliptic curve $E$ with a basis $\{P,Q\}$ of the $N$-torsion. A valid degree insensitive public key is a supersingular elliptic curve with level $N$ structure $(E',P',Q')$ such that there exists $m\in \Z$ and an isogeny of degree $\ell^m$ $\phi\colon E\to E'$ such that $\phi(P)=P'$ and $\phi(Q)=Q'$. Conjecture 1 in \cite{DG}, taking into account the obstruction presented in Section 4 in loc. cit,  states that all supersingular elliptic curves with level $N$ structure are valid degree insensitive key as long as for some integer $k>0$ we have $w(P,Q)=w(P',Q')^{\ell^k}$. This is a direct application of the corollary with $M=1$.

% Corollary \ref{cor:valid_keys} can be of interest for isogeny-based cryptography: for instance, it implies that both Conjecture 1 and Conjecture 2 from \cite{DG} about valid keys for isogeny-based key exchanges, as long as the restriction from \cite[Section 4]{DG} about the Weil pairing are taken into account, are true. 
%We now deduce the two conjectures from \cite{DG} from our corollary; we will use the notations from loc. cit.

%Conjecture 1 in \cite{DG} is a direct application of the Corollary with $M=1$. 

Conjecture~2 is about degree insensitive SIDH squares (these squares are now often called isogeny diamonds). Fix two different primes $\ell_A$ and $\ell_B$, and two numbers $e_A$ and $e_B$. A SIDH square is a commutative diagram
$$\begin{tikzcd}[row sep=huge, column sep=huge]
(E,P_A,Q_A,P_B,Q_B) \arrow[r, "\varphi_A"] \arrow[d, "\varphi_B"'] & (E_A',P_B',Q_B') \arrow[d, "\varphi_{BA}"] \\
(E_B',P_A',Q_A') \arrow[r, "\varphi_{AB}"'] & E_{AB}
\end{tikzcd}
$$
where $\{P_A,Q_A\}$ and $\{P'_A,Q'_A\}$ are basis of the $\ell_A^{e_A}$-torsion, and $\{P_B,Q_B\}$ and $\{P'_B,Q'_B\}$ are basis of the $\ell_B^{e_B}$-torsion; the isogenies $\phi_A$ has degree $\ell_A^{m_A}$ and $\phi_B$ has degree $\ell_B^{m_B}$ for some integers $m_A$ and $m_B$, and respect the level structure.  \cite[Conjecture 2]{DG} says that any four curves with level structure $(E,P_A,Q_A,P_B,Q_B)$, $(E_A',P_B',Q_B')$ ,
$(E_B',P_A',Q_A')$, and $E_{AB}$ can be put in such a diagram as long as there are no obstructions from the Weil pairing, i.e. $w(P_A,Q_A)=w(P_A',Q_A')^{\ell_A^{k_A}}$ and $w(P_B,Q_B)=w(P_B',Q_B')^{\ell_B^{k_B}}$ for some integers $k_A$ and $k_B$. Let us show it.

%consider $(E,P_A,Q_A)$, $(E_A,P_B',Q_B')$, $(E_B, P'_A, Q'_A)$ and $E_{AB}$ in the set $\chi^{di}$ of degree insensitive SIDH squares, as in \cite[Section 2]{DG}. 
%Supposing that there is no obstruction coming from the Weil pairing, % $w(P'_A, Q'_A) = w(P_A,Q_A)^{\ell_B^k}$ for some $k$, 
 There exist infinitely many isogenies $\phi_B \colon E \to E_B$ of degree $\ell_B^m$ sending $(P_A,Q_A) \mapsto (P_A',Q_A')$ (for the ``infinitely many" part we can compose  closed non-backtracking walks in the graph with full level structure with a walk from Corollary \ref{cor:valid_keys}). 
For $m$ big enough, by Proposition \ref{mixing_times} we can find a cyclic isogeny $\psi\colon E_A \to E_{AB}$ of the same degree $\ell_B^m$. 
Denoting $C_B = \ker(\phi_B)$, $K = \ker(\psi)$ and using that $w(P'_B, Q'_B) = w(P_B,Q_B)^{\ell_A^k}$ for some $k$,  Corollary \ref{cor:valid_keys} gives isogeny $\phi_A \colon E\to E_A$ of degree a power of $\ell_A$, sending $P_B \mapsto P'_B$, $Q_B \mapsto Q'_B$ and $C_B \mapsto K$. This last condition implies that we can complete the square with a map $E_B \to E_{AB}$, forming the requesting diagram and completing the proof.%a push-out as in Figure 1 of \cite{DG}, hence the conjecture.

\subsection{Relation with other works} \label{sec:otherworks}

The Ramanujan property of isogeny graphs without level structure is usually attributed to A. K. Pizer \cite{Pizer}. There, the result is stated for graphs constructed using Brandt matrices. These matrices are defined using quaternion algebras and theta series; they represent the action of Hecke operators on spaces of modular forms. Using the moduli space of elliptic curves over $\Z$, one can deduce the bound of the spectra of Hecke operators on modular forms from Deligne's proof of the Weil's conjecture. Finally, the Deuring correspondence provides the relations between Brandt matrices and isogeny graphs. This approach is taken up in full details in \cite[Section 3]{noi}, where it is extended to the case of isogeny graph with Borel level structure. A few weeks after this article appeared on ArXiv, A. Page and B. Wesolowski in \cite{OneEnd} further generalize this approach and gave an independent proof of a variant of our Theorem \ref{thm:main2}. They use a generalization of the Deuring correspondence - the adelic Deuring correspondence - to relate isogeny graphs with level structure to graphs defined using quaternion algebras. They then use the Jacquet–Langlands correspondence to relate the adjacency matrices of this graph to the action of Hecke operators on spaces of modular forms. As in the case without level structure, the bound on the spectra of these operators is a consequence of Deligne's proof of Weil conjecture. Similar arguments can also be found in \cite{CGL2}.

\medskip

Here we follow ideas from \cite{Rib,Emerton}. Using the moduli space of elliptic curves over $\mathbb{Z}$, we directly relate the adjacency matrix of the isogeny graphs with level structure to the action of certain Hecke operators on the coohomology of the moduli space of elliptic curves over $\mathbb{F}_{\ell}$. We then directly apply Deligne's proof of Weil conjecture. Our approach does not involve quaternion algebras, the Deuring and the Jaquet-Langlands correspondences, and modular forms, but still boils down to Deligne's Theorem. %\Giulio{non so se mettere la frase seguente, forse no}It is fair to say that the circle of ideas behind both approaches is the same, and in particular in both cases one eventually deduces the final claim from the Weil conjecture.

Let us also briefly survey a few papers on isogeny graphs which do not focus on the spectrum of the adjacency matrix. The Borel level structure case is studied by Arpin in \cite{Arpin}. The zeta-function of isogeny graphs with Borel level structure is studied in \cite{zeta-functions}. Other interesting papers are \cite{Land,Iez,towers}. 
In \cite{Pin,Arp_cic}, there is nice bound on the number of cycles on classical isogeny graphs obtained using different methods from ours.

Isogeny graphs of ordinary curves are studied by Kohel \cite{Kohel}, they have a rather different (and simpler!) structure from the supersingular ones, sometime they go by the names of volcano graphs or jellyfish graphs.

\begin{Remark}[Relation with the graphs $X^{p,q}$ from \cite{DSV} ]
Fixing $p=2$ and taking $H$ to be the group of scalar matrices modulo~$N$ for a prime level $N$, we get graphs closely related to the Cayley graphs $X^{\ell,N}$ from \cite[Chapter 4]{DSV}, and references therein. Indeed, there is only one supersingular elliptic curve $E/{\overline{\F}_2}$ up to isomorphism, that is the one with equation $y^2+y = x^3$, whose endomorphism ring of $E$ is a quaternion order $\calO$ that contains the integral Hermite quaternions $\Z[i,j,k]$ and it is contained in $\tfrac12 \Z[i,j,k]$, so that all $\ell$-isogenies of our graphs can be represented as elements of norm $\ell$ in $\calO$, and actually, up to postcomposition by automrphisms, they are the elements of $S_\ell \subset \Z[i,j,k]$ defined in \cite[page 68]{DSV}\footnote{Indeed all elements $\gamma_i$ of such an $S_\ell$ define an $\ell$-isogeny of $E$ and they are all distinct up to post-composition by automorphism because for $i \neq j$ the element $\gamma_i^{-1} \gamma_j = \tfrac 1\ell \ol{\gamma_i}\gamma_j$ is in $\tfrac{1}\ell \Z[i,j,k]$ but not in $\Z[i,j,k]$, hence it is not in $\tfrac{1}2 \Z[i,j,k]$ which contains $\calO$}. Moreover, the isomorphism $\F_N[i,j,k] \cong M_2(\F_N)$ in loc. cit. keeps track of the action of $\End(E)$ on $E[N]$: we get a bijection between $\mathrm{PGL}_2(\F_N)$ and $H$-structures on $E$, and also a bijection isomorphism classes of supersingular elliptic over $\overline{\F}_2$ together with an $H$-structure with the quotient of $\mathrm{PGL}_2(\F_q)$ by the image of $\Aut(E)$. In particular, a connected component of the isogeny graph $G(2,\ell,H)$ (all components are isomorphic) is a quotient of the Cayley graph $X^{\ell,N}$ obtained in \cite{DSV} reducing the elements of $S_\ell$ modulo~$N$.

\end{Remark}

% In \cite{Pin} there is nice bound on the number of cycles on classical isogeny graphs.% there the authors also asks if the cycles are uniformly distribuited on the graphs or they cluster in some region (and our main result, together with some general graph thery, should suggest that they do not cluster)

\subsubsection*{Acknowledgments} 
We have had the pleasure and the benefit of conversations about the topics of this paper with S. Arpin, A. Basso, P. Caputo, L. De Feo, T. B. Fouotsa, T. Morrison, M. Sala, M. Salvi, R. Schoof, S. Vigogna and F. Viviani.
The first author also would like to thank the organizers and the participants of the Banff/Bristol 2023 workshop ``Isogeny graphs in Cryptography" for many discussions on the topics of this paper. We also thank the anonymous referee for his/her suggestions.

Both authors are supported by the MIUR Excellence Department Project MatMod@TOV awarded to the Department of Mathematics, University of Rome Tor Vergata, the ``National Group for Algebraic and Geometric Structures, and their Applications" (GNSAGA - INdAM), and the PRIN PNRR 2022 ``Mathematical Primitives for Post Quantum Digital Signatures". The second author is also supported by ``Programma Operativo Nazionale (PON) “Ricerca e Innovazione” 2014-2020.

\section{First properties of isogeny graphs and reduction of Theorems \ref{thm:main1} and \ref{thm:main2} to Theorem \ref{thm:spet1}}\label{S:Weil_pairing}
We fix $p$ and $\ell$ distinct prime numbers, $N$ a positive integer coprime to $p\ell$ and $H$ a subgroup of $\GLmod{N}$, together with the isogeny graph $G = G(p,\ell,H)$ in Definition \ref{def:graph} and its vertex set $V$.
The adjacency matrix $A$ defines a linear operator $A\colon \Q^V\to \Q^V$, hence also an operator $\C^V \to \C^V$, which maps a vertex $v$ to $\sum v_i$, where the sum runs over all edges $v \to v_i$ coming out of $v$.

\subsection{Automorphisms of isogeny graphs}\label{s:aut_G}

All isogeny graphs have the following automorphism, which is usually called Galois or Frobenius automorphism. 
\begin{subDefinition}[Frobenius automorphism]\label{def:Galois_inv}
	Let $\sigma$ be the Frobenius automorphism of $\overline{\F}_p/\F_p$, then $$\langle\sigma\rangle \colon G \to G$$  maps a vertex $(E,\phi)$ to the conjugated $(E^{\sigma},\phi^{\sigma}:=\sigma\circ\phi)$, and an isogeny to the conjugated by $\sigma$.
\end{subDefinition}

% For the next subsection, and for other reasons later on, 
% we will need the following operators.
Adding level structure, naturally enriches isogeny graphs with the following automorphisms.

\begin{subDefinition}[Diamond and matricial automorphisms]\label{def:aut}
	Let $G$ be as in Definition \ref{def:graph}. For every $g$ in the normalizer $\mathcal N_H$ of $H$ in $\GLmod{N}$ we define an automorphism
	$$\langle g \rangle\colon G \to G $$
	$$ (E,\phi)\mapsto (E, \phi \circ g) $$
	
	In particular, for every $d$ in $(\Z/N\Z)^{\times}$, the diamond operator $\langle d \rangle $ is the automorphism associated with the diagonal matrix $ \smt d\,\,d$.
\end{subDefinition}

Observe that if $d = \smt d\,\,d$  belongs to $H$, then $\langle d \rangle$ is the identity. Moreover, even if $-1\notin H$, then $\langle -1 \rangle$ is the identity because $(E,-\phi)$ is always isomorphic to $(E,\phi)$.

Notice that, up to isomorphism, we can suppose that each elliptic curve $E_i$ in our graph is defined over $\F_{p^2}$ and that the Frobenius endomorphism $\Frob_{p^2}\colon E_i \to E_i$ acts as $[-p]$. Since the map $\sigma\colon E(\overline{\F}_p) \to E^\sigma(\overline{\F}_p)$ coincides with the action of $\Frob_p\colon E\to E^\sigma$, we deduce that $\langle \sigma \rangle ^2 = \langle p\rangle$  on the graph: indeed, for each vertex $(E_i,\phi_i)$ we have 
$$
\langle\sigma \rangle ^2 (E_i,\phi_i) = (E_i^{\sigma^2}, \sigma^2\circ \phi_i) = (E_i, \Frob_{p^2}\circ \phi_i) = (E_i, [-p]\circ \phi_i) = \langle -p \rangle (E_i,\phi_i) = \langle p \rangle (E_i,\phi_i)\,,
$$
where the last equality is true because $\langle-1\rangle$ is the identity.

\begin{subProposition}\label{p:quotient}
For every $p$, $\ell$, $N$, and $H$, the isogeny graph $G(p,\ell,H)$ is the quotient of the isogeny graph with full level structure $G(p,\ell,\{\Id\})$ by the action of $H$ given in Definition \ref{def:aut}. In particular, the spectrum of the adjacency matrix of $G(p,\ell,H)$ is a subset of the spectrum of the adjacency matrix of $G(p,\ell,\{\Id\})$.
\end{subProposition}

Using Proposition \ref{p:quotient}, one could deduce most of our results from the case of full level structure. However, we prefer to give proofs that directly work for any level structure.

To introduce the Atkin-Lehner automorphisms we need some more notations. If $N=N'q$, with $N'$ and $q$ coprime, and the group $H$ is a product $H'\times B$ in $\GLmod{N'}\times \GLmod{q}$ , then a level $H$ is structure consist of a level $H'$ and a level $B$ structure. Now assume that $q$ is a prime power, and $B$ is a Borel subgroup of $E[q]$; then a level $B$ structure consist of a cyclic subgroup $C$ of $E[q]$ of rank $q$.

\begin{subDefinition}[Atkin-Lehner automorphism]\label{def:ALgraph}
With the above notation,  the Atkin-Lehner automorphism
$$
w_q\colon G \to G
$$
maps $(E,\phi,C)$ to $(E/C,\pi\circ \phi, E[q]/\pi(C))$, where $\phi$ is a level $H'$ structure, $C$ is a level $B$ structure, and $\pi$ is the quotient map $E\to E/C$.
\end{subDefinition}

 When $N = N' q_1\cdots q_r$ with $N, q_1, \ldots, q_r$ coprime and $q_i$ prime powers, and $H = H' \times \prod B_i$ in $\GLmod{N'}\times \prod \GLmod{q_i}$, the above definition gives $r$ Atkin-Lehner automorphisms $w_{q_1}, \ldots, w_{q_r}$.

%Further automorphism, will be introduced in Definition \ref{def:AL}. 

\subsection{Hermitian form and diagonalization}
With keep the notation of Definition \ref{def:graph}. We introduce the following Hermitian form $H$ on $\C^V$
\begin{subeqn}\label{eq:Hermitian}
H\left( (E_i,\phi_i),(E_j,\phi_j)\right) = \delta_{ij} a_i \,, 
\end{subeqn}
with $a_j = |\Aut(E_i,\phi_i)|$ and $\delta_{ij}$ is the Kronecker delta.

\begin{subProposition}[Adjoint of the adjacency matrix]\label{prop:diag}
Let $G$ and $A$ be as in Definition \ref{def:graph} and let $A^*$ be its adjoint with respect to the Hermitian form (\ref{eq:Hermitian}). Then, 

$$A^*=\langle \ell^{-1} \rangle A\,. $$

The adjacency matrix $A$ is diagonalizable, and the angles of its eigenvalues lie in $\tfrac{\pi}{k'}\mathbb{Z}$, where $k'$ is the minumum positive integer such that $\ell^{k'} \Id \in H$. In particular:
\begin{itemize}
	\item the operators $A$ and $A^*$ commute, are both diagonalizable, have the same spectrum, and are conjugate;
	\item if $\ell$ belongs to $H$, then $A=A^*$ and the spectrum of $A$ is real; 
	\item if $\ell$ belongs to $H$ and $p$ is congruent to 1 modulo 12, the adjacency matrix is symmetric.
\end{itemize}
\end{subProposition}
\begin{proof}
For the first part, we need to prove that, given vertices $(E_i,\phi_i)$ and $(E_j, \phi_j)$ we have
\begin{subeqn}\label{eq:adj}
	H\big(A {\cdot} (E_i,\phi), (E_j,\phi_j)\big) =  H\big((E_i,\phi_i),  \langle \ell^{-1} \rangle A{\cdot} (E_j,\phi_j)\big)\,,
\end{subeqn}
where we interpret  $(E_i,\phi_i)$ and $(E_j, \phi_j)$ as elements of $\C^V$. Let $L$ be the set of degree $\ell$ morphisms $(E_i,\phi_i) \to (E_j,\phi_j)$, and let $M$ be the set of degree $\ell$ morphisms $(E_j,\phi_j) \to (E_i,[\ell]\phi_i)$. Then, using the definition of $A$, and the definition (\ref{eq:Hermitian}) of $H$, we find that 
\[
H(A{\cdot} (E_i,\phi), (E_j,\phi_j)) = \frac{\# L \cdot  \# \Aut(E_j,\phi_j)}{\# \Aut(E_j,\phi_j)} \,,\quad  H((E_i, \phi_i), \langle \ell \rangle A {\cdot} (E_j,\phi_j)) =  \frac{\# M \cdot \#\Aut(E_i, \phi_i)}{\#\Aut(E_i,[\ell]\phi_i)}\,.
\]
We notice that $\Aut(E_i,\ell\phi_i)$ equals $\Aut(E_i, \phi_i)$ as subgroup of $\Aut(E_i)$. Hence equation \eqref{eq:adj} is equivalent to the fact that $ L$ and $M$ have the same cardinality: indeed duality of isogenies gives a bijection between the two.

Since diamond operators commute with $A$, then $A$ is a normal operator, hence diagonalizable. Moreover, the adjoint of $A^{k'}$ is equal to $\langle\ell^{k'}\rangle A^{k'}= A^{k'}$, hence $A^{k'}$ is Hermitian and has real eigenvalues. We deduce that for each $\lambda$ in the spectrum of $A$, its power $\lambda^{k'}$ is real, hence the angle of $\lambda$ lies in $\tfrac{\pi}{k'}\mathbb{Z}$.

The operator $A^*$ is also diagonalizable. Since $A$ and $A^*$ commute, they have the same eigenvectors. The corresponding eigenvalues are conjugated. Since $A$ is real, its spectrum is invariant under conjugation, hence $A$ and $A^*$ have the same spectrum. Since $A$ and $A^*$ are both diagonalizable with the same spectrum, they are conjugate.

If $p$ is congruent to 1 modulo 12, all supersingular elliptic curves have $\{\pm 1\}$ as automorphism group, and hence all vertices $(E_i,\phi_i)$ have the same number $a_i$ of automorphisms: if $-1\in H$, then $a_i=2$, otherwise $a_i=1$. Then, the Hermitian form from Equation \eqref{eq:Hermitian} is a multiple of the standard form, and being self-adjoint coincides with being symmetric.
\end{proof}

\begin{subRemark}\label{rkm_A_star}
Since the Hermitian form (\ref{eq:Hermitian}) is presented in diagonal form, it is easy to write down the entries of $A^*$: for each $i$ we have
\begin{subeqn}\label{eq_A*}
	A^*((E_i,\phi_i))=a_i\sum_j a_j^{-1} (E_j,\phi_j)\,, 
\end{subeqn}
where  $a_i,a_j$ are as in Equation \eqref{eq:Hermitian}, and the sum runs over all edges $(E_j,\phi_j)\to (E_i,\phi_i)$, namely all the edges in $G$ with end-point $(E_i,\phi_i)$. We notice that the entries of $A^*$ are integers: any vertex $(E_j,\phi_j)$ appearing in the right hand side of \eqref{eq_A*} has multiplicity $a_ia_j^{-1} \cdot (\# S/a_i) = \# S/a_j$, for $S$ the set of degree $\ell$ isogenies $(E_j,\phi_j) \to (E_i,\phi_i)$; since $\Aut(E_j,\phi_j)$ acts freely on $S$ by precomposition, then $\# S/a_j$ is an integer.
\end{subRemark}

\subsection{Weil pairing and reduction of Theorems \ref{thm:main1} and \ref{thm:main2} to Theorem \ref{thm:spet1}}

To formulate the arguments in this subsection, we look at the following graph.

\begin{subDefinition}    
    Given $H$ a subgroup $\GLmod{N}$ and $\ell \in (\Z/N\Z)^\times$, the oriented
Caley graph $C = C(N,\det(H), \ell)$ is the graph whose  vertices are the element of  $R_H=\mu_N^\times(\overline{\F}_p)/\det(H)$ and such that there is an edge from $\xi_1$ to $\xi_2$ if and only if $\xi_2=\xi_1^{\ell}$.
\end{subDefinition}

% we 
% introduce t: vertices are the element of  $R_H=\mu_N^\times(\overline{\F}_p)/\det(H)$,  (In the Borel level structure case, this graph is just one vertex with no edges.)

Since $\mu_N^\times(\overline{\F}_p)$ is a principal homogeneous space for the right action of $(\Z/N\Z)^{\times}$, the graph $C(N,\det(H), \ell)$ has simple structure: it is the disjoint union of $n$ cycles $C_1,\ldots C_n$, each having the form of a loop:
% \[
% C_i \cong \quad v_1\to v_2\to \cdots \to v_k \to v_1 \,, 
% \]

\begin{center}
\begin{tikzpicture} 
        \node[draw=none, left] at (-1.7*2,0) {$C_i \cong$};
    \coordinate (1) at ( -1.7*1, 1.7*0 );
    \coordinate (2) at ( -1.7*0.866, 1.7*0.5 );
    \coordinate (3) at ( -1.7*0.5, 1.7*0.866 );
    \coordinate (4) at ( -1.7*0, 1.7*1 );
    \coordinate (6) at ( -1.7*0.866, -1.7*0.5 );
     \node[draw=none, left] at (1) {$v_1\,\,$};
     \node[draw=none, left] at (2) {$v_2\,\,$};
     \node[draw=none, left] at (3) {$v_3\,\,$};
     \node[draw=none, left] at (6) {$v_k\,\,$};
     \foreach \i in {1,2,3,4,6}
         \draw[fill=black](\i) circle (0.15em);
    \draw [-Stealth,thick, domain=180:150] plot ({1.7*cos(\x)}, {1.7*sin(\x)});
    \draw [-Stealth,thick, domain=150:120] plot ({1.7*cos(\x)}, {1.7*sin(\x)});
    \draw [-Stealth,thick, domain=120:90] plot ({1.7*cos(\x)}, {1.7*sin(\x)});
    \draw [dashed, -Stealth,domain=90:-150] plot ({1.7*cos(\x)}, {1.7*sin(\x)});
    \draw [-Stealth,thick, domain=210:180] plot ({1.7*cos(\x)}, {1.7*sin(\x)});

    % \draw [blue,thick,domain=180:270] plot ({cos(\x)}, {sin(\x)});
    % \coordinate (1) at (-4/2,0);
    % \coordinate (2) at (-3.30410/2,2/2);
    % \coordinate (3) at (-1.7/2,3.06410/2);
    % \coordinate (4) at (-3.56410/2,-2/2); % \draw[->] (1) to [bend left] (2);
  % \draw[->] (2) to[bend left] (3);
  %  \draw[->, dashed] (3) to[bend left] (5);
  %  \draw[->, dashed] (3) to[bend left] (4);
\end{tikzpicture}
\end{center}
with $k$ the order of $\ell$ in $(\Z/N\Z)^\times/\det(H)$ and $n= \varphi(N)/(k\, |\!\det(H)|)$.
In particular, the adjacency matrix $P_i$ of each $C_i$ is the cyclic permutation matrix on $k$ elements; its spectrum in $\C$ is thus the set $\mu_k(\C)$ of the $k$-th roots of unity in $\C$.

\medskip

If two elliptic curves with level structure are connected by a degree $\ell$ isogeny, then \cite[Chapter III, Proposition 8.2]{Sil} implies that the Weil invariant of the level structures are one the $\ell$-th power of the other, hence we have the following result.

\begin{subProposition}\label{prop:morphism}
The Weil invariant (see Definition \ref{def:Weil_invariant}) of a level structure gives a surjective map of graphs
\begin{subeqn}\label{eq:Weil_graphs}
	w\colon G(p,\ell,H)   \to C(N,\det(H), \ell)  \,.
\end{subeqn}
Moreover, in the language of Definitions \ref{def:aut} and \ref{def:Galois_inv}, for every matrix $g \in \GLmod N$ that normalizes of $H$ we have $w(\langle g\rangle (E,\phi)) = w((E,\phi))^{\det(g)}$ and, denoting $\sigma$ the  Frobenius automorphism of $\overline{\F}_p/\F_p$, we have $w(\sigma(E,\phi)) = w((E,\phi))^p$.
\end{subProposition}
An example of isogeny graph together with the map $w$ is given in Figure \ref{figure}.

Denoting $V(\cdot)$ the set of vertices of a graph, the above map extends to  linear maps 
\begin{subeqn}\label{eq:Weil_push_all}
w_*\colon \Q^{V(G)} \lto \Q^{V(C)}  \,.
\end{subeqn}
Fix a connected components $C_i$ of $C$, let $G_i:=w^{-1}(C_i)$  (this definition of $G_i$ coincides with the one in the Introduction), then the above map  restricts to a morphism
\begin{subeqn}\label{eq:Weil_push}
w_{i,*}\colon \Q^{V(G_i)} \lto \Q^{V(C_i)}  \,.
\end{subeqn}

\begin{subRemark}[Description of $\ker(w_{i,*})$]\label{rem:genratori_nuclei} The kernel of the map $w_{i,*}$ defined in Equation \eqref{eq:Weil_push} will play an important role in this paper, so let us describe it explicitly. %The analogue description holds for $\ker(w_{i,*}^{\Q})$.

If, as in the Borel case, $\det(H)= (\Z/N\Z)^\times$, then $\ker(w_{i,*})$  is the space of linear combination of vertices of the graphs whose coefficient sum up to zero, i.e. the orthogonal complement of the vector $(1,\dots , 1)$ in $\Q^{V(G)}$ for the standard scalar product.

In general,  $\ker(w_{i,*})$ is equal to the subspace of $\Q^{V(G_i)}$ spanned by linear combinations of elements of $V(G_i)$ with the same Weil invariant and such that the coefficients sum up to zero. In other words, calling $V_\xi$ the set of vertices of $G$ with Weil invariant $\xi \in R_H$, then 
\[
\ker(w_*) = \bigoplus_{\xi \in R_H} \left\{ x \in \Q^{V_\xi}:  \sum x_v = 0  \right\} \,, \quad \ker(w_{i,*}) = \bigoplus_{\xi \in V(C_i)} \left\{ x \in \Q^{V_\xi}:  \sum x_v = 0  \right\} 
\]
where we identify $\Q^{V(G)}$ with the direct sum of the various $\Q^{V_\xi}$.

% $$\ker(w_*) = \bigoplus_{\xi \in R_H} \left\{ \sum_{v\in V_\xi} a_v v : \sum a_v = 0  \right\}$$
% and, restricting to $G_i$, we have
% $$\ker(w_{i,*}) = \sum_{\xi \in C_i} \{ \sum_{v\in V_\xi} a_v v : \sum a_v = 0 \}.$$

% \begin{equation}\label{eq:T_dual}
% \begin{aligned}
% 	T^\vee  = \bigoplus_{\xi \in R_H} T_{\xi}^\vee \quad \textrm{with} \quad T_{\xi}^\vee \cong \left\{ x \in \Z^{V_\xi} \,:\, \textstyle \sum_{v\in V_\xi} x_v = 0  \right\} \,.
% \end{aligned}
% \end{equation}
\end{subRemark}

Since $w_i$ is a map of regular graphs, $\ker(w_{i,*})$ is stable for the action of the adjacency matrix $A_i$ of $G_i$.  
\begin{subProposition}\label{prop:spet}
Let $G$ be an isogeny graph as in Definition \ref{def:graph}, $G_i$ be one of its subgraphs defined above, and $A_i$ the adjacency matrix of $G_i$.

The spectrum of $A_i$ over $\C$ is equal to the union of $(\ell+1)\mu_k(\C)$  and the spectrum of $A_i$  restricted to $\ker(w_{i,*})\otimes \C$, where $k$ is as in Theorem \ref{thm:main2} and $\mu_k(\C)$ is the group of $k$-th roots of unity in $\C$.
\end{subProposition}
\begin{proof}

The spectrum of $A_i$ is the union of the spectra of $A_i$ restricted to $\ker(w_{i,*})$ and of $A_i$ when acting on the quotient $\C^{V(G_i)}/\ker(w_{i,*})$. Since $w_i$ is a map of graphs and $G_i$ is $\ell+1$ regular while $C_i$ is $1$ regular, the action of $A_i$ on $\C^{V(G_i)}/\ker(w_{i,*}) \cong \C^{V(C_i)}$  is $\ell+1$ times the adjacency matrix of $C_i$, that is 
% . In particular the restriction of $A_i$ to a suitable complement of $\ker(w_{i,*})$ is is conjugated to the matrix 
$$
\left( \begin{matrix}
	0 & && \\
	\vdots && (\ell{+}1) \Id_{k-1} & \\
	0 & && \\
	(\ell{+}1)  &0 & \cdots &0
\end{matrix}\right)
$$
which  is diagonalizable with spectrum $(\ell{+}1)\mu_{k}(\C)$.
\end{proof}

The study of the spectrum of $A_i$ to $\ker(w_{i,*})$ is rather delicate; Sections \ref{S:modular_curve} and \ref{S:relation} are devoted to the proof of the following result.

\begin{subTheorem}[see Theorem \ref{thm:spet2}]\label{thm:spet1}
Let $G$ be as in Definition \ref{def:graph}, let $G_i$ be one of its subgraphs defined above, with adjacency matrix $A_i$, acting on the kernel $\ker(w_{i,*})$ of the map (\ref{eq:Weil_push}).
Then the absolute values of the eigenvalues of $A_i$ restricted to $\ker(w_{i,*})\otimes \C$ are strictly smaller than $2\sqrt{\ell}$.
%where $k'$ is the smallest integer such that $\ell^{k'}\Id\in H$, and $d$ is the dimension of $\ker(w_{i,*})$ (equivalently, $d$ is the number of vertices of $G_i$ minus the order $k$ of $\ell$ in $(\Z/N\Z)^{\times}/\det(H)$).
\end{subTheorem}

Let us see how this statement implies the other theorems described in the Introduction.

\begin{subCorollary}\label{cor:conn_comp}
With the notation as in Theorem \ref{thm:spet1}, each $G_i$ is connected. If $p,\ell$ and $\det \mathcal N_H$ generate $(\Z/N\Z)^\times$, then all $G_i$'s are isomorphic.
\end{subCorollary}
\begin{proof}
By general graph theory, e.g. the argument from \cite[Proposition 1.1.2]{DSV}, the number of connected component of an $\ell{+}1$ regular graph is the multiplicity of the eigenvalues $\ell{+}1$ for the adjacency matrix, hence Proposition \ref{prop:spet} and Theorem \ref{thm:spet1} implies that $G_i$ is connected. 

For the second part we notice that $p,\ell$ and $\det \mathcal N_H$ generate $(\Z/N\Z)^\times$ if and only if $\langle p, \det \mathcal N_H\rangle$ acts transitively on the set of orbits $\{C_1, \ldots, C_n\}$.
If, for $g$ in $\mathcal N_H$, $\det(g)$ maps $C_i$ to $C_j$, then $\langle g \rangle$ and $\langle g^{-1} \rangle$ give an isomorphism between $G_i$ and $G_j$. Analogously, if  $p$ maps $C_i$ to $C_j$, then $\langle \sigma \rangle$ gives an isomorphism between $G_i$ and $G_j$.
\end{proof}

\begin{figure}[H]
\centering
    \includegraphics[width=16cm]{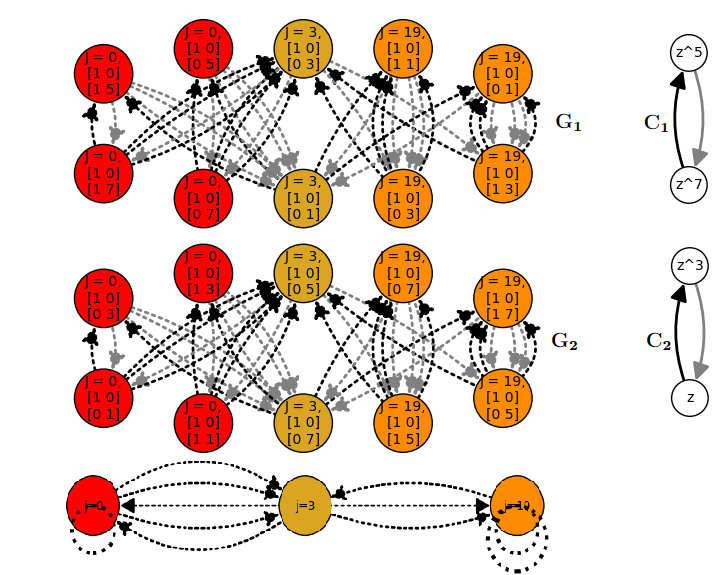}
    \caption{\emph{This is an example of isogeny graph $G(p,\ell,H)$ with $p=23$, $\ell=3$ and $H = \langle \smt5621,\smt1201,\smt7027,\smt5005,\smt2771,\smt1401,\smt1041 \rangle $ the only index 8 subgroup of $\GLmod{8}$.
    \newline
    $\qquad$ The color indicates the elliptic curve, or equivalently the $j$-invariant. The level structure is given through a matrix: for each of the three elliptic curves we have chosen a (non-canonical) basis of the $8$-torsion, and the matrix gives the change of basis. For mere visual clarity, arrows going up are black, arrows going down are gray.
    \newline
    $\qquad$ The graph $G(p,\ell,H)$ has two components, $G_1$ and $G_2$, which correspond via the Weil invariant to the two connected components $C_1$ and $C_2$ of the Cayley graph, depicted on the right.  Since each $C_i$ has two vertices, each $G_i$ is bipartite, see Remark \ref{rem:partition}.  At the bottom, the graph without level structure. 
    }} 
\label{figure}
\end{figure}    

To prove Theorem \ref{thm:main1} and Theorem \ref{thm:main2} we will use the following lemma.

\begin{subLemma}\label{lem:bound_poly}
Let $p(x)\in \Z[x]$ be a monic integral polynomial of degree $d$, 
$\ell$ and $k'$ two positive integers such that for every complex root $\lambda$ of $p(x)$ one has $|\lambda| < 2\sqrt{\ell}$ and the phase of $\lambda$ is in $ \frac{\pi}{k'} \Z$; then $$|\lambda| < 2\sqrt{\ell}-\left(4\sqrt{\ell}\right)^{-2dk'+1} \,.
$$
\end{subLemma}

\begin{proof}

Let $F$ be the subfield of $\C$ obtained adding to $\Q$ all roots of $p(x)$, all $k'$th roots of unity and $\sqrt{\ell}$. The field $F$ is a Galois extension of $\Q$ of degree at most $2dk'$.

Let $\lambda$ be a root of  $p(x)$ of phase $\xi$. Consider the product
$$P:=\prod_{\sigma \in \Gal(F/\Q)}\sigma(2\sqrt{\ell}-\xi^{-1}\lambda) \,.$$
Observe first that $P$ is $\Gal(F/\Q)$-invariant, hence it is in $\Q$. It is a product of algebraic integers, hence it is in $\Z$.
Because of the assumption on the absolute values of the roots of $p(x)$, $P$ is different from 0, hence $|P|\geq 1$.   Furthermore, $|\sigma(2\sqrt{\ell}-\xi^{-1}\lambda)| \le  |\sigma(2\sqrt{\ell})|+|\sigma(\xi^{-1})||\sigma(\lambda)|< 4\sqrt{\ell}$, hence
$$|2\sqrt{\ell}-\xi^{-1}\lambda| = \frac{|P|}{ \prod_{\sigma \in \Gal(F/\Q)\setminus\{1\}} |\sigma(2\sqrt{\ell}-\lambda)| } > \left(4\sqrt\ell\right)^{-|G|+1} \,. $$
\end{proof}

\emph{Proof of Theorems \ref{thm:main1} and \ref{thm:main2}}  The statement about the connected components is Corollary \ref{cor:conn_comp}.  Diagonalizability and the angles of the eigenvalues are given in Proposition \ref{prop:diag}. The eigenvalues of absolute value $\ell{+}1$ are described in Proposition  \ref{prop:spet}. 

To bound the aboslute values of the other eigenvalues, we first apply Theorem \ref{thm:spet1}, and then we apply Lemma \ref{lem:bound_poly} to the characteristic polynomial of the adjacency matrix $A_i$ restricted to $\ker(w_{i,*})\otimes \C$.

\subsection{Isomorphism between Borel and Cartan level structure}\label{S:BorelversusCartan}
Here we define an isomorphism between graphs with Borel and Cartan level structure; because of this isomorphism, in the sequel, when we spell out our results in special cases, we will consider the Borel level structures and not the Cartan ones.

Fix $p$ and $\ell$ distinct primes; let $N$ be a positive integer coprime with $p$ and $\ell$; let $B_0(N^2)$ be the Borel subgroup of $\GLmod{N^2}$ and $T(N)$ the split Cartan of $\GLmod{N}$. 
Consider the map
\begin{equation*}\label{eq:iso_BC}
    \begin{aligned}
F\colon G(p,\ell,B_0(N^2))&\lto G(p,\ell,T(N))\\
(E,C)&\longmapsto (E/NC,C/NC,E[N]/NC) \ .
    \end{aligned}
\end{equation*}

\begin{subProposition}
The map $F$ defined above 
%in Equation \eqref{eq:iso_BC} 
gives an isomorphism of graphs. 
\end{subProposition}
\begin{proof} The map $F$ naturally extends to edges:  given vertices $(E,C)$ and $(E',C')$ in the Borel graph, for each $\ell$-isogeny $\varphi \colon E \to E'$ such that $\varphi(C) = C'$, we also have $\varphi(NC) = NC'$, hence there exists an isogeny  $\tilde \varphi$ of degree $\ell$ that fits in the commutative diagram
$$
\begin{tikzcd}
    E \arrow[d] \arrow[r, "\varphi"] & E'\arrow[d]
    \\
    E/NC \arrow[r, "\tilde \varphi"]  & E'/NC 
\end{tikzcd} 
$$
In particular, the isogeny $\tilde \varphi$ has degree $\ell$ and sends the subgroups $C/NC$, $E[N]/NC$ respectively to $C'/NC'$, $ E'[N]/NC'$. 
The inverse of $F$, at the level of vertices, is the map sending $(E, C_1, C_2)$ to $(E/C_2, (N^{-1} C_1)/C_2)$, where $N^{-1}C_1$ is the set of points $P\in E[N^2]$ such that $NP$ lies in $C_1$.

\end{proof}

\section{Preliminary results on modular curves}\label{S:modular_curve}
In this section we collect some facts about the compactified moduli space of elliptic curves with level structure, and Hecke operators. Our main references is \cite{DeRa}, as it treats these concepts in the generality that we need. There are however many other references about these classical topics.

We need the construction of the moduli space as stack: first we need the definition of (generalized) elliptic curve over an arbitrary scheme $S$; then, as usual, the moduli space will be the space $M$ such that maps from $S$ to $M$ are equivalent to (generalized) elliptic curves over $S$. A modern reference about stacks is \cite{Alper}, it contains also an introduction with a high level description of moduli spaces and stacks.

Following \cite[Chapter II page 173]{DeRa}, a generalized elliptic curve is either an elliptic curve or a N\'eron polygon.  A N\'eron polygon is a curve obtained taking $n$ copies of $\P^1$ indexed by $\Z/n$, and gluing the point $0$ of the $i$-th copy, with the point $\infty$ of the $i+1$- copy; its smooth locus has a natural group structure.

A generalized elliptic curves $\pi\colon E\to S$ over an arbitrary scheme $S$ is a family of genus one curves, whose geometric fibers are either elliptic curves or N\'eron polygons, and with a group structure on the smooth locus \cite[Definition 1.12 page 36]{DeRa}.

Fix a positive integer $N$ and a subgroup $H$ of $\GLmod{N}$.
A level $H$ structure on an elliptic curve or on a N\'eron polygon with $N$ edges $E$ is an isomorphism $\phi$ between the $N$ torsion of $E$ and $(\Z/N\Z)^{\oplus 2}$; two level structures $\phi_1$ and $\phi_2$ are isomorphic if  there exists an $h$ in $H$ such that $\phi_1=\phi_2 \circ h$. Observe that to have such a level structure the characteristic of the base field can not divide $N$.

A level $H$ structure on a generalized elliptic curves $\pi\colon E\to S$ over an arbitrary scheme $S$ is  an isomorphism $\phi$ of the $N$ torsion of $E$ with $(\Z/N\Z)^{\oplus 2}_S$; two level structures $\phi_1$ and $\phi_2$ are isomorphic if \'etale locally on $S$ there exists an $h$ in $H$ such that $\phi_1=\phi_2 \circ h$.

%Given a  positive integer $N$ and a subgroup $H$ of $\GLmod{N}$ we denote  $\calM_H$ the stack over $\Z[1/N]$ parametrizing generalized elliptic curves with level $H$ structure, namely  generalized elliptic curves $\pi \colon E \to S$ such that the fibers are either smooth or N\'eron polygons with $N$ edges, together with an isomorphism $\phi$ of the $N$ torsion of $E$ with $(\Z/N\Z)^{\oplus 2}_S$; two level structures $\phi_1$ and $\phi_2$ are isomorphic if \'etale locally on $S$ there exists an $h$ in $H$ such that $\phi_1=\phi_2 \circ h$. 

A key result of \cite{DeRa} is that the moduli stack $\calM_H$ parameterizing generalized elliptic curve with level $H$ structure is a proper and smooth Deligne-Mumford stack over $\Z[1/N]$, see \cite[Section IV.3, and Theorem 3.4]{DeRa}. The moduli space of elliptic curves is an open substack of $\calM_H$; in particular, it is not proper.

For the proofs of our results, we need a more general definition of level structure that works also when the characteristic divides the level. The most general notion is the one of Drinfeld level structure, see \cite{KM}; in this paper we will only need a generalization of Borel level structure, already discussed in \cite{DeRa}, which we recall below.

 %For any generalized elliptic curve $E/S$, a $B_0(k)$-structure is a cyclic locally free group subscheme $G\subset E$ of rank $k$, which intersects every irreducible components of each geometric fiber of $E\to S$. (This makes sense for any $S$, if $k$ is invertible in $S$ we recover the previous definition, and if $k=p$ is prime, then the kernel of Frobenius is a ``new'' example of $B_0(p)$-structure for an elliptic curve/$\ol{F_p}$.)

For every positive integer $k$, let $B_0(k) = \{ \smt *0**\}$ be the standard Borel subgroup of $\GLmod{k}$. Given $M=Nq_1\cdots q_r$, with $q_i$ prime powers that are pairwise coprime and prime to $N$, we will consider level structures associated with subgroups $K$ of $\GLmod{M}$ of the form 

\begin{equation}\label{eq_K}
K = H \times B_0(q_1) \times \cdots \times B_0(q_r)      \quad < \quad    \GLmod{M} = \GLmod{N}\times \prod_{i=1}^r\GLmod{q_i}\,,
\end{equation}
for $H$ a subgroup of $\GLmod{N}$. When $r=1$ and $q_1=p$ is prime, we write
\begin{equation}\label{eq_Hp}
	H_p := H \times B_0(p)  \quad < \GLmod{N}\times \GLmod{p} = \GLmod{Np}\,.
\end{equation}

For these kind of subgroups, a level $K$ structure on a generalized elliptic curve $\pi \colon E\to S$ is the datum of a level $H$ structure $\phi$, and, for each $q_i$, a cyclic locally free subgroup scheme $C_i$ of rank $q_i$ such that the subgroup generated by $C_i$ and the image of $\phi$ intersects every irreducible component of every geometric fiber of $\pi$ (see \cite[1.4.1, page 100]{DeRa}). 

Since a Borel subgroup $B_0(q)$ is the stabilizer of a line in $(\Z/q\Z)^2$, we observe that over $\Z[1/M]$ this second definition of level $K$ structure is equivalent to the previous one, i.e. to an isomorphism between the $M$-torsion and $(\Z/M\Z)^{\oplus 2}$ up to the action of $K$. On the other hand, if, for example, $q_i$ is prime, over $\ol{\F}_{q_i}$ the kernel of the Frobenius is a ``new'' example of $B_0(q_i)$-structure for an elliptic curve which does not fit in the previous definition.

The stack $\calM_K$ parametrizes generalized elliptic curves with level $K$ structure such that the N\'eron polygons have only $M$ edges. It is a proper and regular Deligne-Mumford stack over $\Z[1/N]$, it is smooth outside the $\overline{\F}_{q_i}$ points parametrizing supersingular elliptic curves, see \cite[Chapter V, Theorem 1.6, Proposition 1.10, Variants 1.14 and 1.20]{DeRa}.

For $d$ in $(\Z/N\Z)^{\times}$, we define an automorphism $\langle d \rangle$  of $\calM_K$ called diamond operator. A point of $\calM_K$ is a generalized elliptic curve with level $K$ structure $(E,\phi, C_1,\dots , C_r)$; we let\footnote{This definition of the diamond operator generalizes the one from \cite[Section 7.9]{DS}, which is given just for torsion point level structures. It  differs slightly from the one given in \cite[Appendix A]{RSDZB} for full level structures. Our definition applies to arbitrary level structure, and it is compatible with the Eichler-Shimura relation \ref{Eichler-Shimura}.} 
\begin{equation}\label{diamond}
\langle d \rangle (E,\phi, C_1,\dots , C_r):= (E,d\phi, C_1,\dots , C_r) \, ,
\end{equation}
where $d\phi$ is the level $H$ structure obtained composing $\phi$ with the multiplication by $d$. This definition makes sense for curves over an arbitrary scheme $S$, so gives an automorphism of $\calM_K$. The inverse of $\langle d \rangle$ is $\langle e \rangle$, where $e$ is a multiplicative inverse of $d$ modulo ~$N$.

We now introduce two key maps, that will play more than one role for us
\begin{equation}\label{eq:maps_Tl}
\begin{aligned}
	&\pr_p\colon \calM_{H_p} \to  \calM_H\,, \quad & \pr_p (E\to S, \phi, C) = (E\to S, \phi)\,, \\
	&\quot_p\colon \calM_{H_p} \to  \calM_H \quad & \quot_p (E\to S, \phi, C) = (E/C\to S, \pi_C\circ \phi)\,,
\end{aligned}
\end{equation}
where $\pi_C$ is the quotient map $E \to E/C$.

Following \cite[Section V]{DeRa}, we first use them to study the fiber $\calM_{H_p,\F_p} = \calM_{H_p}\times \Spec \F_p$. The maps $\pr_p$ and $\quot_p$ have right inverses when restricted to  $\calM_{H_p,\F_p}$. Indeed, an elliptic curve $E$ over $\overline{\F}_p$ has only two subgroup or rank $p$: the kernel of the Frobenius and the kernel of the Verschiebung. Recall that, by definition, the Verschiebung is the dual isogeny of the Frobenius; we denote it by $\Ver$. They are equal if and only if the curve is supersingular. % Denoting by $\pr_{p,p}$ (resp. $\quot_{p,p}$) the restriction of $\pr_p$ (resp. $\quot_p$) to $\calM_{H_p,\F_p}$,  w
We obtain two morphisms
\begin{equation}\label{eq:pieces_B0p}
\begin{aligned}
	\pr_{p,p}^{-1}\colon \calM_{H,\F_p} &\lto \calM_{H_p,\F_p}\,, \qquad   (E/S/\F_p, \phi) \mapsto (E/S/\F_p, \phi, \ker(\Frob))\,,
	\\ \quot_{p,p}^{-1}\colon \calM_{H,\F_p} &\lto \calM_{H_p,\F_p}\,, \qquad  (E/S/\F_p, \phi) \mapsto (E^{(p)}/S/\F_p, \phi\circ (\cdot \tfrac{1}{p}) \circ\Frob ,\ker(\Ver))\,,
\end{aligned}
\end{equation}
which provide a description of $\calM_{H_p,\F_p}$ as the union of two copies of $\calM_{H,\F_p}$ nodally attached at the supersingular elliptic curves, see \cite[Section 5, Theorem 1.16 and Variant 1.18]{DeRa}. Here we apologize for an abuse of notations: $\pr_{p,p}^{-1}$ and $\quot_{p,p}^{-1}$ are not the inverse of $\pr_{p,p}= \pr_{p,\F_p}$ and $\quot_{p,p} = \quot_{p,\F_p}$, but just the right inverse.

\medskip

Every Deligne-Mumford stack $\mathcal{M}$ admits a coarse space $M$, in particular $\mathcal{M}_K$ has a coarse space $M_K$. Every map between stacks, such as $\pr_p$ and $\quot_p$, induces a map between coarse spaces. A key fact is that in our set-up the formation of the coarse space is compatible with base change. More precisely, let $\ell$ be any prime number not dividing $N$ (possibly it can also be a divisor of the $q_i$'s); the universal property of coarse spaces gives a map from the coarse space of $\mathcal{M}_{K,\mathbb{F}_{\ell}}$ to $M_{K,\mathbb{F}_{\ell}}:=M_K\times \mathbb{F}_{\ell}$. In \cite[Cor 6.10 page 145]{DeRa} it is shown that this map is an isomorphism (observe that if $\ell$ divides $N$ then this compatibility is not known for general $H$, see for instance \cite[Section 8.5]{KM}).

The definition of Picard group generalizes in a standard way from varieties to Deligne-Mumford stack, see e.g. \cite[Section 4.1.7]{Alper}.
We use the Picard group fo $\calM_K$ and the maps (\ref{eq:maps_Tl}) to define the \emph{Hecke operator} $T_{\ell}$. 

\begin{Definition}[Hecke operators]\label{def:Hecke}

With $K$ as in Equation \eqref{eq_K}, and for a prime $\ell$ which does not divide $M$, the Hecke operator $T_\ell$ is the map
\[
T_\ell := (\quot_\ell)_* \circ \pr_\ell^* \colon \Pic(\calM_{K}/\Z[1/N])  \to \Pic(\calM_{K}/\Z[1/N])\,,
\]
where the push-forward is a cycle push-forward.	 

The analogue definition works for the coarse space $M_K$.
\end{Definition}

Observe that the diamond operator $\langle d \rangle$, which is defined for every $d$ which does not divide $N$, commutes with $\pr_{\ell}$, $\quot_{\ell}$ and $T_{\ell}$.

The moduli space $\calM_{K, \F_\ell}$ is one dimensional, and we have the following celebrated description of the  restriction of the Hecke operator $T_{\ell}$ to the Jacobian $\Pic^0(\calM_{K,\F_\ell})$ of $\calM_{K, \F_\ell}$.

\begin{Theorem}[Eichler-Shimura relation] \label{Eichler-Shimura}
With the notations of Definition \ref{def:Hecke}, 
denoting by $T_{\ell,\F_\ell}$ the restriction of $T_{\ell}$ to either $\Pic^0(\calM_{K,\F_\ell})$ or $\Pic^0(M_{K,\F_\ell})$, we have
\[
T_{\ell,\F_\ell} = \Frob_* + \langle \ell\rangle_* \Frob^*
\]
where $\langle \ell\rangle$ is the diamond automorphism \eqref{diamond} and $\Frob$ is the Frobenius %and the Vershiebung 
of the curve $\calM_{K,\F_\ell}$ or $M_{K,\F_\ell}$.
\end{Theorem}

\begin{proof}
We first prove the result on the stacks. Looking at the description of $\quot_{\ell,\F_\ell}$ and $\pr_{\ell,\F_\ell}$ on the two irreducible components of $\calM_{K_{\ell}}$, we can write
$$
T_{\ell,\F_\ell}=(\quot_{\ell,\F_\ell}\circ \pr_{\ell,\F_\ell}^{-1})_*\circ (\pr_{\ell,\F_\ell}\circ \pr_{\ell,\F_\ell}^{-1})^*+(\quot_{\ell,\F_\ell}\circ \quot_{\ell,\F_\ell}^{-1})_*\circ (\pr_{\ell,\F_\ell}\circ \quot_{\ell,\F_\ell}^{-1})^*
$$
Both $\pr_{\ell,\F_\ell}\circ \pr_{\ell,\F_\ell}^{-1}$ and $\quot_{\ell,\F_\ell}\circ \quot_{\ell,\F_\ell}^{-1}$ are the identity on $\Pic^0\calM_{K,\F_\ell}$, so we are left with 
$$
T_{\ell,\F_\ell}=(\quot_{\ell,\F_\ell}\circ \pr_{\ell,\ell}^{-1})_*+(\pr_{\ell,\F_\ell}\circ \quot_{\ell,\ell}^{-1})^*
$$
We observe that $(\quot_{\ell,\F_\ell}\circ \pr_{\ell,\ell}^{-1})_*=\Frob_*$ because it maps $(E,\phi)$ to $(E^{(\ell)},\Frob \circ \phi)$ . To conclude, $(\pr_{\ell,\F_\ell}\circ \quot_{\ell,\ell}^{-1})^*=\langle\ell\rangle_* \Frob^*$ because it maps $(E,\phi)$ to $(E^{(\ell)},\Frob\circ \phi \circ \left( \cdot\frac{1}{\ell}\right))$.

The property on the coarse spaces follows from their universal property.

%We now focus on $M_K\times \F_{\ell}$. If $\ell \geq 5$ the result follows directly from Theorem \ref{thm:base_change}. Otherwise focus on the map $B_{\ell}$ coming from Equation \eqref{eq:base_change_map}. In our case, it is a map between smooth curves \Giulio{da controllare bene se e' suriettiva} which induces an isogeny 
%$$
%B_\ell^*\colon \Pic^0(M_K\times \F_\ell) \to \Pic^0(M_{K,\F_\ell})
%$$
%which is equivariant for the algebraic group homomorphism $F:=T_{\ell,\F_\ell} - \Frob_* - \langle \ell\rangle_* \Frob^*$. The image of $F$ on the domain, if non-trivial, it is positive dimensional. Since $B_\ell^*$ is an isogeny, we deduce that if $F$ it is not trivial on the domain, then it is not trivial on the codomain, but we already know that $F$ vanishes on the co-domain.

\end{proof}

The spectral bounds in Theorem \ref{thm:spet1} will eventually be a consequence of the following bound, which in turn is a consequence of the above mentioned Eichler-Schimura relation and Weil's conjectures. A reference for the \'{e}tale cohomology used below is \cite{MilnEt}. 
For Weil's conjectures we refer to \cite{Del}.

\begin{Theorem}[Bound on the eigenvalues of the Hecke operator]\label{thm:bound}
With the above notations, let $\ell,\ell'$ be different primes not diving $M$, then the roots of the characteristic polynomial of the action $T_{\ell}$ on  $H^{i,\text{\'et}}(\Pic^0(M_{K, \F_{\ell}}),\Q_{\ell'})$ have complex absolute value less than or equal to $2\ell^{i/2}$.
\end{Theorem}
\begin{proof}
The curve $M_{K, \F_\ell}$ is proper and smooth, hence $X:=\Pic^0(M_{K, \F_\ell})$ is an abelian variety defined over $\F_{\ell}$. Weil's conjectures, proved by Deligne  \cite[Theoreme 1.6]{Del}, implies that the roots of the characteristic polynomial of the action $\Frob_X$, which is the Frobenius of $X$, on  $H^{i,\text{\'et}}(X,\Q_{\ell'})$ have complex absolute value $\ell^{i/2}$ (in loc. cit. Deligne uses the term variety to denote also possibly non-irreducible reduced schemes).

The Frobenius $\Frob_X$ is the endomorphism $\Frob_*$ appearing in Theorem \ref{Eichler-Shimura}.
The maps $\Frob_*$ and $\Frob^*$ commutes, $\Frob^*\circ \Frob_*$ is the multiplication by $\ell$, hence also $\Frob_*$ has eigenvalues of complex absolute value $\ell^{i/2}$.
The map $\langle \ell \rangle$ is an automorphism of finite order of $X$, hence its eigenvalues are roots of unity.

Since the maps $\Frob_*$, $\Frob^*$ and $\langle \ell \rangle$ commute,  the claim follows from Theorem \ref{Eichler-Shimura}. 
\end{proof}

\
We close this section by introducing some automorphisms of modular curves. They will be related to the automorphisms of isogeny graphs from Section \ref{s:aut_G}, see Theorem \ref{thm:riassunto_confronti}.

The following automorphisms will correspond to the automorphisms from Section \ref{s:aut_G} with the same name.

\begin{Definition}[Matricial automorphisms]\label{def:aut_modular}
	Given a level structure $K= H\times \prod B_0(q_i)$ as in \eqref{eq_K}, for any element $g$ in the normalizer $\mathcal N_H<\GLmod{N}$ of $H$, the automorphism $\langle g\rangle\colon \calM_{K}\to \calM_{K}$ maps a curve $(E,\phi, C_1, \ldots C_r)$ to $(E, \phi\circ g, C_1, \ldots C_r)$. \newline
	In particular, for every $d$ in $(\Z/N\Z)^{\times}$, the diamond operator $\langle d \rangle $ in \eqref{diamond} is the automorphism associated with the diagonal matrix $ \smt d\,\,d$.
\end{Definition}

\begin{Definition}[Atkin-Lehner automorphisms]\label{def:AL}
Given $K= H\times \prod B_0(q_i)$ as in \eqref{eq_K}, each $q_i$ yields the Atkin-Lehner map
\begin{equation}\label{eq_AL}
w_{q_i}\colon \calM_K\to\calM_K\,,\quad (E,\phi, C_1, \ldots C_r) \mapsto \left(E/C_i, \pi_i\circ \phi, \pi_i(C_1), \ldots , E[q_i]/C_i, \ldots, \pi_i (C_r)\right) 
\end{equation}
where $\pi_i \colon E \to E/C_i$ is the projection.
%Given an isogeny graph of the form $G= G(p,\ell,K)$, its vertices are tuples $(E,\phi,C_1, \ldots, C_r)$,  and formula \eqref{eq_AL} defines an automorphism of $G$.
\end{Definition}

For a level structure $H_p$, the Atkin-Lehner automorphism $w_p$ plays a special role, and it is related to the Galois or Frobenius automorphism of the graphs. We will call it Fricke automorphism because it is equal to the classical Fricke involution when $K=B_0(p)$, and to highlight its special role.

\begin{Definition}[Fricke automorphism]\label{def:fricke}
For a level structure $H_p$, the Fricke automorphism $\sigma \colon \calM_{H_p}\to \calM_{H_p}$ maps a curve $(E,\phi, C)$ to $(E/C,\pi \circ \phi, E[p]/C)$, where $\pi\colon E \to E/C$ is the projection.
\end{Definition}

\section{Relation between modular curves and isogeny graphs and proof of Theorem \ref{thm:spet1}}\label{S:relation}

In this section we explain the relation between the isogeny graph, together with its adjacency matrix, and the coarse moduli space $M_{H_p,\mathbb{F}_p}$, together with the Hecke operator $T_{\ell}$ (see Remark \ref{rem:analogue_stacks} for the analysis on the stack).

We fix $p,N,H$ as in Definition \ref{def:graph}: $p$ is a prime number, $N$ is an integer not divisbile by $p$, and $H$ is a subgroup of $GL_2(\Z/N\Z)$. The maps \eqref{eq:maps_Tl} give the desingularization
\begin{equation}\label{eq:res}    
\pr_{p,p}^{-1} \sqcup \quot_{p,p}^{-1}\colon M_{H,\F_p} \sqcup M_{H,\F_p} \to M_{H_p,\F_p}\,.
\end{equation}

Since the singularities of $M_{H_p,\F_p}$ are nodal, the pull-back induces an exact sequence 
\begin{equation}\label{eq:T_1}
0 \to T \to \Pic^0\left(  M_{H_p,\F_p}\right) \lto \Pic^0\left( M_{H,\F_p}\right)^{ 2} \to 0
\end{equation}
with $T$ the toric part of the semi-abelian variety $\Pic^0\left( M_{H_p,\F_p} \right)$.

To analyze $T$, we need first to count the connected components of $M_{H,\overline{\F}_p}$. To this end, recall that the Weil invariant of a level structure, see Definition \ref{def:Weil_invariant}, gives a morphism 
$$
w\colon M_{H}  \to  \Spec \left(\mathbb{Z}[\tfrac 1N, \zeta_N]^{\det(H)}\right)
$$
where $\zeta_N$ is a primitive $N$-th root of the unity, see \cite[Chapter 3, Subsection 3.20]{DeRa}, and the exponentiation to $\det(H)$ means that we take invariants of $\det(H) \subset (\Z/N\Z)^\times = \Gal(\Q(\zeta_N)/\Q)$. If we base change to a field of characteristic prime to $N$, the fibers of $w$ are irreducible, see \cite[Chapter 3, Corollary 5.6]{DeRa}. In particular, there is a bijection between the connected components of $\calM_{H,\overline{\F}_p}$  and  $R_H=\mu_N^\times(\overline{\F}_p)/\det(H)$, see Definition \ref{def:Weil_invariant} and above.

Call these components $M_{\xi}$, for $\xi$ in $R_H$.
The discussion below Equation \eqref{eq:pieces_B0p} implies that the map $\pr_{p,\F_p}$ is surjective and gives a bijection between the connected components of $M_{H,\overline{\F}_p}$ and the ones of $M_{H_p,\overline{\F}_p}$.

By definition, points on $T$ correspond to line bundles $L$ over $M_{H_p}$ such that both $\left(\pr_{p,p}^{-1}\right)^* L$ and $\left(\quot_{p,p}^{-1}\right)^*L$ are trivial. As recalled in Appendix \ref{S:nodal}, to describe such an $L$ we need to give a scalar for each node of $ M_{H_p,\F_p}$, modulo a diagonal action of $\Gm$ for every connected component $M_{\xi}$. Recall that the nodes of $ M_{H_p,\F_p}$ are the points representing supersingular curves.  Call $V_\xi$ the set of vertices of $G=G(p,\ell,H)$ with Weil invariant $\xi$, which are in turn the points of $M_{\xi}$ such that $\pr_{p,p}^{-1}(v)$ is singular in $M_{H_p,\overline{\F}_p}$. 
With this notation we have a canonical isomorphism
\begin{equation}\label{eq:T}
T\cong\,\, \prod_{\xi \in R_H} T_{\xi} \quad \textrm{with} \quad T_{\xi}:= \Gm^{V_\xi}/\Gm \,.
\end{equation}

For the groups of characters $T^\vee  := \Hom(T,\Gm)$ and $T_\xi^\vee := \Hom(T_\xi,\Gm)$,  we obtain
\begin{equation}\label{eq:T_dual}
\begin{aligned}
	T^\vee  = \bigoplus_{\xi \in R_H} T_{\xi}^\vee \quad \textrm{with} \quad T_{\xi}^\vee \cong \left\{ x \in \Z^{V_\xi} \,:\, \textstyle \sum_{v\in V_\xi} x_v = 0  \right\} \,.
\end{aligned}
\end{equation}

This identifies $T^\vee$ with a submodule of $\Z^V$ and, comparing with Remark \ref{rem:genratori_nuclei}, we have a canonical isomorphism $T^\vee_\xi \otimes \Q = \ker(w_*)$. Moreover, considering the decomposition $R_H = C_1 \sqcup \ldots \sqcup C_n$  of $R_H$ into the orbits of $\xi \to \xi^\ell$, as in the discussion below Proposition \ref{prop:morphism}, we have for each $C_i$ a canonical isomorphism
\begin{equation}\label{eq:F_i}
\bigoplus_{\xi \in C_i} T^{\vee}_{\xi}\otimes \Q= \ker(w_{i,*}) \,,
\end{equation}
where $\ker(w_{i,*})$ is the subspace of $\Q^V$ described in Remark \ref{rem:genratori_nuclei}.

\begin{Theorem}\label{thm:comparison}

Let $G = G(p,\ell,H)$ be the graph in Definition \ref{def:graph}, with $G_i$ the subgraphs defined above Theorem \ref{thm:main2}, and let $T = \prod_{\xi \in R_H} T_\xi$ be the maximal torus of $\Pic^0(M_{H_p,\overline{\F}_p})$ , as in Equations \eqref{eq:T_1} and \eqref{eq:T}. 

For each $i$, the isomorphism \eqref{eq:F_i} conjugates the action of the Hecke operator $T_{\ell}$ to the adjoint action of the adjency matrix of the graph $G_i$: i.e. the following diagram is commutative
$$
\begin{tikzcd}
	\displaystyle\bigoplus_{\xi \in C_i} T^{\vee}_{\xi}\otimes \Q  \ar[d,equals]  \arrow[rr, "T_\ell"]
 && 
	\displaystyle \bigoplus_{\xi \in C_i} T^{\vee}_{\xi}\otimes \Q  \ar[d,equals] 
	\\ \ker(w_{i,*}) \arrow[rr, "A_i^*"] && \ker(w_{i,*}) 
\end{tikzcd}
$$

where $\ker(w_{i,*})$ is the subspace of $\Q^V$ described in Remark \ref{rem:genratori_nuclei}, and $A^*$ is the  adjoint of the adjacency matrix $A$ with respect to the Hermitian form \eqref{eq:adj}, see also Proposition \ref{prop:diag}. 

\end{Theorem}

\begin{proof}
Let $V$ be the set of vertices of $G$. Equation \eqref{eq:T_dual} gives an embedding of $T^\vee$ and  $\bigoplus_{\xi \in C_i} T^{\vee}_{\xi}$ inside $\Z^V$.

In the Appendix \ref{S:nodal} we described the Picard group of a nodal curve, and how morphism such as the Hecke operator act on it (it is a rather general theory, and this is why we have discussed it in an appendix). In particular, Proposition \ref{prop:correspondence_T} tells us that $T_\ell\colon T^\vee \to T^\vee$ (and in particular also its restriction to $\bigoplus_{\xi \in C_i} T^{\vee}_{\xi}$) extends to a map $T_\ell \colon \Z^V \to \Z^V$. It is enough to prove the commutativity of the diagram 
$$
\begin{tikzcd}
	\Z^V \otimes\Q  \arrow[d,"\sim" labl]  \ar[rr, "T_\ell \otimes \Q"]&& 
	\Z^V \otimes \Q  \arrow[d,"\sim" labl] 
	\\ \Q^V \arrow[rr, "A^*"] && \Q^V
\end{tikzcd}\,.
$$
In particular, it is enough checking the commutativity on the elements $(E_i,\phi_i)$ of the canonical basis of $\Z^V$. 

On $M_{H_p}$ we have $T_\ell = (\quot_\ell)_* \circ \pr_\ell^*$. Writing $M_{H_p,\F_p}$ as the union of two copies of $M_{H,\F_p}$, Proposition \ref{prop:correspondence_T} allows us to describe $T_\ell$ on $\Z^V$ only looking at what is happening on one of the two copies of $M_{H,\F_p}$. In particular,  for each supersingular point $(E_i,\phi_i)$ on $M_{H}(\overline{\F}_p)$ we have
\begin{equation}\label{eq:Tl_chars}
	\begin{aligned}
		T_{\ell}(E_i,\phi_i) = \!\!\sum_{(E_j,\phi_j,C)} \!\! \mathrm{ord}_{(E_j,\phi_j,C)}(\quot_\ell) {\cdot} \pr_{\ell}(E_j,\phi_j,C)   = \!\!\sum_{(E_j,\phi_j,C) } \!\! \mathrm{ord}_{(E_j,\phi_j,C)}(\quot_\ell) \cdot (E_j,\phi_j)\,,
	\end{aligned}
\end{equation}
where $(E_j,\phi_j,C)$ varies in the fiber $\quot_\ell^{-1}(E_i,\phi_i) \subset M_{H_\ell}(\overline{\F}_p)$.

To compute the orders $\ord(\quot_\ell) $ we start by noticing that when $H$ structures are rigid (i.e. when $\Aut(E,\phi) = \{ 1\}$ for each $(E,\phi)$ in $M_H(\overline{\mathbb{F}}_p)$), then $\ord (\quot_\ell)=1$: indeed $\quot_\ell$ has degree $\ell{+}1$ and duality of isogenies gives a bijection between the set of points $(E_j,\phi_j,C)\in \quot_{\ell}^{-1}(E_i,\phi_i)$ and the set of points $(E_i,\tfrac 1\ell \phi_i,C)\in M_{H_\ell}(\overline{\F}_p)$ which has cardinality $\ell{+}1$ because $\Aut(E_i, \phi_i)$ is trivial, hence for different subgroups $C_1,C_2 \subset E_i[\ell]$ the triples $(E_i,\tfrac 1\ell \phi_i,C_1)$ and $(E_i,\tfrac 1\ell \phi_i,C_2)$ are not isomorphic.

For general $H$ structure, even not rigid, write $M_{H,\F_p}=M_{K,\F_p}/G$ for $K$ a rigid level structure and $G$ a finite group, with quotient map $\pi_G$ (for example take $K$ to be full-level structures of level $3N$, see \cite[Corollary 4.7.2]{KM}, and $G<\GLmod{3N}$ to be the inverse image of $H$ under reduction modulo $N$). Analogously we have $M_{H_\ell}=M_{K_\ell}/G$, with quotient map $\pi_{G,\ell}$. Now, given $(E_j,\phi_j,C)$ supersingular point on $M_{H_{\ell}}$, we can lift it to a point $(E_j,\psi_j,C)$ on $M_{K_{\ell}}$, and, using the commutation $\quot_\ell\circ\pi_{G,\ell} = \pi_G\circ\quot_\ell$, we compute
$$ 
\begin{aligned}
	\ord_{(E_j,\phi_j,C)} \quot_{\ell} & = \frac{\ord_{(E_j,\psi_j,C)} (\quot_\ell \circ \pi_{G,\ell})}{\ord_{(E_j,\psi_j,C)} \pi_{G,\ell}} = \frac{\ord_{(E_j,\psi_j,C)} ( \pi_{G}\circ \quot_\ell)}{\ord_{(E_j,\psi_j,C)} \pi_{G,\ell}} 
	\\ & =  \frac{\ord_{(E_j,\psi_j,C)} (\quot_\ell) \cdot \ord_{(E_i,\psi_i)}\pi_G }{\ord_{(E_j,\psi_j,C)} \pi_{G,\ell}} = \frac{1 \cdot |\Aut (E_i,\phi_i)|}{|\Aut(E_j,\phi_j,C)|} \,.    
\end{aligned}
$$
Substituting in Equation \eqref{eq:Tl_chars}, and using the definition of $a_i$ in \eqref{eq:Hermitian}, we get
\begin{equation}\label{eq:boh}
	T_{\ell}(E_i,\phi_i)  = \sum_{(E_j,\phi_j,C) }  \frac{ |\Aut (E_i,\phi_i)|}{|\Aut(E_j,\phi_j,C)|} \cdot (E_j,\phi_j) = a_i \sum_{(E_j,\phi_j,C) }  |\Aut(E_j,\phi_j,C)|^{-1} \cdot (E_j,\phi_j)\,,
\end{equation}
where the sums run over the isomorphism classes of triples $(E_j,\phi_j,C)\in M_{H_\ell}(\overline{\F}_p)$ such that  $\quot_\ell(E_j,\phi_j,C):= (E_j/C,\pi_C\circ\phi_j)$ is isomorphic to $(E_i,\phi_i)$. We want to compare the last term of Equation \eqref{eq:boh} with the description of $A^*$ given in Remark \ref{rkm_A_star}. 

Observe that $(E_j,\phi_j)$ appears in the right hand side of \eqref{eq:boh} if and only if there is an arrow $(E_j,\phi_j) \to (E_i,\phi_i)$. The number of such arrows equals the number of nontrivial subgroups $C\subset E_i[\ell]$ such that $ (E_j/C,\pi_C\circ\phi_j)\cong (E_i,\phi_i)$. Two triples $(E_j,\phi_j,C_1)$ and $(E_j,\phi_j,C_2)$ give the same element of $M_{H_\ell}(\overline{\F}_p)$ if and only if there exist $\sigma$ in $\Aut(E_j,\phi_j)/\Aut(E_j,\phi_j,C_1)$ such that $\sigma(C_1) = C_2$. Such $\sigma$, if it exists, is unique because we quotiented out exactly by the stabilizer of $(E_j,\phi_j,C_1)$ in $\Aut(E_j,\phi_j)$. We conclude that the coefficient of $(E_j,\phi_j)$ in the right hand side of Equation \eqref{eq:boh} is  
\[
a_i \sum_{\substack{0 \subsetneq C\subsetneq E_j[\ell] \text{ s.t.} \\ (E_j/C,\pi_C\circ\phi_j)\cong (E_i,\phi_i)}} |\Aut(E_j,\phi_j)/\Aut(E_j,\phi_j,C)|^{-1}  |\Aut(E_j,\phi_j,C)|^{-1}\,.
\]
As in Remark \ref{rkm_A_star} we have $a_i=|\Aut(E_j,\phi_j)|$, we have the claim.

\end{proof}

 The automorphisms of modular curves act on the Picard groups, and hence on $T^{\vee}$, via pull-back. The following proposition explains how the canonical isomorphism \eqref{eq:F_i} %the $F_i$'s 
relates the automorphisms of the isogeny graph and the automorphism of the modular curve.
\begin{Proposition}\label{prop:aut}
The canonical isomorphism
\[
 T^\vee  \otimes \Q  = \oplus_{i=1}^n \ker(w_{i,*})\,,   
\]
which is the direct sum over $i$ of the  isomorphisms from Equation \eqref{eq:F_i}, conjugates the Galois map \ref{def:Galois_inv} to the Fricke map \ref{def:fricke}, and the automorphisms from Section \ref{s:aut_G} to the automorphisms from \ref{def:aut_modular} and \ref{def:AL} with the same name.
\end{Proposition}
\begin{proof}
    This is an application of Proposition \ref{prop:correspondence_T} in the case where $G$ is the identity of $M_{H_p,\F_p}$ and $F$ is one of the automorphisms of $M_{H_p,\F_p}$ we have considered. In particular,it is enough checking that the action of matricial automorphisms, respectively Atkin-Lehner automorphisms and Fricke map, on the supersingular points of  $M_{H_p,\F_p}$  is exactly the action of the corresponding automorphisms of the graph. In the first two cases this is straight forward. For the Fricke map $\sigma$, we observe that, given a point $(E,\phi,\ker(\Frob_p))$ of $M_{H_p,\mathbb{F}_p}(\overline{\mathbb{F}}_p)$ representing a supersingular elliptic curve, we have 
    \[
     \sigma \big(E,\phi,\ker(\Frob_p)) = \big(E/\ker(\Frob_p), \pi\circ \phi, E[p]/\ker(\Frob_p)\big) \,.
    \]
    This is equal to $ \big(E^\sigma, \sigma\circ \phi, \ker \Frob_p\big)$ since $E/\ker(\Frob_p)$ is supersingular, hence $E[p]/\ker(\Frob_p)$ must be equal to the kernel of its Frobenius, and the quotient map $\pi\colon E \to E/\ker(\Frob_p)$ is exactly the Frobenius map $\Frob_p \colon E \to E^\sigma$. We conclude that the action of the Fricke map on the points of $M_{H_p,\mathbb{F}_p}$ representing superingular elliptic curves is equal to the Galois action on the corresponding points of the graph.

\end{proof}

\begin{Remark}[Analogous construction on the moduli stack]\label{rem:analogue_stacks}
One could carry out the constructions of this section on the stack $\mathcal{M}_{H_p,\mathbb{F}_p}$ rather than the coarse space $M_{H_p,\mathbb{F}_p}$. Observe that when $p \geq 5$, so the characteristic of the base field does not divide the automorphism group, this stack is a twisted curve, as in \cite[Section 2]{AOV}. Twisted curves are also called stacky curves in the literature. At least in these cases, in loc. cit. is explained how the Picard group is an extension of the Picard group of the coarse space by a finite \'{e}tale group over $\mathbb{F}_p$ related to the automorphism groups. The study of this extension might give further information about isogeny graphs.
\end{Remark}

\begin{Definition}\label{def:A}
Let $p$ be a prime number, $N$ is an integer coprime with $p$, and $H$ is a subgroup of $GL_2(\Z/N\Z)$ as in Definition \ref{def:graph}, %a positive integer $N$, a subgroup $H<\GLmod N$ and a prime $p$ not dividing $N$, 
let $\calA = \calA_{H,p}$ over $\Z[1/N]$ be the connected component of the identity of the kernel of the map
$$
\left(\pr_{p,*}, \quot_{p,*}\right) \colon \Pic^0\left(M_{H_p}\right)  \lto \Pic^0\left(M_{H}\right)\times \Pic^0\left(M_{H}\right)
$$ 
\end{Definition}

The action of the Hecke operator $T_{\ell}$, and the automorphism from Definitions \ref{def:fricke}, \ref{def:aut_modular} and \ref{def:AL} preserve $\calA$, hence we can and do consider their restriction to $\calA$.

\begin{Proposition}
Let $p$ be a prime number, $N$ is an integer coprime with $p$, and $H$ is a subgroup of $GL_2(\Z/N\Z)$ as in Definition \ref{def:graph}. The fiber $\calA_{\F_p}$ is equal to the torus $T$ introduced in Equation \eqref{eq:T_1}.
\end{Proposition}
\begin{proof}
Since $\Pic^0\left(M_{H,\mathbb{F}_p}\right)$ is an abelian variety, and there are no non-trivial map from a torus to an abelian variety, we have the inclusion $T\subseteq \calA_{\F_p}$.

Since $\dim T= \dim \Pic^0\left(M_{H_{p, \ol{F}_p}}\right)-\dim \left(\Pic^0\left(M_{H_{\ol{F}_p}}\right)\times \Pic^0\left(M_{H_{\ol{F}_p}}\right)\right)$, to conclude we have to show that the reduction modulo $p$ of $\left(\pr_{p,*}, \quot_{p,*}\right)$ is surjective. 

We look at the resolution given by Equation \eqref{eq:res} and we consider the map 
\[
\lambda \colon \Pic^0(M_{H,\F_p})^{ 2} \lto \Pic^0(M_{H_p,\F_p})\,,\quad (x,y)\longmapsto (\pr_{p,p}^{-1})_*(x)+ (\quot_{p,p}^{-1})_*(y)\,.
\] 
By the same arguments used in the proof of Theorem \ref{Eichler-Shimura}, (or see also the diagram in \cite[page 145]{DeRa}), we have that 
$\left(\pr_{p,*}, \quot_{p,*}\right)_{\F_p}
\circ \lambda$ equals $\smt{\Id}{\Frob}{\Frob}{\Id}$ as endomorphism of $\Pic^0(M_{H,\F_p})^{ 2}$; this endomorphism is surjective, hence the same is true for $\left(\pr_{p,*}, \quot_{p,*}\right)_{\F_p}$.

\end{proof}

The following key technical lemma uses the theory of N\'eron models, a general reference for N\'eron models is \cite{BLR}. 

\begin{Lemma}\label{lem:Neron_exact}
Fix $p,N,H$ as in Definition \ref{def:graph} and let $\calA = \calA_{H_p}$. Then, for every endomorphism $F$ of $\calA$ and every prime number $q$ not dividing $N$, we have
$$\dim (\mathrm{Im} \left(F|_{\calA_{\C}}\right)) =  \dim \left(\mathrm{Im}\left( F|_{\calA_{\overline{\F}_q}}\right)\right) \,.$$
\end{Lemma}
\begin{proof}
By \cite[Proposition 6.7 and Theorem 6.9, pages 143-145]{DeRa}, both $M_H/\Z[\tfrac 1N]$ and $M_{H_p}/\Z[\tfrac 1N]$ have reduced fibers, and geometrically irreducible generic fiber. Again by loc$.$ cit$.$, $M_H$ is regular, but $M_{H_p}$ might not be: it is smooth away from supersingular elliptic curves $(E,\phi, C)$ in characteristic $p$, and locally around such points it is isomorphic to $\Z_p[[w,z]]/(wz-p^k)$, where $k$ is either $\# Aut(E,\phi,C)$, or half of it if $-1$ is an automorphism. To reduce to the regular case we can blow-up the non-regular points. In this way, we introduce a chain of $\P^1$'s on the fiber over $p$; this chain does not alter the $\Pic^0$, hence we can assume by abuse of notation that also $M_{H_p}$ is regular.

We now localize at $q$ and apply \cite[Theorem 4 (b), Section 9.5, page 267]{BLR}: both $ \Pic^0\left(M_{H_p}\right) $ and $ \Pic^0\left(M_{H}\right)$ are the connected component of the identity of the N\'eron models of $ \Pic^0\left(M_{H_p}\right)_{\Q} $ and $ \Pic^0\left(M_{H}\right)_{\Q}$, hence $\calA$ is the connected component of the identity of the N\'eron model of $\calA_{\Q}$ (this last assertion can checked using the universal property of N\'eron models). Moreover, by Lemma \ref{lem:fibra} and  \cite[Proposition 6.7, page 143]{DeRa}, $\calA$ has semi-abelian reduction.

When there is semi-abelian reduction, by \cite[Proposition 3, section 7.5, page 186]{BLR}, taking N\'eron models is exact up to isogeny, so we have the claim.

\end{proof}

Now, we start looking at the singular cohomology of the complex variety $\calA(\C)$, we will denote it by $H^{1\text{sing}}(\calA(\C),\Z)$. This cohomology group will be used to compare the adjacency matrix of the graph with the action of the Hecke operator on the \'{e}tale cohomology $H^{1,\text{\'et}}(\calA_{\overline{\F}_\ell}, \Q_{\ell'})$, where we have the bound on the eigenvalues from Theorem \ref{thm:bound}.

\begin{Lemma}\label{lem:fibra} 
Fix $p,\ell,H$ as in Definition \ref{def:graph}. Let $\calA$ be the abelian variety in Definition \ref{def:A} and let $T^{\vee}$ be the group of characters of the torus $T$ introduced in Equation \eqref{eq:T_1}.

There is a (non-canonical) isomorphism of $T_{\ell}$ modules
$$
( T^{\vee} \otimes \Q )^{\oplus 2} \cong  H^{1,\text{sing}}(\calA(\C), \Z)\otimes \Q
$$
which is equivariant for the automorphisms from Definitions \ref{def:aut_modular}, \ref{def:fricke} and \ref{def:AL}.
% Fix $p,\ell,H$ as in Definition \ref{def:graph}. There is a (non-canonical) isomorphism of $T_{\ell}$ modules
% $$
% ( T^{\vee} \otimes \C )^{\oplus 2} \cong  H^{1,\text{sing}}(\calA(\C), \Z)\otimes \C
% $$
% where  $T^{\vee}$ is the group of characters of the torus $T$ introduced in Equation \eqref{eq:T_1} and $\calA$ is the abelian variety in Definition \ref{def:A}. This isomorphism is also equivariant for the automorphism from Definitions \ref{def:aut_modular}, \ref{def:fricke} and \ref{def:AL}
\end{Lemma}
\begin{proof}
First we show that there exists a non-canonical isomorphism $\gamma$ of $T_{\ell}$-modules. For this it is enough showing a $\Q$-linear isomorphism between   $T^{\vee} \otimes \Q$ and $H^{1,\text{sing}}(\calA(\C), \Z)\otimes \Q$, as $\Q[x]$-modules, with $x$ acting as $T_{\ell}$. Since $\Q[x]$ is a PID, it is enough showing that for every polynomial $q$ in $\Z[x]$, the rank of $F:=q(T_{\ell})$ is equal on both spaces. %, we have the claim. 
The morphism $F$ is an endomorphism of $\calA$. The rank of $F$ restricted to $T^{\vee}\otimes \Q$ is equal to $\dim(\mathrm{Im}\left( F|_{\calA_{\overline{\F}_\ell}}\right))$. The rank of $F$ on $H^{1,\text{sing}}(\calA(\C), \Q)=H^{1,\text{sing}}(\calA(\C), \Z)\otimes \Q$ is equal to twice $\dim(\mathrm{Im} \left(F|_{\calA_{\C}}\right))$. We obtain the claim by Lemma \ref{lem:Neron_exact}.

% Now we have to show that we can choose a $\gamma$ which is equivariant for all automorphisms. Let $G$ be the group formed by these automorphisms. Theorem \ref{thm:comparison} and Proposition \ref{prop:diag} imply that $T_{\ell}$ is semi-simple, hence that both $V = T^{\vee} \otimes \Q )^{\oplus 2}$ and $W = H^{1,\text{sing}}(\calA(\C), \Z)\otimes \Q$ can be decomposed in eigenspaces, i.e. subspaces $V_\mu, W_\mu$of the form $\ker(\mu(T_\ell))$. Since  $G$ commutes with $T_{\ell}$, then it preserves the eigenspaces of $T_{\ell}$. Each eigenspace is a $G$ module, and we have to show that these $G$ module are isomorphic. To this end, since $G$ is finite, it is enough to show that the characters are the same. This can be proved by looking at the rank of endomorphisms induced by polynomials in elements of $G$, and applying again Lemma \ref{lem:Neron_exact}.

We now show that $\gamma$ can be choosen equivariant for all automorphisms. Let $G$ be the group formed by these automorphisms. Theorem \ref{thm:comparison} and Proposition \ref{prop:diag} imply that $T_{\ell}$ is semi-simple: we can decompose both $V = (T^{\vee} \otimes \Q )^{\oplus 2}$ and $W = H^{1,\text{sing}}(\calA(\C), \Z)\otimes \Q$ in ``eigenspaces'', i.e. subspaces $V_\mu, W_\mu$ of the form $\ker(\mu(T_\ell))$ for  $\mu\in \Q[x]$ an irreducible polynomial. Notice that $\Q[T_\ell]$ acts as the field $K = \Q[x]/\mu$ on $V_\mu$ and  $W_\mu$. Since  $G$ commutes with $T_{\ell}$, then it preserves such eigenspaces and we are left to prove that $V_\mu \cong W_\mu$ as $K[G]$-modules. To this end, since $G$ is finite, it is enough to show that the characters of these two representations are the same. Since each $g\in G$ has finite order, the actions of $g$ on $V_\mu, W_\mu$ are separable, and the traces over $K$ are equal if for each $p\in K[x]$ the ranks of $p(g)$ are equal: this is a consequence of the fact that $V_\mu$ and $W_\mu$ can be described as the images of a certain polynomial in $T_\ell$ and of Lemma \ref{lem:Neron_exact} applied to endomorphisms induced by polynomials in $g$ and $T_\ell$. 
\end{proof}

The following lemma is a rather general fact.
\begin{Lemma}\label{lem:et_vs_sing}
Fix $p,\ell,H$ as in Definition \ref{def:graph}. Let $\calA$ be the abelian variety in Definition \ref{def:A} and denote $H^{*,\text{sing}}$ the singular cohomology.

For any prime $\ell'$ which does not divide $p\ell N$, we have an isomorphism of $T_{\ell}$  modules
$$
H^{1,\text{\'et}}(\calA_{\overline{\F}_\ell}, \Q_{\ell'})\cong H^{1,\text{sing}}(\calA(\C), \Z) \otimes_\Z \Q_{\ell'}\,,
$$
which is equivariant for the automorphisms from Definitions \ref{def:fricke}, \ref{def:aut_modular} and \ref{def:AL}.

% Fix $p,\ell,H$ as in Definition \ref{def:graph}. For any prime $\ell'$ which does not divide $p\ell N$, we have an isomorphism of $T_{\ell}$  modules
% $$
% H^{1,\text{\'et}}(\calA_{\overline{\F}_\ell}, \Q_{\ell'})\cong H^{1,\text{sing}}(\calA(\C), \Z) \otimes_\Z \Q_{\ell'}\,,
% $$
% where $H^{1,\text{sing}}$ denotes the singular cohomology  and $\calA$ is the abelian variety in Definition \ref{def:A}. This isomorphism is also equivariant for the automorphism from Definitions \ref{def:fricke}, \ref{def:aut_modular} and \ref{def:AL}.
\end{Lemma}
\begin{proof}
The isomorphism is given by the cospecialization map, let us explain the argument. By proper-smooth base change theorem (see \cite[Theorem 20.4]{MilnEt}), the cospecialization map 
\begin{equation} \label{eq:etale_proper_smooth}
	H^{1,\text{\'et}}(\calA_{\overline{\F}_\ell}, \Q_{\ell'}) \lto H^{1,\text{\'et}}(\calA_\C, \Q_{\ell'}) \,,
\end{equation}
is an isomorphism. Since the cospecialization map is functorial, then it is an isomorphism of $T_{\ell}$ modules. 

Moreover, since $\calA_\C$ is a smooth variety over $\C$, then the comparison theorem \cite[Theorem 21.1]{MilnEt} tells us that, for each positive integer $k$, we have isomorphisms
\[
H^{1,\text{\'et}}(\calA_\C, \Z/(\ell')^k\Z) \cong H^{1,\text{sing}}(\calA(\C), \Z/(\ell')^k\Z)
\]
Since the above isomorphism is functorial, then, again, it also an isomorphism of $T_{\ell}$ modules. The proof of the second statement is analogous. 
\end{proof}

The following result sums up Theorem \ref{thm:comparison}, Lemma \ref{lem:fibra} and Lemma \ref{lem:et_vs_sing}
\begin{Theorem}[Relation between isogeny graphs and modular curves]\label{thm:riassunto_confronti}
Fix $p,\ell,H$ as in Definition \ref{def:graph}. Let $\calA$ be the abelian variety in Definition \ref{def:A}, and $\ker(w_{i,*})$ as in Remark \ref{rem:genratori_nuclei}.

For any prime $\ell'$ which does not divide $p\ell N$, there is a non-canonical isomorphism
$$
H^{1,\text{\'et}}(\calA_{\overline{\F}_\ell}, \Q_{\ell'})\cong \left(\ker(w_{i,*})\otimes \Q_{\ell'}\right)^{\oplus 2}
$$
which conjugates the action of the Hecke operator $T_{\ell}$ to the action of the adjacency matrix of the isogeny graph. 

Moreover, one can find an isomorphism that conjugates the Galois map \ref{def:Galois_inv} to the Fricke map \ref{def:fricke}, and the automorphisms from Section \ref{s:aut_G} to the automorphisms from \ref{def:aut_modular} and \ref{def:AL} with the same name.
\end{Theorem}

(As funny byproduct, this result shows the well-known fact that the minimal polynomial of the Hecke operator is defined over $\Z$.)

We are now ready to prove Theorem \ref{thm:spet1} about isogeny graphs used in Section \ref{S:Weil_pairing}.
\begin{Theorem}[see Theorem \ref{thm:spet1}]\label{thm:spet2}
The absolute values of the eigenvalues of $A_i$ restricted to $\ker(w_{i,*})\otimes \C$ are strictly smaller than $2\sqrt{\ell}$.

\end{Theorem}
\begin{proof}
%To prove that the absolute values of the eigenvalues are less or equal than $2\sqrt{\ell}$, apply in the following order Theorem \ref{thm:comparison}, Lemma \ref{lem:fibra}, Lemma \ref{lem:et_vs_sing}, and then the combination of Eichler-Shimura relation and Weil conjectures stated in Theorems \ref{Eichler-Shimura}, \ref{thm:bound}. 

Because of Theorem \ref{thm:riassunto_confronti}, the eigenvalues of $A_i$ are the eigenvalue of the Hecke operator $T_\ell$ acting on cohomology a sub-abelian variety $\calA$ of the Jacobian of a modular curve. Then, the combination of Eichler-Shimura relation and Weil conjectures stated in Theorems \ref{Eichler-Shimura}, \ref{thm:bound} implies that the absolute values of the eigenvalues are less or equal than $2\sqrt{\ell}$.

\medskip

In the spirit of the proof of \cite[Theorem 2.1]{CE}, we prove arguing by contradiction that the absolute values of the eigenvalues cannot be equal to $2\sqrt{\ell}$.  
 Suppose that $T_\ell \acts  H^{1,\text{\'et}}(\calA_{\overline{\F}_\ell}, \Q_{\ell'})$  has a complex eigenvalue of absolute value $2\sqrt\ell$. By Proposition \ref{prop:diag}, such an eigenvalue, has an angle in $\tfrac{\pi}{k'}\Z$ for $k'$ an integer, implying that $T_\ell^{2k'}-(4\ell)^{k'}$ is not surjective on $ H^{1,\text{\'et}}(\calA_{\overline{\F}_\ell}, \Q_{\ell'})$ nor, by lemma \ref{lem:et_vs_sing}, on $H^{1,\text{sing}}(\calA(\C), \Q)$. We deduce that $\ker(T_\ell^{2k'}-(4\ell)^{k'})$ in $\calA$ is not finite, hence its connected component of identity $B$ is a sub-abelian variety of $\calA$ defined over $\Q$ of positive dimension. 

% If $2\sqrt\ell$ is the absolute value an eigenvalue of $T_\ell \acts  H^0(A,\Omega^1)$, then $\ker(T_\ell^{2k'}-(4\ell)^{k'})$ has positive dimension, hence its connected component of identity $B$ is a sub-abelian variety of $A$ defined over $\Q$. 
The abelian variety $B$ has good reduction modulo~$\ell$, since it is a quotient of the Jacobian of a modular curve which has good reduction modulo~$\ell$ (indeed having good reduction is stable under isogeny and quotients). 

On the cohomology of $B_{\F_\ell}$ we have the Eichler-Shimura relation \ref{Eichler-Shimura} $T_\ell = \Frob+\langle\ell\rangle\Ver$. Because of this, if an eigenvalue of $T_{\ell}$ has absolute value $2\sqrt\ell$, then $\Frob$ and $\langle\ell\rangle\Ver$ must have exactly the same eigenvalues relative to the same eigenvectors; we conclude that $T_\ell=2\Frob$ on $B_{\F_\ell}$. This implies that the action $T_\ell\acts  H^0(B_{\F_\ell}, \Omega^1)$ is zero. Furthermore, in the case $\ell=2$, it implies that the action of $T_\ell$ on $H^0(B,\Omega^1)$ is a multiple of $2$, and  also the action $\tfrac{T_\ell}{2}\acts  H^0(B_{\F_\ell}, \Omega^1)$ is zero. By the good reduction, we can consider the free $\Z_\ell$-module $M:=H^0(B_{\Z_\ell}, \Omega^1)$, and $M\otimes_{Z_\ell} \F_\ell = H^0(B_{\F_\ell}, \Omega^1)$. Knowing the action of $T_\ell$ and $\tfrac{T_\ell}{2}$  on $M\otimes \F_\ell$, we deduce that the action $T_\ell\acts M$ must be a multiple of $2\ell$, hence its determinant must be a multiple of $(2\ell)^{\rank M} = 2^{\dim B}\ell^{\dim B}$. This is absurd since by looking at the eigenvalues, the determinant of $T_\ell\acts M$ is a root of unity times $(2\sqrt\ell)^{\dim B} = 2^{\dim B}\ell^{\dim B/2}$. Hence there are no eigenvalues of $T_\ell$ of absolute value $2\sqrt\ell$.

%We can now apply Lemma \ref{lem:bound_poly} to the characteristic polynomial of the adjacency matrix $A_i$ restricted to $\ker(w_{i,*})$ to  show that the absolute values of the eigenvalues are smaller than $2\sqrt{\ell}-\left(4\sqrt{\ell}\right)^{-2dk'+1}$.

\end{proof}

\section{Relation with modular forms}\label{s:mocular_forms}
In this section we identify our spaces $\ker(w_{i,*})$ from Remark \ref{rem:genratori_nuclei} with spaces of modular forms. We start from the following lemma.
\begin{Lemma}\label{lem:hodge}
Fix $p,\ell,H$ as in Definition \ref{def:graph}. Let $\calA$ be the 
abelian variety in Definition \ref{def:A} and let $T^{\vee}$ be the group of characters of the torus $T$ introduced in Equation \eqref{eq:T_1}. 

We have a (non-canonical) isomorphism of $T_{\ell}$ modules
$$
T^{\vee}\otimes \C \cong H^0(\calA_{\C},\Omega^1)
$$
which is equivariant for the automorphisms $u$ from Definitions \ref{def:aut_modular}, \ref{def:fricke} and \ref{def:AL}, acting by pullback on $\calA$, hence as $u^{*,\vee}$ on $T^\vee$ and as $(u^*)^*$ (see Remark \ref{rmk:pullpullpush}) on the differentials of $\calA_\C$.

% Fix $p,\ell,H$ as in Definition \ref{def:graph}. We have a (non-canonical) isomorphism of $T_{\ell}$ modules
% $$
% T^{\vee}\otimes \C \cong H^0(\calA_{\C},\Omega^1)
% $$
% where  $T^{\vee}$ is the group of characters of the torus $T$ introduced in Equation \eqref{eq:T_1} and $\calA$ is the 
% abelian variety in Definition \ref{def:A}.
% This isomorphism is equivariant for the automorphisms $u$ from Definitions \ref{def:aut_modular}, \ref{def:fricke} and \ref{def:AL}, acting by pullback on $\calA$, hence as $u^{*,\vee}$ on $T^\vee$ and as $(u^*)^*$ (see Remark \ref{rmk:pullpullpush}) on the differentials of $\calA_\C$.
\end{Lemma}
\begin{proof}
It is enough giving an isomorphism $T^\vee \otimes \C \cong H^0(\calA_\C,\Omega^1)$, which is analogous to Lemma \ref{lem:fibra}.
\end{proof}

\begin{Remark}\label{rmk:pullpullpush}
For a map of curves $u\colon X\to Y$, we have the pullback $u^*\colon \Pic^0(Y) \to \Pic^0(X)$ and its pullback 
$$(u^*)^*\colon H^0(\Pic^0(X), \Omega^1) = H^0(X, \Omega^1) \lto H^0(\Pic^0(Y), \Omega^1)  = H^0(Y, \Omega^1)\,.$$
Then, the above map is equal to  the pushforward of differentials $u_*\colon H^0(X, \Omega^1) \to H^0(Y, \Omega^1)$.  In particular, in Lemma \ref{lem:hodge}, an automorphism $u$ acts as the restriction of $u_*$ on $H^0(\calA_\C, \Omega^1)$
\end{Remark}

The above Lemma, together with Theorem \ref{thm:comparison}, suggests the study differentials on $\calA$: in Theorems \ref{thm:modular_general} and \ref{thm:modular_special} we relate these differentials with modular forms.

To include non-connected modular curves in our analysis, in Subsections \ref{s:6.1} - \ref{s:6.4} we recall the notation, mainly following \cite{DS}, and collect some slightly cumbersome computations.

\subsection{Complex points on modular curves}\label{s:6.1}
Analogously to \cite[IV.5.3]{DeRa}, using the definition $\HH^\pm := \C - \R$ and its ``compactification'' $\ol{\HH}^\pm := \HH^\pm \cup \P^1(\Q)$, we have a (canonical) isomorphism of Riemann surfaces
\begin{subeqn} \label{eq_complex_points_can}
\begin{gathered}
    \text{GL}_2(\Z) \backslash \big(\ol{\HH}^\pm \times(\GLmod{N}/H)\big)  \overset{\sim}{\lto}  M_H(\C),
\end{gathered}
\end{subeqn}
where for each $\tau$'s in $\HH^\pm$ (on proper elliptic curves) the map identifies 
\[
    (\tau,\gamma H)  \longmapsto  (E_{\tau},\phi_{\tau}\circ\gamma ) = (\C/(\Z {+} \Z\tau) , \phi_{\tau}\circ\gamma),
    \quad 
    \phi_\tau \left(\begin{smallmatrix}1 \\ 0\end{smallmatrix}\right)=\tfrac{\tau}{N}\,, \phi_\tau \left(\begin{smallmatrix}0 \\ 1\end{smallmatrix}\right)=\tfrac{1}{N} \,.
\]
In the above isomorphism $\text{GL}_2(\Z) $ acts by 
\begin{subeqn}\label{action_GL2}
\begin{gathered}
 g \cdot (\tau ,\gamma H) := (g(\tau), \bar{g}^{-T} \gamma H) \qquad\qquad \text{i.e.} \\
\smt abcd (\tau ,\gamma H) = \left(  \tfrac{a\tau+b}{c\tau+d}, \,\, \tfrac{1}{\det g} \smt{d}{-c}{-b}{a} \gamma H \right)\,.
\end{gathered}
\end{subeqn} 
For the subgroup $H_p < \GLmod{Np}$, Equation \eqref{eq_complex_points_can} can be rephrased as
\begin{subeqn}\label{Hp_complex_points_can} 
	\begin{aligned}
		%\text{GL}_2(\Z) \backslash \big(\ol\HH^\pm \times\tfrac{\GLmod{n}}{H} \times \tfrac{\GL_2(\F_p)}{B_0(p)} \big) \cong
  \Gamma^0(p) \backslash \big(\ol\HH^\pm \times\tfrac{\GLmod{n}}{H}\big) & \overset{\sim}{\lto} M_{H_p}(\C) 
  %\\ 
	\,,\qquad 	(\tau,\gamma)  \longmapsto (E_\tau, \phi_{\tau} \circ \gamma, \langle \tfrac \tau p \rangle )\,,
	\end{aligned}
\end{subeqn}
where $\Gamma^0(p)$ is the subgroup of $\text{GL}_2(\Z)$ made of matrices congruent to $\smt *0**$ modulo~$p$. 

Using the above  isomorphisms %\eqref{eq_complex_points_can} \eqref{Hp_complex_points_can}, 
the maps $\pr_p$ and $\quot_p$ in \eqref{eq:maps_Tl} become
\begin{comment}
    \begin{subeqn}\label{pr_C}
	\begin{aligned}
		M_{H_p}(\C) = \Gamma^0(p)  \backslash \big(\HH^\pm \times \tfrac{\GLmod{N}}{H}\big) & \overset{\pr_p}{\lto}  \text{GL}_2(\Z) \backslash \big(\HH^\pm \times \tfrac{\GLmod{N}}{H} \big) = M_{H}(\C) %{\xrightarrow{\hspace{1cm}}}  
		\\
		(\tau,\gamma) = \big(\C/\langle 1,\tau\rangle,     \phi_\tau \circ \gamma, \langle \tfrac \tau p \rangle \big) &\longmapsto  \big(\C/\langle 1,\tau\rangle, \phi_{\tau} \circ\gamma \big) =  (\tau,\gamma)
		% \\    &\longmapsto (\tau,\gamma)
	\end{aligned}
\end{subeqn}
\begin{subeqn}\label{quot_C}
	\begin{aligned}
		M_{H_p}(\C) = \Gamma^0(p)  \backslash \big(\HH^\pm \times \tfrac{\GLmod{N}}{H}\big)  & \overset{\quot_p}{\lto}  %{\xrightarrow{\hspace{1cm}}} 
		\text{GL}_2(\Z) \backslash \big(\HH^\pm \times \tfrac{\GLmod{N}}{H} \big) = M_{H}(\C)
		\\  \qquad \qquad \big(\C/\langle 1,\tau\rangle, \phi_\tau \circ \gamma, \langle \tfrac \tau p \rangle \big)   & \longmapsto  \big(\C/\langle 1, \tfrac\tau p\rangle,  \phi_\tau \circ \gamma \big)  \cong \big(\C/\langle 1, \tfrac\tau p \rangle,  \phi_{\tau/q} \circ \smt q001 \gamma \big) 
		\\ (\tau,\gamma) & \longmapsto(\smt 100q \tau,\smt q001 \gamma) % = ( q\tau,\smt 100q \gamma)
	\end{aligned}
\end{subeqn}
\end{comment}
\begin{subeqn}\label{pr_quot_C}
\begin{gathered}
    \pr_p,\, \quot_p\,\colon  \Gamma^0(p)  \backslash \big(\HH^\pm \times \tfrac{\GLmod{N}}{H}\big) \lto \text{GL}_2(\Z) \backslash \big(\HH^\pm \times \tfrac{\GLmod{N}}{H} \big) \,,
  \\
\pr_p (\tau,\gamma) = (\tau,\gamma) \,, \qquad
	\quot_p	(\tau,\gamma) = ( \smt 100p\tau,\smt p001  \gamma)
\end{gathered}
\end{subeqn}

%As already recalled in Section \ref{S:relation},  if $\det(H)\neq (\Z/n\Z)^\times$, the curves $M_H$ and $M_{H_p}$ might be not connected and t
The isomorphisms \eqref{eq_complex_points_can} \eqref{Hp_complex_points_can} also help us recognize the components, over $\C$, of modular curves: choosing representatives $g_1,\ldots, g_r$ for the quotient $\GLmod{N}/(H \cdot \text{SL}_2(\Z/N\Z))$, we get the following (non-canonical) decomposition into connected components
\begin{subeqn}
\label{components_XH}
\begin{gathered}
     M_H(\C) \cong \bigsqcup_{j=1}^r  \Gamma_{g_jHg_i^{-1}} \backslash \ol{\HH}\,,\qquad  \left( E_\tau, \phi_{\tau}\circ g_j \right) \longmapsfrom (\tau, g_j)\,,
     \\[-1ex]
     M_{H_p}(\C) \cong \bigsqcup_{j=1}^r  \big(\Gamma^0(p){\cap} \Gamma_{g_jHg_i^{-1}}\big) \backslash \ol{\HH}\,, \qquad  
	%\left(\Gamma_{g_jHg_i^{-1}} \cap \Gamma_0(p)\right) 
  \left( E_\tau, \phi_{\tau}\circ g_j, \langle \tfrac \tau p\rangle \right) \longmapsfrom (\tau, g_j)\,,
\end{gathered}  
\end{subeqn}
where $\ol\HH= \HH \cup \P^1(\Q)$ is the ``compactification'' of $\HH = \{ \tau \in \C: \mathrm{Im}(\tau)>0 \}$, and where
\begin{equation*}\label{eq:Gamma_H}
	\Gamma_{H}:= \{ \gamma \in \text{SL}_2(\Z): \gamma^T\!\!\!\pmod n \text{ lies in }H\}.
\end{equation*}

\begin{subRemark}\label{rem_00}
In Equation \eqref{Hp_complex_points_can} we use $\Gamma^0(p) = \Gamma_{B^0(p)}$, with $B^0(p)$ the Borel group $\{\smt**0*\}$ (notice the transposition in \eqref{action_GL2}). Since conjugation of the $H_p$ gives an isomorphic modular curve, we can also use $B_0(p) = \{\smt*0**\} = \smt 0110 B^0(p) \smt0110^{-1}$, yielding a variant of \eqref{Hp_complex_points_can}: 
\begin{subeqn}\label{Hp_complex_points_can_2} 
	\begin{aligned}
  \Gamma_0(p) \backslash \big(\ol\HH^\pm \times\tfrac{\GLmod{n}}{H}\big) & \overset{\sim}{\lto} M_{H_p}(\C) \,, 
  \qquad 
		(E_\tau, \phi_{\tau} \circ \gamma, \langle \tfrac 1p \rangle ) \longmapsfrom (\tau,\gamma) \,,
	\end{aligned}
\end{subeqn}
for $\Gamma_0(p) = \Gamma_{B_0(p)} = \{\smt abcd \in \mathrm{GL}_2(\Z) : c \equiv 0 \mod p \}$.
%Consistently, a variant of \ref{pr_quot_C} is 
%\begin{subeqn}\label{pr_quot_C_0basso}
%\begin{gathered}
%    \pr_p,\, \quot_p\,\colon  \Gamma_0(p)  \backslash \big(\HH^\pm \times \tfrac{\GLmod{N}}{H}\big) \lto \text{GL}_2(\Z) \backslash \big(\HH^\pm \times \tfrac{\GLmod{N}}{H} \big) \,,
%  \\
%\pr_p (\tau,\gamma) = (\tau,\gamma) \,, \qquad \quot_p	(\tau,\gamma) = ( \smt p001\tau,\smt 100p  \gamma)
%\end{gathered}
%\end{subeqn}
\end{subRemark}

\subsection{Modular forms and differentials}
For any congruence subgroup $\Gamma$ of $\SLZ$, the map $f \mapsto f \mathrm{d}\tau$ gives an isomorphism between the space $S_2(\Gamma)$ of cuspidal modular forms of weight 2 and the space $H^0(\Gamma \backslash \ol\HH, \Omega^1)$ of holomorphic differentials on $\Gamma \backslash \ol\HH$, see \cite[Section 3.3 and Excercise 3.3.6]{DS} or \cite[Theorem 2.3.2]{Miy}. 
This, together with \eqref{components_XH} implies the isomorphisms
\begin{subeqn}\label{Omega_XH}
 H^0(M_{H,\C} , \Omega^1) \cong \bigoplus_{j=1}^r S_2\left(\Gamma_{g_jHg_i^{-1}}\right) \,, 
 \qquad H^0(M_{H_p,\C} , \Omega^1) \cong \bigoplus_{j=1}^r S_2\left(\Gamma_{g_jHg_i^{-1}} \cap \Gamma^0(p)\right) \ .
\end{subeqn}

\subsection{Full level case}
When $H = \{ \Id \}$, we write $M_N$ for $M_H$ and 
$\Gamma(N)$ for $\Gamma_H$, which contains matrices in $\SLZ$ congruent to $\smt 1001$ modulo $N$. Choosing $\{ g_i\} = \{ \smt a001: a \in (\Z/N\Z)^\times \}$, Equation \eqref{components_XH} gives  
\begin{subeqn}\label{iso_M_np}
 M_N(\C) \cong \bigsqcup_{a\in (\Z/N\Z)^\times}\!\!\!\! \Gamma(N) \backslash \ol{\HH}\,,\quad  M_{\{\Id\}\times B_0(p)}(\C) \cong  \bigsqcup_{a\in (\Z/N\Z)^\times}\!\!\!\! \big(\Gamma^0(p){\cap}\Gamma(N)\big) \backslash \ol\HH \,,
 % \quad &\big(\C/\Z{+}\Z\tau, (\tfrac \tau N, \tfrac aN), \langle \tfrac 1p\rangle \big)  \leftrightarrow ([\tau], a)
\end{subeqn}
and, compatibly with this isomorphisms, the maps $\pr$, $\quot$ are 
\begin{subeqn}\label{pr_quot_C_Mn}
\begin{gathered}
    \pr_p,\, \quot_p\,\colon  \!\!\!\! \bigsqcup_{a\in (\Z/N\Z)^\times}\!\!\!\! \big(\Gamma^0(p){\cap}\Gamma(N)\big) \backslash \ol\HH  \lto \!\!\!\! \bigsqcup_{a\in (\Z/N\Z)^\times}\!\!\!\! \Gamma(N) \backslash \ol\HH  \,,  
  \\
  \pr_p (\tau,a) = (\tau,a)\,, \qquad  \quot_p	(\tau,a) = ( \smt 100q\tau, pa )
\end{gathered}
\end{subeqn}

Moreover, Equation \eqref{Omega_XH} becomes 
\begin{subeqn}\label{Omega_Xn}
\begin{aligned}
& H^0(M_{N,\C} , \Omega^1) \cong \bigoplus_{a\in (\Z/N\Z)^\times}\!\!\!\! S_2\left(\Gamma(N)\right) = S_2\left(\Gamma(N)\right) \otimes_\C \C^{(\Z/N\Z)^\times}  \,, 
\\
& H^0(M_{\{\Id\}\times B^0(p), \C} , \Omega^1) \cong \bigoplus_{a\in (\Z/N\Z)^\times}\!\!\!\! S_2\left(\Gamma(N) {\cap} \Gamma^0(p)\right)  = S_2\left(\Gamma(N){\cap} \Gamma^0(p)\right) \otimes_\C \C^{(\Z/N\Z)^\times} 
\end{aligned}
\end{subeqn}

\subsection{Hecke operators}\label{s:6.4}
As in \cite[Section 5.1]{DS}, we recall the definition of double coset operators: given $\Gamma_1, \Gamma_2 < \SLZ$ congruence subgroups, and given $\alpha\in \GL_2^{\det >0}(\Q)$, we have the operator 
\begin{subeqn}\label{double_coset}
	[\Gamma_1\alpha\Gamma_2]_2 \colon M_2(\Gamma_1) \to M_2(\Gamma_2)\,, \quad  f [\Gamma_1\alpha\Gamma_2]_2 
 = \sum_j f[\alpha \gamma_j]_2 
 \,,    
\end{subeqn}
where $
f[\smt abcd]_2 (\tau) = \det\smt abcd 
\tfrac{1}{(c\tau+d)^2} f\left( \tfrac{a\tau+b}{c\tau+d} \right)$, and $\{\gamma_j\}$ is a set of representatives for $\Gamma_3\backslash \Gamma_2$, with $\Gamma_3 = \alpha^{-1}\Gamma_1 \alpha \cap \Gamma_2$. We can interpret the operator \eqref{double_coset} as follows: we have maps
\begin{subeqn}\label{double_coset_geo}
	\begin{tikzcd}
		\Gamma_3\backslash\ol\HH \arrow[rrr, "\alpha\colon \tau \mapsto \alpha\tau"] \arrow[d,"\pi_2\colon \tau \mapsto \tau"] &&&     \alpha\Gamma_3\alpha^{-1}\backslash\ol\HH \arrow[d,"\pi_1\colon \tau \mapsto \tau"]
		\\ 
		\Gamma_2\backslash\ol\HH &&&  \Gamma_1\backslash\ol\HH
	\end{tikzcd}    
\end{subeqn}
and, under the isomorphism \eqref{Omega_XH}, we have
$[\Gamma_1\alpha\Gamma_2]_2 = \pi_{2,*} \circ (\pi_{1}\alpha)^*$.
A particular case are the classical Hecke operators in the theory of modular forms, see \cite[Section 5.2]{DS}: 
	\begin{subeqn}\label{tT_pullpush}
		\tT_\ell :=[\Gamma\smt 100\ell \Gamma]_2  = \pi_* \circ {\smt 100\ell}^*\,, \qquad {\smt 100\ell}\,, \pi  \colon \left( \Gamma^0(\ell) {\cap} \Gamma \right) \backslash\ol\HH \to  \Gamma\backslash\ol\HH
	\end{subeqn}
where 
${\smt 100\ell}\tau = \tfrac\tau \ell$, $\pi\tau = \tau$, and 
we consider $\Gamma = \Gamma_H$ for $H<\GLmod{N}$ any subgroup that is normalized by diagonal matrices.

In the case $\Gamma = \Gamma(N)$, we want to compare $\tT_\ell$ with the Hecke operator $T_\ell$ in Definition~\ref{def:Hecke}. Indeed $T_\ell$ acts as $\quot_{\ell,*}\circ \pr_\ell^*$  on $\Pic^0(M_N)$, hence it acts by pull back as $\pr_{\ell,*}\circ \quot_\ell^*$ on $H^0(\Pic^0(M_{N,\C}),\Omega^1) = H^0(M_{N,\C},\Omega^1)$. 
By \eqref{Omega_Xn}, this space of differentials is isomorphic to $S_2\left(\Gamma(N)\right) \otimes_\C \C^{(\Z/N\Z)^\times}$ and, under this identification, Equation
\eqref{pr_quot_C_Mn} tells that that $\pr_{\ell,*} = \pi_* \otimes \Id$ and that $\quot_\ell^* = {\smt 100\ell}^* \otimes \sigma_\ell$, where  $\sigma_{\ell}\colon \C^{(\Z/N\Z)^\times}\to \C^{(\Z/N\Z)^\times}$ is the ``shift by $\ell$'' namely $(z_a)\mapsto (z_{a\ell})$, and the maps $\pi, \smt 100\ell$ are the same appearing in \eqref{tT_pullpush}. We deduce that 
\begin{subeqn}\label{Hecke_Mn_C}
    T_\ell = \tT_\ell\otimes_{\C}\,\sigma_\ell \quad \text{ in } %H^0(X_{B_2(N),\C},\Omega^1) = 
    H^0(\Pic^0(M_{N,\C}),\Omega^1) = S_2(\Gamma(N))      \otimes_{\C} \C^{(\Z/N\Z)^\times}\,.
\end{subeqn}
We have an analogous equality for $H = \{\Id \} \times B_0(p)$: using the second line in \eqref{Omega_Xn} 
\begin{subeqn}\label{Hecke_Mnp_C}
    T_\ell = \tT_\ell\otimes_{\C}\,\sigma_\ell \quad \text{ in }
    H^0(\Pic^0(M_{\{\Id \} \times B_0(p),\C}),\Omega^1) = S_2( \Gamma^0(p) {\cap}\Gamma(N))      \otimes_{\C} \C^{(\Z/N\Z)^\times}\,.
\end{subeqn}

\subsection{Graphs versus modular forms} \label{s:forms_vs_graphs}
In this section we study $H^0(\calA_\C, \Omega^1)$. Definition \ref{def:A} gives the canonical isomorphism 
$$ 
H^0(\calA_\C, \Omega^1)  = \frac{H^0(M_{H_p,\C}, \Omega^1)}{\pr_p^* \,H^0(M_{H,\C}, \Omega^1) + \quot_p^*\, H^0(M_{H_p,\C}, \Omega^1)}\,.
$$       %H^0(\calA_\C, \Omega^1)  = \frac{H^0(\Pic^0(M_{H_p,\C}), \Omega^1)}{\pr_p^* H^0(\Pic^0(M_{H,\C}), \Omega^1) + \quot_p^* H^0(\Pic^0(M_{H_p,\C}), \Omega^1)}
%and, using \eqref{Omega_XH} we can express the right hand side in terms of modular forms. 
We start by looking at the case $H=\{\Id\}$, where Equation \eqref{pr_quot_C_Mn} gives an explicit description of $\pr_p^*, \quot_p^*$. Instead of taking a quotient, we can take the orthogonal complement with respect to the Petersson inner product (see \cite[Section 5.5]{DS}): following \cite{Rib}, we define the space of $p$-new forms as
$$ S_2^\pnew (\Gamma^0(p) {\cap} \Gamma(N)) := \Big( S_2(\Gamma(N)) + S_2(\Gamma(N)[\smt 100p]_2)  \Big)^\perp \quad \subset  S_2(\Gamma^0(p) {\cap} \Gamma(N))\,,$$
which, by the same arguments in \cite[Proposition 5.5.2 and Proposition 5.6.2]{DS}, is $\tT_\ell$-stable. In particular, using the description (\eqref{Hecke_Mnp_C} of the Hecke operator, we get the isomorphism 
\[
T_\ell \lefttorightarrow H^0(\calA_{\{\Id\},p,\C}, \Omega^1) \cong S_2^\pnew (\Gamma^0(p) {\cap} \Gamma(N)) \otimes_\C \C^{(\Z/N\Z)^\times}\righttoleftarrow \tT_\ell \otimes \sigma_\ell \,.
\]
To treat the case of a general $H$, we recall that $\GLmod N$ acts on $M_{\{\Id\}\times B_0(p)}$ by the law  $(E,\phi, C)^g = (E,\phi\circ g, C)$. Using \eqref{action_GL2} and \eqref{iso_M_np}, we can characterise this action as follows:
\begin{equation*}\label{g_tilde}
\begin{aligned}
 (\tau, \smt a001)^{g} = (\tau, \smt{ad}001)  \qquad & \text{ if } g = \smt d001\,,
\\
(\tau, \smt a001)^g = (\tilde{g_a}\tau, \smt a001) \qquad &\text{ if } \det g=1  \,,
  % \\ & \qquad 
  \end{aligned}
\end{equation*}
% \end{subeqn*}
where  $\tilde{g_a}$  is any matrix in $\Gamma^0(p)$ that is congruent to  $\left( \smt a001 g \smt a001^{-1}  \right)^t$ modulo $N$. 
We get an action of $\GLmod N$ by pullback on $H^0(\calA_{\{\Id\},p,\C}, \Omega^1) \subset H^0(M_{\{\Id\}\times B_0(p)})$ as follows:  
\begin{subeqn}\label{action_GL2_2}
\begin{gathered}
\GLmod{N} \lefttorightarrow  % H^0(\calA_{\{\Id\},p,\C}, \Omega^1)   \cong 
S_2^\pnew (\Gamma^0(p) {\cap} \Gamma(N)) \otimes_\C \C^{(\Z/N\Z)^\times} = \bigoplus_{a\in (\Z/N\Z)^\times}\!\!\!\! S_2^\pnew (\Gamma^0(p) {\cap} \Gamma(N))  \ ,
\\ \smt d001 \cdot (f_a)_a  = (f_{ad})_a \,, \quad g \cdot  (f_a)_a = (f_a[\tilde{g_a}]_2)_{a} \text{ if } \det g = 1\,,
\end{gathered}
\end{subeqn}
where he operation $[\cdot]_2$ is as in \eqref{double_coset}, and $\tilde{g_a}$ is chosen as above.
Since pullback of differentials along the natural projection $M_{\{\Id\}\times B_0(p)} \to M_{H_p} $ identifies $H^0(\calA_{H,p}, \Omega^1)$ with the subspace of $H^0(\calA_{\{\Id\},p}, \Omega^1)$ made of $H$-invariant differentials, we get the isomorphism 
\[
T_\ell \lefttorightarrow H^0(\calA_{H,p,\C}, \Omega^1) \cong \left( S_2^\pnew (\Gamma^0(p) {\cap} \Gamma(N)) \otimes_\C \C^{(\Z/N\Z)^\times} \right)^H \righttoleftarrow \tT_\ell \otimes \sigma_\ell \,.
\]
This, together with Lemma \ref{lem:hodge}, Theorem \ref{thm:comparison} and the fact that $A$ is conjugate to $A^*$ (Proposition \ref{prop:diag}) 
implies the following result.

\begin{subTheorem}\label{thm:modular_general}
Let  $G = G(p,\ell,H)$ be the graph in Definition \ref{def:graph}, %with vertices $V$,  
let $\ker(w_{*}), \ker(w_{i,*})$ %$,\ldots, \ker(w_{n,*})$ 
be the subspaces of $\C^{V(G)}$  described in \ref{rem:genratori_nuclei} and let $S$ be the $p$-new part of  $S_2(\Gamma^0(p) {\cap} \Gamma(N))$. 
Then 
\[
\begin{tikzcd}
 \bigoplus_{i=1}^n \ker(w_{i,*}) 
 \arrow[loop left,looseness=6,"A"]  
 \arrow[rr,"\sim"] 
 &&    \left( S \otimes_\C \C^{(\Z/N\Z)^\times} \right)^H \arrow[ll]
 \arrow[loop right,looseness=5,"\tT_\ell\otimes \sigma_\ell"] 
\end{tikzcd} \,.
\]
In  words $\ker(w_*) = \oplus_i \ker(w_{i,*})$, as a module over the adjacency matrix of the graph, is isomorphic to the subspace of  $ S \otimes_\C \C^{(\Z/N\Z)^\times}$ fixed by $H$, as a module over $\tT_\ell\otimes \sigma_\ell$ (for the action of $H<\GLmod N$ see \eqref{action_GL2_2}). % S_2^\pnew (\Gamma^0(p) {\cap} \Gamma(N))
\end{subTheorem}

\begin{subRemark}\label{bobobob}
In Remark \ref{rem_00} we pointed out that $M_{H_p}$ can be described using either $\Gamma^0(p)$ or $\Gamma_0(p)$. Following the same lines, Theorem \ref{thm:modular_general} remains  true after substituting $S_2^\pnew (\Gamma^0(p) {\cap} \Gamma(N))$  with  
\[
S_2^\pnew (\Gamma_0(p) {\cap} \Gamma(N)) := \Big( S_2(\Gamma(N)) + S_2(\Gamma(N))[\smt p001]_2  \Big)^\perp \quad \subset  S_2(\Gamma_0(p) {\cap} \Gamma(N))\,.
\]
and after slightly modifying the action of $\GLmod{N}$ in \eqref{action_GL2_2}, i.e. asking that $\tilde{g_a}\in \Gamma_0(p)$.
\end{subRemark}

\medskip
We also rephrase Theorem \ref{thm:modular_general}, for certain choices of $H$, using  modular forms for 
$$\Gamma_1(k) = \{m \in \SLZ : m \equiv \smt 1*01 \bmod k\} \,, \quad \Gamma_0(k) = \{m \in \SLZ : m \equiv \smt **0* \bmod k \} \,.$$
Such modular forms received more attention in the literature, e.g. in the asymptotic estimates in \cite{Serre} which we later use. We use the decomposition, (see \cite[Section 4.3, page 119]{DS}),
\begin{subeqn}\label{chi}
S_2(\Gamma_1(k))  = \bigoplus_{\chi \in (\Z/k\Z)^{\times,\vee} } S_2(\Gamma_1(k),\chi) \,,
\end{subeqn}
where $\chi$ varies across all characters modulo $k$. In particular, it follows from the definitions that $S_2(\Gamma_0(p) \cap \Gamma_1(N))$ is a subspace of $S_2(\Gamma_1(Np))$ and precisely the subspace fixed by all the diamond operators (in the sense of \cite[Section 5.2]{DS}) $\langle d\rangle$ for $d\equiv 1 \bmod N$. This implies that 
\[
S_2(\Gamma_0(p) \cap S_2(\Gamma_1(N)) =  \bigoplus_{\chi \in (\Z/N\Z)^{\times,\vee} } S_2(\Gamma_1(pN),\chi) \,.
\]
where we notice that we are not summing over all characters $\chi$ modulo $Np$, as in \eqref{chi}, instead we only look at the characters $\chi\colon (\Z/Np\Z)^\times \to \C^\times$ that factor through the projection $(\Z/Np\Z)^\times \to (\Z/N\Z)^\times$. Moreover, if $f$ is a modular form in $S_2(\Gamma_1(N),\chi)$ for some character $\chi$ modulo~$N$, then both $f$ and $f[\smt p001]_1$ belong to $S_2(\Gamma_1(Np),\chi)$  by \cite[Proposition 5.6.2]{DS}. Using this fact we define the spaces of $p$-new forms
\[
\begin{gathered}
    S_2^{p-\new}(\Gamma_0(p)\cap \Gamma_1(N)) := \Big( S_2(\Gamma_1(N)) + S_2(\Gamma_1(N))[\smt p001]_2  \Big)^\perp \quad \subset  S_2(\Gamma_0(p) {\cap} \Gamma_1(N)) \,,
    \\ S_2^{p-\new}(\Gamma_1(pN), \chi) := \Big( S_2(\Gamma_1(N), \chi) + S_2(\Gamma_1(N), \chi)[\smt p001]_2  \Big)^\perp \quad \subset  S_2(\Gamma_1(pN), \chi) \,,
\end{gathered}
\]
where $\chi$ is modulo~$N$ and  the orthogonal is taken with respect to the Petersson inner product.

\begin{subTheorem}
\label{thm:modular_special}
Let  $G(p,\ell,H)$ be the graph in Definition \ref{def:graph}, with vertices $V$ and adjacency matrix $A$, and let $\ker(w_{1,*}),\ldots, \ker(w_{n,*})$ be the subspaces of $\C^V$  described in Remark \ref{rem:genratori_nuclei}. 

Then,
 
\begin{itemize}
\item if $H = \{\Id\} < \GLmod N$, each $\ker(w_{i,*})$, as an $A$-module, is isomorphic to $S'\otimes_\C \C^L$, as a module over  $\tT_\ell\otimes \sigma_\ell$, where $L = \langle \ell \rangle\subset \modstar N$,  $\sigma_\ell\colon \C^L \to \C^L$ sends $(a_x)_{x\in L}$ to $(a_{x\ell})_{x\in L}$, 
and $S'$ is the following space of modular forms
\[
S' = \bigoplus_{\chi \in (\Z/N\Z)^{\times,\vee} } S_2^{p\text{-new}}(\Gamma_1(pN^2),\chi)  \,,
\]
with $\chi$ varying across the characters that factor through the projection $(\Z/pN^2\Z)^\times \to (\Z/N\Z)^\times$. 

\item if $H=B_0(N)=\{\smt*0**\}$ then $n=1$ and $\ker(w_{1,*}) = \{ (x_v)_v \in \C^V : \sum_v x_v = 0\}$, as a module over $A$ is isomorphic to $S_2^{p-\new} (\Gamma_0(pN))$ as a module over $\tT_\ell$. 

\item if $H=B_1(N)=\{\smt*0*1\}$ then $n=1$ and $\ker(w_{1,*}) = \{ (x_v)_v \in \C^V : \sum_v x_v = 0\}$, as a module over $A$ is isomorphic to $S'$ as a module over $\tT_\ell$, with 
$$
S' = S_2^{p-\new}(\Gamma_0(p)\cap \Gamma_1(N)) = \bigoplus_{\chi \in (\Z/N\Z)^{\times,\vee}}  S_2^{p-\new}(\Gamma_1(pN), \chi) .
$$

\item if $H$ is a non-split Cartan of level $N$, then  $n{=}1$ and $\ker(w_{1,*})$ as an $A$-module, is isomorphic to 
$$\bigoplus_{d|N}S_2^{\new}(\Gamma_0(pd^2)) \,,$$
as a $\tT_\ell$-module (see \cite[Section 5.6]{DS} for the definition of $S_2^{\new}$).

\end{itemize}   
\end{subTheorem}

\begin{proof}
By Lemma \ref{lem:hodge} it is enough to describe the $T_\ell$-module $H^0(\calA_{H,p},\Omega^1)$.

The cases $H = B_0(N) = \{\smt*0**\}$ and $B_1(N) = \{\smt*0*1\}$ can be treated with the same arguments used for the full level structure in Theorem \ref{thm:modular_general}, even slightly easier: $M_{B_0(N)_p}(\C)$ and $M_{B_1(N)_p}(\C)$ are connected and isomorphic to $\Gamma_0(pN)\backslash\ol\HH$ and $(\Gamma_0(p){\cap}\Gamma_1(N))\backslash\ol\HH$, and, since $\smt \ell001$ belongs to $H$, the graph is connected and then $T_\ell$ acts exactly as $\tT_\ell$. 

The full level structure case is a consequence of the Hecke-equivariant isomorphisms
\begin{equation*}\label{eq_Gamma(n)_Gamma_1(n2)}
\begin{aligned}
\calM_{B'(N^2)} &\lto \calM_N \,,\qquad \qquad \qquad(E,(P,Q)) \longmapsto (E/\langle nQ\rangle,( nP, Q)) 
%    \\
%    \calM_{B'(N^2)} &\longleftarrow \calM_N \,, \,\quad (E/\langle P \rangle ,N^{-2}P,N^{-1}Q)) \longmapsfrom (E, P, Q) \,,
\\ \calM_{B'(N^2)_p} &\lto \calM_{\{\Id\}\times B_0(p)} \,,\qquad \qquad \qquad(E,(P,Q), G) \longmapsto (E/\langle nQ\rangle, (nP, Q), G)  \,,
\end{aligned}
\end{equation*}
where $B'(N^2)$ is the subgroup $\{\smt{1+N*}{0}{*}{1+N*}\}$ of $\GLmod{N^2}$ and where we  identify isomorphisms $\phi\colon (\Z/k\Z)^2\to E[k]$ with basis $(P,Q)$ of the group $E[k]$. 

We reduced to $B'(N^2)$ structures. The inclusion $B'(N^2) \supset B_2(N^2) :=\{\smt 10*1\}$ induces a map $M_{B_2(N^2)} \to M_{B'(N^2)}$ that identifies  $H^0(M_{B'(N^2)}, \Omega^1)$ with the $B'(N^2)/M_{B_2(N^2)}$-invariant subspace of $H^0(M_{B'(N^2)}, \Omega^1)$.  Choosing $\{ g_i\} = \{ \smt a001: a \in (\Z/N^2\Z)^\times \}$, Equation \eqref{components_XH} gives   %Using the isomorphisms
\[
\begin{aligned}
M_{B_2(N^2)}(\C) &\cong  \bigcup_{a\in (\Z/N^2\Z)^\times} \Gamma_1(N^2) \backslash \ol\HH \,, \quad &\big(\C/\Z{+}\Z\tau, (\tfrac{a\tau}N, \tfrac 1N)\big)  \leftrightarrow (\tau, a) \,,
	\\
M_{B_2(N^2)_p}(\C) &\cong  \bigcup_{a\in (\Z/N^2\Z)^\times} \big(\Gamma_1(N^2){\cap}\Gamma_0(p)\big) \backslash \ol\HH \,, \quad &\big(\C/\Z{+}\Z\tau, (\tfrac {a\tau} N, \tfrac1N), \langle \tfrac 1p\rangle \big)  \leftrightarrow (\tau, a)\,.
\end{aligned}
\]

The action of $B'(N^2)/B_2(N^2)$ identifies certain components (two points $(\tau, a)$, $(\tau, a')$ are identified iff $a\equiv a' \pmod N$) and that within the same components identifies a point $(\tau,a)$ with the point $(\dia{d}\tau,a)$ for $d\equiv 1 \pmod N$ and $\dia d$ the diamond operator in \cite[Section 5.2]{DS}. We deduce the following isomorphism of Hecke-modules
\[
\begin{aligned}
H^0(\calA_{\{\Id\}, p} ,\Omega^1) &\cong   H^0(\calA_{B_2(N^2), p} , \Omega^1)^{B'(N^2)/B_2(N^2)} = \bigoplus_{a\in (\Z/N\Z)^\times} \bigoplus_{\chi \in (\Z/N\Z)^{\times, \vee}} S_2^{p-\new}(\Gamma_1(pN^2), \chi) 
\\ & = \left(  \bigoplus_{\chi \in (\Z/N\Z)^{\times, \vee}} S_2^{p-\new}(\Gamma_1(pN^2), \chi)\right) \otimes \C^{(\Z/N\Z)^\times} \,,
\end{aligned}
\]
on which, by the same arguments used in Theorem \ref{thm:modular_general},
the Hecke operator acts as $\tT_\ell\otimes \sigma_\ell$. 

For $H$ a non-split Cartan our result follows from the  $T_\ell$-equivariant isogenies \cite[Lemma 3.1 and Theorem 3.8]{manypoints}
\[
\Pic^0(M_H) \sim \prod_{d|N} J_0^\new(d^2) \,,\quad \Pic^0(M_{H_p}) \sim \prod_{d|N} \left( J_0^\new(d^2)^2 \times J_0^\new(pd^2) \right)\,,
\]
where $J_0^\new(k)$ denotes the new part of the Jacobian of $M_{B_0(k)}$.
\end{proof}

\subsection{Automorphisms of the graphs versus automorphisms of spaces modular forms}\label{s:forms_vs_aut}

We now study how the automorphisms in Definitions  \ref{def:fricke}, \ref{def:aut_modular} and \ref{def:AL} act on a point of $M_{\{\Id\}\times B_0(p)}$  (or a quotient $M_{H_p}$) under the isomorphism \eqref{iso_M_np}. Recall that a point $(a, \tau)$ corresponds to the elliptic curve $E_\tau = \C/\Z{+}\Z\tau$ together with the subgroup $\langle \tfrac\tau p\rangle$ and the basis $(\tfrac{a\tau}{N}, \tfrac1N)$ of $E[N]$ (such a basis corresponds to the isomorphism $\phi_\tau \colon (\Z/N\Z)^2 \to E[N]$ sending the standard basis to it). 

The Fricke automorphism $\sigma$ sends the  point $(a,\tau)$ to the elliptic curve $\C/\Z{+}\Z\tfrac\tau p$, with the subgroup $\langle\tfrac1p\rangle$ and with the basis $(\tfrac{a\tau} N, \tfrac1N)$ of the $N$-torsion. The multiplication by $\tau' = -\tfrac{p}{\tau}$ inside $\C$ induces an isomorphism between this elliptic curve and the elliptic curve $E_{\tau'}$, with the subgroup $\langle \tfrac{\tau'}p \rangle$ and the basis $(-\tfrac{ap}N, \tfrac{\tau'}N)$, namely the point of $(\tau', \smt{0}{1}{-ap}{0})$ under the canonical isomorphism \eqref{Hp_complex_points_can}. If we now apply the action \eqref{action_GL2} of a matrix 
\begin{equation*}\label{tildem} 
\tilde m \in \Gamma^0(p)  \quad \text{ such that } \tilde m \equiv \smt 0{-1}10 \bmod N \,,
\end{equation*}
we see that this point is equivalent to the point  $(\tilde m(\tau'), \smt{ap}{0}{0}{1})$, that is the point $(\tilde m \smt0{-p}10 \tau, ap)$. We deduce that 
\[
\sigma^* = [\tilde m \smt0{-p}10]_2 \otimes \sigma_p  \quad \text{in } H^0(M_{\{\Id\}\times B^0(p), \C} , \Omega^1) \cong S_2\left(\Gamma(N){\cap} \Gamma^0(p)\right) \otimes_\C \C^{(\Z/N\Z)^\times} 
\]
where $\sigma_p\lefttorightarrow \C^{(\Z/N\Z)^\times} $ is the shift $(x_a)\mapsto (x_{ap})$. 
Inspired by the above discussion we give the following
\begin{subDefinition}\label{def:Fricke_forms}
The Fricke automorphism on full level modular forms is  
$$ w_p\colon S_2\left(\Gamma(N){\cap} \Gamma^0(p)\right) \lto S_2\left(\Gamma(N){\cap} \Gamma^0(p)\right) \,, \quad f \longmapsto f[m_\sigma]_2 $$
for $m_\sigma=\tilde m \smt0{-p}10$ and $\tilde m \in \Gamma^0(p)$ a matrix congruent to $\smt 0{-1}10$ modulo~$N$.
\end{subDefinition}

For matricial automorphisms as in Definition \ref{def:aut_modular} we have already computed their action in Equation \eqref{action_GL2_2}. In particular, diamond operators $\langle d \rangle$ act as $\dia d \otimes \sigma_{d^2}$ for $\dia d$ as in the next definition (which coincides with the diamond operator in \cite[Section 5.2]{DS})
\begin{subDefinition}
    Given $H<\GLmod{N}$, for each $d\in (\Z/N\Z)^\times$, we have a diamond operator 
    $$ \dia d \colon S_2(\Gamma_H) \lto S_2(\Gamma_H) \,, \quad f \longmapsto f[\tilde m_d]_2 \,, $$
    for $\tilde m_d \in \SLZ$ a matrix congruent to  $\smt{d^{-1}}00d$ modulo~$N$.
\end{subDefinition}

Let us now suppose that $N = Mq$ for $M,q$ coprime, $q$ a prime power, and that $H = \tilde H \times B_0(q)$ as in \eqref{eq_K}. Under the canonical isomorphism \eqref{Hp_complex_points_can}, a point $(\tau, \smt abcd) \in M_{H_p}(\C)$  corresponds to the elliptic curve $E_\tau =  \C/\Z{+}\Z\tau$ together with the subgroups $\langle \tfrac \tau p\rangle \subset E_\tau[p]$ and $\langle \tfrac{b\tau+d}{q} \rangle \subset E_\tau[q]$ and the basis $(\tfrac{a\tau+c}M, \tfrac{b\tau+d}M)$ of $E_\tau[M]$. The image of a point $(\tau, a)$ under the $q$-th Atkin-Lehner $w_q$ is the elliptic curve $\C/\Z \tfrac1q {+} \Z\tau$ together with the subgroups $\langle \tfrac\tau p\rangle $ and $\langle \tfrac\tau q\rangle$ and the basis $(\tfrac{a\tau}{M}, \tfrac1M)$ of the $M$-torsion, which, for $\tau' = p\tau$ is isomorphic (under the map $z \to qz)$ to the  the elliptic curve $\C/\Z {+} \Z\tau'$ together with the subgroups $\langle \tfrac{\tau'} p\rangle $ and $\langle \tfrac{\tau'} q\rangle$ and the basis $(\tfrac{a\tau'}{M}, \tfrac qM)$ of the $M$-torsion. This last datum corresponds to a point $(q\tau, m)$ for $m\in \GLmod{qM}$ that is congruent to $\smt a00q$ modulo~$M$ and congruent to $\smt ***0$ modulo~$q$. If we apply the action \eqref{action_GL2} by a matrix
\begin{subeqn}\label{tildem_q}
\tilde m_q \in \Gamma^0(p) \quad \text{such that } \tilde m_q\equiv \smt{q}{0}{0}{q^{-1}} \bmod M\,, \tilde m_q\equiv \smt{0}{-1}{1}{0} \bmod q\,,
\end{subeqn}
the same point is moved to the point $(\tilde m_q \smt q001 \tau, \smt{a(q+M)}001)$. We deduce that 
\begin{subeqn}\label{AL_male}
w_q^* = [\tilde m_q \smt q001]_2 \otimes \sigma_{q+M} \lefttorightarrow %H^0(M_{\{\Id\}\times B^0(p), \C} , \Omega^1) \cong 
\left(S_2\left(\Gamma(N){\cap} \Gamma^0(p)\right) \otimes_\C \C^{(\Z/N\Z)^\times} \right)^H    \,,
\end{subeqn}
where $\sigma_{q+M} \lefttorightarrow \C^{(\Z/N\Z)^\times} $ is the shift $(x_a)\mapsto (x_{a(q+M)})$.

This discussion, together with Proposition \ref{prop:aut}, Theorem \ref{thm:comparison}, and Lemma \ref{lem:hodge} implies the following result. Notice that by Remark \ref{rmk:pullpullpush} the automorphisms act by pushforward, or equivalently by pullback of their inverses, on the  1-forms.

%\underline{Warning:} since Proposition \ref{prop:diag} tells us that $A$ and $A^*$ are conjugated (and in many cases equal), in Theorems \ref{thm:modular_general} and \ref{thm:modular_special} we have stated the result for at the action of $A$ instead that of its adjoint $A^*$; in the next result we use $A^*$, since we do not want to conjugate the automorphisms by unknown change of basis, we look at $A^*$.

\begin{subTheorem}\label{thm:modular_aut}
Let  $G = G(p,\ell,H)$ be the graph in Definition \ref{def:graph}, with $V$ the set of vertices and $\ker(w_{1,*}),\ldots,\ker(w_{n,*}) $  the subspaces of $\C^V$  described in Remark \ref{rem:genratori_nuclei}. 

Then there is an isomorphism  
$$ \bigoplus_{i=1}^n \ker(w_{i,*}) \cong  \left( S_2^\pnew (\Gamma^0(p) {\cap} \Gamma(N)) \otimes_\C \C^{(\Z/N\Z)^\times} \right)^H \,, $$
that simultaneously intertwines the action of the adjoint of the adjacency matrix $A^*$ (see also Proposition \ref{prop:diag}), the matricial automorphisms $\langle g \rangle$ in Definition \ref{def:aut}, the Galois action in Definition \ref{def:Galois_inv} and, if there, the Atkin-Lehner involutions $w_q$ in Definition \ref{def:AL} on the left, with the action of $\tT_\ell \otimes \sigma_\ell$, the action of a matrix $g^{-1}$ in \eqref{action_GL2_2}, the map $w_{p}\otimes \sigma_{1/p}$ (see Definition \ref{def:Fricke_forms}) and, if there, the inverse of the map \eqref{AL_male} on the right (we denote $\sigma_d\lefttorightarrow \C^{(\Z/N\Z)^\times}$ is the shift $(x_a)_a\mapsto (x_{ad})_a$). 
\end{subTheorem}

In some special cases we can be slightly more explicit. %qui
\begin{subTheorem} \label{thm:modular_special_aut}
Keep the notation as in Theorem \ref{thm:modular_special} and let $A^*$ be the adjoint of the adjacency matrix, as in Proposition \ref{prop:diag}.
 
\begin{itemize}
\item if $H = \{\Id\}$, then $\oplus_i \ker(w_{i,*})$, as module over $A^*$, over the Galois action, and over the diamond operators $\langle d \rangle$, is isomorphic to $S'\otimes_\C \C^{(\Z/N\Z)^\times}$, as a module over  $\tT_\ell\otimes \sigma_\ell$, over $w_{p}\otimes \sigma_{1/p}$ and over $\dia d^{-1}\otimes \sigma_{d^{-2}}$. 

\item if $H=B_0(N)=\{\smt*0**\}$ then $n=1$ and $\ker(w_{1,*})$, as a module over $A^*$,  over the Galois action, and over the Atkin-Lehner involutions $w_q$, is isomorphic to $S_2^{p-\new} (\Gamma_0(pN))$ as a module over $\tT_\ell$, over the Fricke involution $w_p$, and over the other Atkin-Lehner involutions $w_q$ in \cite{AL}.

\item if $H=B_1(N)=\{\smt*0*1\}$ then $n=1$ and $\ker(w_{1,*})$, as a module over $A^*$, over the Galois action, and over the diamond operators $\langle d \rangle$, is isomorphic to $S'$, as a module over  $\tT_\ell$, over $w_p$ and over $\dia{d^{-1}}$. 

\item if $H $ is a non-split Cartan, then  $n=1$ and $\ker(w_{1,*})$ as a module over $A^*$,  over the Galois action, and over the nontrivial matricial automorphisms $\langle g_q \rangle$ for $q^e$ a prime power in the factorization of $N$ and $g_q$ the only elements in the normalizer of $H$ such that $ g_q \equiv \Id \pmod{N/q^e}$, is isomorphic to $\bigoplus_{d|N}S_2^{\new}(\Gamma_0(pd^2))$ as a module over $\tT_\ell$-module, over the $p$-th Atkin Lehner involution (see \cite{AL}) and over the $q$-th Atkin-Lehner involution (that acts trivially on $S_2^{\new}(\Gamma_0(pd^2))$ when $q\nmid d$)
\end{itemize}   
\end{subTheorem}

\begin{subRemark}
To have an isomorphism which is \emph{simultaneously} equivariant with respect to all automorphisms, in Theorems \ref{thm:modular_aut} and \ref{thm:modular_special_aut} we used the adjoint $A^*$ of the adjacency matrix. Instead, in Theorems \ref{thm:modular_general} and \ref{thm:modular_special}, for merely aesthetic reasons, we preferred using the adjacency matrix, which is conjugate to its adjoint.
\end{subRemark}

\subsection{Asymptotic distribution of the eigenvalues}\label{s:forms_vs_distrib}

Following Serre \cite{Serre}, given a linear diagonalizable operator $P$ with spectrum $\sigma(P)$ and domain $V$ of finite dimension $r$, we introduce the probability measure 
$$
\mu(P,V):=\frac{1}{r}\sum_{\lambda\in \sigma(P)}\delta_{\lambda} \,,
$$
where $\delta_{\lambda}$ is a Dirac mass at $\lambda$. Let us also recall the Kesten-McKay measure supported on the Hasse interval $[-2\sqrt{\ell},2\sqrt{\ell}]$ from Equation \eqref{eq:mckay}
\[
\mu_{\ell}=\frac{\ell+1}{\pi}\frac{\sqrt{\ell-x^2/4}}{\ell(\ell^{1/2}+\ell^{-1/2})^2-x^2}dx \,.
\]

%? The following corollary is about the Hecke modules described in Theorem \ref{thm:modular_special}, it is needed to study the asymptotic distribution of the spectrum of isogeny graphs, as in Corollary \ref{cor:asy_dist}.
We are interested in $\mu(P,V)$ when $P$ is a Hecke operator and $V$ is one of the spaces appearing in Theorem \ref{thm:modular_special}. The  following theorem gives asymptotics, implying Corollary \ref{cor:asy_dist}.

\begin{subTheorem}\label{thm:trace_asymptotic}

Fix a prime $\ell$, a positive integer $N$ coprime with $\ell$, and let $p_i$ be an increasing sequence of prime numbers coprime with $N\ell$. Then 
$$
\lim_{i\to \infty}\mu \left(T_{\ell}, S_2^{p_i-new}(\Gamma_0(p_iN))\right) =\lim_{i\to \infty}\mu \left(T_{\ell}, \bigoplus_{d|N}S_2^{\new}(\Gamma_0(p_id^2))\right)    
=\mu_{\ell} \,,$$
and, for each character $\chi$ modulo $N$,
$$\lim_{i\to \infty}\mu \left(T_{\ell}, S_2^{p_i\text{-new}}(\Gamma_1(p_iN),\chi)\right) = \lim_{i\to \infty}\mu \left(T_{\ell}, S_2^{p_i\text{-new}}(\Gamma_1(p_iN^2),\chi)\right)=\sqrt{\chi(\ell)}\mu_{\ell}  \,. $$

\end{subTheorem}
Observe that $\mu_{\ell}=-\mu_{\ell}$, so it does not matter which sign of the square root of $\chi(\ell)$ we choose.
\begin{proof}
Let us first prove the theorem for $S_2^{p_i-new}(\Gamma_0(p_iN))$. As Hecke modules we have
$$
S_2(\Gamma_0(p_iN))=S_2^{p_i-\new}(\Gamma_0(p_iN))\oplus S_2(\Gamma_0(N))^{\oplus 2} \,.
$$
Passing to measures, and denoting $d(k) = \dim S_2(\Gamma_0(k))$, $d(p,k) = \dim S_2^\pnew(\Gamma_0(pk))$, we get 
$$
\mu(T_{\ell},S_2(\Gamma_0(p_iN))=
\frac{d(p_i,N)}{d(p_iN)}\mu(T_{\ell},S_2(\Gamma_0(p_iN))^{p_i-new})+2 \frac{d(N)}{d(p_iN)} \mu(T_{\ell}, S_2(\Gamma_0(N))) \,,
%\mu(T_{\ell},S_2(\Gamma_0(p_iN))=\frac{\dim S_2^{p_i-new}(\Gamma_0(p_iN))}{\dim S_2(\Gamma_0(p_iN))}\mu(T_{\ell},S_2(\Gamma_0(p_iN))^{p_i-new})+2\frac{\dim S_2(\Gamma_0(N))} {\dim S_2(\Gamma_0(p_iN))}\mu(T_{\ell}, S_2(\Gamma_0(N)))
$$
the second summand on the right hand side goes to zero when $i$ goes to infinity, hence we deduce the claim from \cite[Theorem 1]{Serre}.

The other cases are proved in the same way, replacing \cite[Theorem 1]{Serre} first with \cite[Theorem 1]{Serre} and then with \cite[Theorem 4]{Serre}.

\end{proof}

\appendix
%\appendixpage	
\section{Correspondences on nodal curves}\label{S:nodal}

In the first part of this Appendix we recall for the reader convenience well-known facts and notations about the Picard group of nodal curves. We then use it to state and prove Proposition \ref{prop:correspondence_T}.

Suppose we are given two smooth projective curves $C_1, C_2$ over a field $k=\ol k$. 
We allow for $C_1$ and $C_2$ to be disconnected with the same number of connected components which we denote $C_1^1, \ldots, C_1^r$ of $C_1$, and $C_2^1, \ldots, C_2^r $. We suppose that for each $j=1,\ldots, r,$ we are given distinct points $x_1^j, \ldots x_{n_j}^j\in C_1^j(k)$ and $y_1^j, \ldots y_{n_j}^j\in C_2^j(k)$, and we look at the nodal curve 

% We allow for $C_1$ and $C_2$ to be disconnected, with connected components $C_1^j$ and $C_2^j$. 
% For a certain $r$, we identify points in the first $r$ components of $C_1$ with points in the first $r$ components of $C_2$: we suppose that for each $j=1,\ldots, r$, with $r \le \min(r_1, r_2)$, we are given distinct points $x_1^j, \ldots x_{n_j}^j\in C_1^j(k)$ and $y_1^j, \ldots y_{n_j}^j\in C_2^j(k)$, with $n_j\ge 1$ and we look at the nodal curve 
\begin{equation}\label{eq:glued_curve}
X = (C_1\sqcup C_2)/x_i^j = y_i^j \,.
\end{equation}
For simplicity we assume that $n_j\ge 1$, so that $X$ has exactly $r$ connected components, namely the curves $X_j = (C_1^j\sqcup C_2^j)/x_i^j = y_i^j$, each having $2$ irreducible components.  
Anyway Proposition \ref{prop:correspondence_T} remains true if we suppose that the number of indices $j$ such that $n_j\ge 1$ (keeping the notation $r$ for this number) is strictly smaller than the number of components of $C_1$ and of $C_2$, possibly different.

%For each $j=1,\ldots, r$ let $J^j = \Pic^0_{X^j/k}$ and 
Let  $ J = \Pic^0_{X/k}$ be the scheme representing invertible sheaves on $X$ having degree $0$ when restricted to each irreducible component of $X$. In particular,the natural maps $C_1 \to X$ and $C_2\to X$ induce by pull back a map
\begin{equation} \label{eq:tech_JCi}
J \lto \Pic^0_{C_1/k} \times \Pic^0_{C_2/k}\,.
\end{equation}
Such a map is surjective: given invertible sheaves $\calL_i$ over $C_i$, we can construct a (non-canonical) lift of $(\calL_1, \calL_2)$ by choosing generators $v_i^j, w_i^j$ of $(x_i^j)^*\calL_1, (y_i^j)^*\calL_2$ and defining the invertible sheaf $\calL = \calL_{\calL_1, \calL_2, (v_i^j, w_i^j)_{i,j}}$ on $X$ associating to each open $U\subset X$, the module 
\begin{equation}\label{eq:glue_L}
\calL(U) = \{(f,g)\in \calL_1(U\cap C_1)\times\calL_2(U\cap C_2):\, f(x_i^j)/v_i^j = g(y_i^j)/w_i^j \text{ for each }i,j\}\,.
\end{equation}
We notice that the structure sheaf is a particular case of the above construction, namely when $\calL_i = \calO_{C_i}$ and $v_i = x_i^*1, w_i = y_i^*1$. 
Moreover, all the lifts of $(\calL_1, \calL_2)$ are obtained with this construction: given a lift $\calM$, we choose for each $i$ a section trivializing $\calM_{x_i}$, which determines by pull back sections $v_i, w_i$; then the pull back of sections to $C_i$ determines a morphism of $\calO$-modules $\calM \to \calL_{\calL_1, \calL_2, (v_i, w_i)_i}$, which is an isomorphism because of how the structure sheaf is defined.

Since map (\ref{eq:tech_JCi}) is surjective, we have an exact sequence of group schemes over $k$
\begin{equation}\label{eq:nodal_torus}
0 \lto T \lto J \lto  \Pic^0_{C_1/k} \times \Pic^0_{C_2/k} \lto 0\,,
\end{equation}
for a certain group scheme $T$.
For every $k$-algebra $A$ we can describe the points on $T$ explicitly using (\ref{eq:glue_L}): for every choice of $i,j$, the line bundle $(y_i^j)_{\Spec A}^*\calO_{C_{2,\Spec A}}$ is canonically trivial, hence  hence its generating sections are canonically elements of $A^\times$; in particular, every line bundle on $X_{\Spec A}$ that is trivial on the $C_i$'s is isomorphic to
\[
\calL_{a} := \calL_{\calO_{C_1}, \calO_{C_2}, (1, a(y_i^j))} \qquad \text{for some function } a\colon Y =  \{y_1^1, \ldots, y_r^{n_r}\} \lto A^\times   \,.
\]
%for some function
%\[
%\qquad = \{\text{generating sections of } %(y_i^j)^*\calO_{C_{2,A}} \}
%\]

More formally, denoting by  $\G_m^Y$ the set of maps from $Y$ to $\G_m$, we have a surjective map
$$
\begin{aligned}
	\Gm^Y & \longrightarrow  T \\
    a & \longmapsto  \calL_a
\end{aligned}
$$
Let us study the kernel. Which of the invertible sheaves $\calL_{a}$ are trivial?  Exactly those where $a(y_i^j)$ does not depend on $i$ but only on $j$: indeed $\calL_a$ is trivial if an only if it is trivial when restricted to each connected component $X^j$ of $X$, and, since $\mathcal L_{a}|_{X^j}$ has degree $0$, then it is trivial  if and only if it has a non trivial global section, which implies our claim using (\ref{eq:glue_L}) and the fact the only global functions on $C_1^j$ and $C_2^j$ are constant. This discussion implies that the following sequence of group schemes over $k$ is exact
\begin{equation}\label{eq:T_general}
\begin{tikzcd}
	0 \arrow[r] &\Gm^r \arrow[r, "\Delta"] &\Gm^Y \arrow[r] & T \arrow[r] &0
	%\\
    %&&\Delta(b_1, \ldots, b_r)(y_i^j)=b_j
 %   & (b_1, \ldots b_r) \arrow[r, mapsto] &   a\colon Y\to \Gm\,,  \,   \tilde a(y_i^j) = b_j 
	%\\ &&   a \arrow[r, mapsto] & \calL_a
\end{tikzcd}
\end{equation}
where $\Delta(b_1, \ldots, b_r)(y_i^j)=b_j$.

The above exact sequence, allows us to describe the characters of $T$. We have canonical isomorphisms $(\Gm^Y)^\vee= \Hom (\Gm^Y, \Gm) = \Z^Y =  \bigoplus_{i,j} \Z y_i^j$ and $(\Gm^r)^\vee= \Hom (\Gm^r, \Gm) = \Z^r$ and the map $\Delta$ induces 
\[
\Sigma = \Delta^\vee\colon  \bigoplus_{i,j} \Z y_i^j \lto \Z^r  \,, \quad  \sum_{i,j} m_i^j  y_i^j \longmapsto \Big( \sum_{i=1}^{n_1}m_i^1, \ldots, \sum_{i=1}^{n_r} m_i^r \Big)\,.
\]
Then, the exact sequence (\ref{eq:T_general}) gives the following isomorphism 
\begin{equation}\label{eq:T_dual_general}
\begin{tikzcd}
	T^\vee = \Hom(T,\Gm) \arrow[r,equal]&  \ker(\Delta^\vee\colon \Gm^{T,\vee} \to \Gm^{r,\vee})  \arrow[r,equal] & \ker (\Sigma) % \colon \Z^Y \to \Z^r \big)
	\\
	\calL_a \mapsto \prod_{i,j} a(y_i^j)^{m_i^j}&&   \sum_{i,j} m_i^j  y_i^j \arrow[ll, mapsto] \,.
\end{tikzcd}
\end{equation}

In the next proposition we describe how certain correspondences act on $T$ and on its characters, which is applied in the proof of Theorem \ref{thm:comparison} to the Hecke operator \ref{def:Hecke}. In the notation of the proposition, we do not keep track of the connected components.

\begin{Proposition}\label{prop:correspondence_T}
Let $k$ be an algebraically closed field and let $C= (C_1 \sqcup C_2)/(x_i = y_i)_{i=1}^n$ and $D= (D_1 \sqcup D_2)/(v_j = w_j)_{i=1}^m$ be curves over $k$ described as in (\ref{eq:glued_curve}), with $C_i, D_i$ smooth. 

Let $F,G\colon D\to C$ be maps restricting to $F_i, G_i\colon D_i\to C_i$ and sending the smooth part of $D$ into the smooth part of $C$ and the nodal points to the nodal points.
Then, for each $a \colon \{y_1,\ldots, y_n\}\to k^\times$ we have 
\begin{equation} \label{eq:correspondece_torus}
	G_* F^* \calL_a \cong \calL_b \qquad \text{ for } b := a\circ {F_2}_*G_2^* \colon \, y_i \mapsto \prod_{G_2(v) = y_i} a(F_2(v))^{\mathrm{ord}_v(G_2)} \,,
\end{equation}
where $G_*$ is a cycle push-forward.

Let $T$ be the maximal torus of $\Pic_{C/k}$, as in (\ref{eq:T_general}), and let $T^\vee$ be its groups of characters. Keeping track of how the points $y_i$ are distributed among the components of $C_2$, we get an isomorphism, analogous to (\ref{eq:T_dual_general}),
$$ 
T^\vee = \ker \left( \Sigma\colon \bigoplus_{i=1}^n \Z y_i \to \Z^r \right) \,. %\,, \quad W^\vee = \ker \Theta \colon \Z^{\{w_1, \ldots, w_m\}} \to \Z^s\,.
$$
Using the above isomorphism, the map $(G_* F^*)^\vee$ is the restriction of the map $H$ below
\begin{equation}\label{eq:correspondence_torus_dual}
	\begin{tikzcd} \displaystyle
		T^\vee \arrow[d, "(G_* F^*)^\vee"] \arrow[rr,hook] &&  \displaystyle  \bigoplus_{i=1}^n  \Z y_i \arrow[d, "H"] & y_i \arrow[d,mapsto]
		\\ T^\vee  \arrow[rr,hook]  &&  \displaystyle  \bigoplus_{i=1}^n \Z y_i &   \displaystyle 	\sum_{G_2(v)=y_i}\mathrm{ord}_v(G_2) F_2(v) \,.
	\end{tikzcd}	
\end{equation}
\end{Proposition}
\begin{proof}

We first give a description of $T$ in terms of Cartier divisors. For a function $a\colon \{y_i\}\to k^\times$, take a meromorphic function $f\in k(C_2)$ such that $f(y_i) = a(y_i)$ for every $i$. By (\ref{eq:glue_L}), the pair $(1,f)$ defines a meromorphic section of $\calL_a$. The divisor attached to this section is supported in $C_2\setminus \{v_1, \dots v_n\}$, and can be identified with the divisor $\divisor f$. As explained for instance in \cite[Section 1, Proposition 1.4 (b)]{Ful}, the push-forward of a cycle attached to a meromorphic function can be computed using the norm, so 
\begin{equation}\label{eq_1}
	G_* F^* \calL_a \cong G_* F^* (\divisor (1,f)) = G_*  \divisor (F^*(1,f) ) = \divisor ((1,\Norm_{G_2}(F_2^*f))) 
	= \mathcal L_{c}  \,,
\end{equation}
for 
\[
c = \Norm_{G_2}(F_2^*f)|_{\{y_i\}}\,.
\]

To prove (\ref{eq:correspondece_torus}), it remains to prove $c=b$. The norm is compatible with pull-backs, i.e. if we want to compute $\Norm_{G_2}(F_2^*f))(y_i)$ we can look at the base change $G_2\colon G_2^{-1}(y_i) \to y_i$, the pull-back of $F_2^*f$ to $G_2^{-1}(y_i)$ and then compute the norm; we conclude that
$$
(\Norm_{G_2}(F_2^*f))(y_i) = \prod_{G_2(v)=y_i}(F_2^*f)(v)^{\mathrm{ord}_v G_2}\,.
$$
Since $G_2$ and $F_2$ send the smooth part of $D_2$ in the smooth part of $C_2$ (and analogously for the inverse images), then all the $v$'s appearing above lie in the set $\{w_j\}$ and consequently the points $F_2(v)$ lie in the set $\{y_j\}$, so
\[
\prod_{G_2(v)=y_i}(F_2^*f)(v)^{\mathrm{ord}_v G_2} = \prod_{G_2(v)=y_i}f(F_2(v))^{\mathrm{ord}_v G_2} = \prod_{G_2(v)=y_i}a(F_2(v))^{\mathrm{ord}_v G_2}\,,
\]

For the second part of the proposition, namely Equation \eqref{eq:correspondence_torus_dual}, it is enough proving that for each $i,j$ we have $(G_*F^*)^\vee(y_i - y_j) = H(y_i-y_j)$, which is true since 
\begin{equation*} %\label{eq:tech4}
	\begin{aligned}
		& (G_*F^*)^\vee (y_i-y_j)(\calL_a)    = (y_i-y_j)(\calL_b) = \tfrac{b(y_i)}{b(y_j)} = 
		\prod_{G_2(v) = y_i} \!\! a(F_2(v))^{ \mathrm{ord}_v(G_2)} \cdot \prod_{G_2(v) = y_j}  \!\!  a(F_2(v))^{-\mathrm{ord}_v(G_2)}
		\\ &
		\quad =a\left(\sum_{G_2(v)=y_i}\mathrm{ord}_v(G_2) F_2(v) - \sum_{G_2(v)=y_j}\mathrm{ord}_v(G_2) F_2(v)\right) = a\left( H(y_i-y_j) \right) = H(y_i-y_j) (\calL_a)\,.
	\end{aligned}
\end{equation*}\end{proof}

\section{Numerical experiments on the largest non-trivial eigenvalue}\label{appB}

In this appendix we focus on isogny graphs with trivial level structure, so we omit $H$ from the notations. We study the value $\eta=\eta(p,\ell)$ for the graph $G(p,\ell)$. Recall from Equation \eqref{eq:gap} that
$$
\eta(p,\ell):=\min\{2\sqrt{\ell}-|\lambda|\}\,,
$$
where the minimum ranges among all non-trivial eigenvalues of the adjacency matrix of the graph $G(p,\ell)$. In other words, $\eta(p,\ell)$ is the biggest number such that all non-trivial eigenvalues of the adjacency matrix of $G(p,\ell)$ lie in the shrunk Hasse interval $[-2\sqrt{\ell}+\eta(p,\ell),2\sqrt{\ell}-\eta(p,\ell)]$.

For a fixed $\ell$, the number of vertexes of $G(p,\ell)$ is linear in $p$, and Alon-Boppana inequality implies that $\eta\leq C \log(p)^{-2}$, for some constant $C$ which depends only on $\ell$. Recall that lower bounds on $\eta$ give bounds on the mixing time of the graph via Proposition \ref{mixing_times}. To best of our knowledge, sharp lower bounds and the asymptotic of $\eta$ are not known, see Question \ref{q:gap}.
%\Guido{forse scriverei proprio. We are interested in good lower bounds for $\eta$, see Question \ref{q:gap}. Such lower bounds would also give bounds on the mixing time of the graph via Proposition \ref{mixing_times}. To the best...}

\medskip

We have computed some values of $\eta(p,\ell)$ using the database \cite{isodb}, which lists graphs with $\ell=2,3,5,7,11$ and $p<30.000$. The results are displayed in Figure \ref{figuraeta}.

For fixed $\ell$,  the order of magnitude of $\eta$ varies a lot, so it is more convenient for graphical reasons to plot $\log(1/\eta)$ (we plot these values with light blue dots). Under the transformation $x \mapsto \log(1/x)$, our bound in Theorem \ref{thm:spet1} on $\eta$ is linear in $p$, reaching the thousands; being quite far from the data below, we do not even plot it. 

Alon-Boppana bound becomes $2\log(\log(p))$ up to an additive constant. We plot the function $2\log(\log(p))$ with a red line, to compare its shape with the data on $\eta$.

\begin{figure}[H]
\centering
    \includegraphics[width=17.5cm]{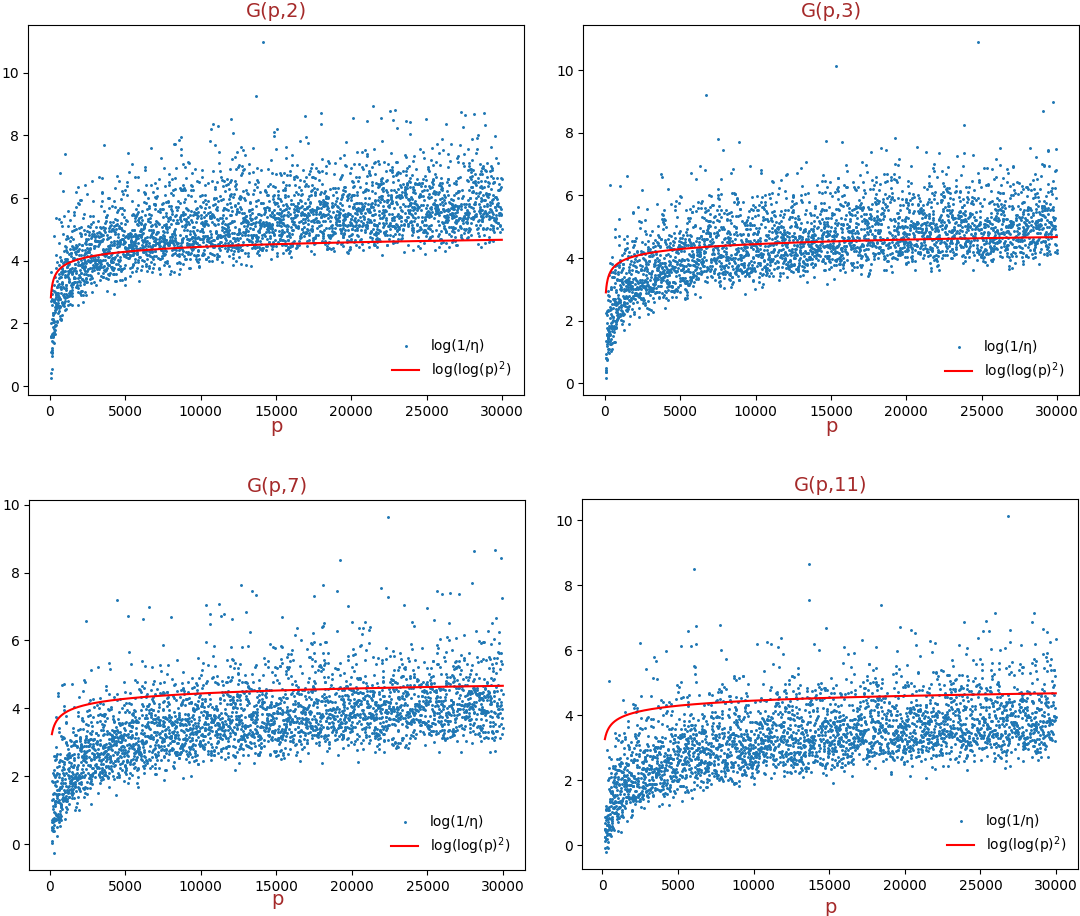}
    \caption{Numerical experiments on $\eta(p,\ell)$} \label{figuraeta}
\end{figure}

\end{document}